\documentclass[11pt]{amsart}
\usepackage{bbm}
\usepackage{amsmath,amssymb}
\makeatletter
\newcommand*{\mint}[1]{%
  \mint@l{#1}{}%
}
\newcommand*{\mint@l}[2]{%
  \@ifnextchar\limits{%
    \mint@l{#1}%
  }{%
    \@ifnextchar\nolimits{%
      \mint@l{#1}%
    }{%
      \@ifnextchar\displaylimits{%
        \mint@l{#1}%
      }{%
        \mint@s{#2}{#1}%
      }%
    }%
  }%
}
\newcommand*{\mint@s}[2]{%
  \@ifnextchar_{%
    \mint@sub{#1}{#2}%
  }{%
    \@ifnextchar^{%
      \mint@sup{#1}{#2}%
    }{%
      \mint@{#1}{#2}{}{}%
    }%
  }%
}
\def\mint@sub#1#2_#3{%
  \@ifnextchar^{%
    \mint@sub@sup{#1}{#2}{#3}%
  }{%
    \mint@{#1}{#2}{#3}{}%
  }%
}
\def\mint@sup#1#2^#3{%
  \@ifnextchar_{%
    \mint@sup@sub{#1}{#2}{#3}%
  }{%
    \mint@{#1}{#2}{}{#3}%
  }%
}
\def\mint@sub@sup#1#2#3^#4{%
  \mint@{#1}{#2}{#3}{#4}%
}
\def\mint@sup@sub#1#2#3_#4{%
  \mint@{#1}{#2}{#4}{#3}%
}
\newcommand*{\mint@}[4]{%
  \mathop{}%
  \mkern-\thinmuskip
  \mathchoice{%
    \mint@@{#1}{#2}{#3}{#4}%
        \displaystyle\textstyle\scriptstyle
  }{%
    \mint@@{#1}{#2}{#3}{#4}%
        \textstyle\scriptstyle\scriptstyle
  }{%
    \mint@@{#1}{#2}{#3}{#4}%
        \scriptstyle\scriptscriptstyle\scriptscriptstyle
  }{%
    \mint@@{#1}{#2}{#3}{#4}%
        \scriptscriptstyle\scriptscriptstyle\scriptscriptstyle
  }%
  \mkern-\thinmuskip
  \int#1%
  \ifx\\#3\\\else_{#3}\fi
  \ifx\\#4\\\else^{#4}\fi
}
\newcommand*{\mint@@}[7]{%
  \begingroup
    \sbox0{$#5\int\m@th$}%
    \sbox2{$#5\int_{}\m@th$}%
    \dimen2=\wd0 %
    \let\mint@limits=#1\relax
    \ifx\mint@limits\relax
      \sbox4{$#5\int_{\kern1sp}^{\kern1sp}\m@th$}%
      \ifdim\wd4>\wd2 %
        \let\mint@limits=\nolimits
      \else
        \let\mint@limits=\limits
      \fi
    \fi
    \ifx\mint@limits\displaylimits
      \ifx#5\displaystyle
        \let\mint@limits=\limits
      \fi
    \fi
    \ifx\mint@limits\limits
      \sbox0{$#7#3\m@th$}%
      \sbox2{$#7#4\m@th$}%
      \ifdim\wd0>\dimen2 %
        \dimen2=\wd0 %
      \fi
      \ifdim\wd2>\dimen2 %
        \dimen2=\wd2 %
      \fi
    \fi
    \rlap{%
      $#5%
        \vcenter{%
          \hbox to\dimen2{%
            \hss
            $#6{#2}\m@th$%
            \hss
          }%
        }%
      $%
    }%
  \endgroup
}
\usepackage{mathrsfs}
\usepackage{amssymb}
\usepackage{amsmath}
\usepackage{amsthm}
\usepackage{amsfonts}
\usepackage{color}
\usepackage{graphicx}

\usepackage[cp1252]{inputenc}

\usepackage{mathrsfs}
\usepackage{graphicx}

\usepackage[active]{srcltx}

\allowdisplaybreaks

\def\rr{{\mathbb R}}
\def\rn{{{\rr}^n}}

\def\zz{{\mathbb Z}}
\def\nn{{\mathbb N}}

\def\fz{\infty}
\def\az{\alpha}

\def\dist{{\mathop\mathrm{\,dist\,}}}
\def\loc{{\mathop\mathrm{\,loc\,}}}

\def\dz{\delta}

\def\ez{\epsilon}

\def\kz{\kappa}
\def\bz{\beta}

\def\gz{{\gamma}}

\def\vz{\varphi}

\def\sz{\sigma}

\def\wz{\widetilde}

\def\bint{{\ifinner\rlap{\bf\kern.35em--}
\int\else\rlap{\bf\kern.45em--}\int\fi}\ignorespaces}

\def\bbint{{\ifinner\rlap{\bf\kern.35em--}
\hspace{0.078cm}\int\else\rlap{\bf\kern.45em--}\int\fi}\ignorespaces}

\def\esup{\mathop\mathrm{\,esssup\,}}

\def\usc{\mathop\mathrm{\,USC\,}}

\newtheorem{thm}{Theorem}[section]
\newtheorem{lem}[thm]{Lemma}
\newtheorem{rem}[thm]{Remark}
\newtheorem{cor}[thm]{Corollary}
\numberwithin{equation}{section}

\begin{document}

 \title[A quantative Sobolev regularity for
absolute minimizers ]{A quantative Sobolev regularity for
absolute minimizers involving Hamiltonian $H(p)\in C^0 (\mathbb{R}^2)$  in plane}

\author{ Peng Fa, Qianyun Miao and Yuan Zhou}

          \address{ Department of Mathematics, Beihang University, Beijing 100191, P.R. China}
                    \email{SY1609131@buaa.edu.cn }

                      \address{ School of Mathematical Sciences, Peking University, Beijing 100871, P. R. China}
                    \email{qianyunm@math.pku.edu.cn}

\address{ Department of Mathematics, Beihang University, Beijing 100191, P.R. China}
                    \email{yuanzhou@buaa.edu.cn}


\date{\today}
\arraycolsep=1pt
\allowdisplaybreaks
 \maketitle

\begin{center}
\begin{minipage}{13.5cm}\small
 \noindent{\bf Abstract.}\quad
 Suppose that    $H  \in C^0 (\mathbb{R}^2)$ satisfies
 \begin{enumerate}
  \item[(H1)]  $H$ is locally strongly convex and locally strongly concave in $\rr^2$,
    \item[(H2)]  $H(0)=\min_{p\in\rr^2}H(p)=0$.
   \end{enumerate}
  Let $\Omega\subset \rr^2$ be any domain. For
    any $u$ absolute minimizer for $H$ in $\Omega$, or  if $H\in C^1(\rr^2)$ additionally,
      for any viscosity solution to the Aronsson equation $$\mathscr A_H[u]=\sum_{i,j=1}^2 H_{p_i}(Du) H_{p_j}(Du)u_{x_ix_j}=0 \quad \mbox{  in $\Omega$,}$$
  the following are proven in this paper:
   \begin{enumerate}
 \item[(i)]  We have
    $[H(Du)]^\alpha\in W^{1,2}_\loc(\Omega)$ whenever   $\alpha>1/2-\tau_H(0)$;
    some quantative upper bounds  are also given.  Here $\tau_H(0)=1/2$ when  $H\in C^2(\rr^2)$,
    and $0< \tau_H(0)\le 1/2$ in general.

\item[(ii)] If $H\in C^1(\rr^2)$,   then  the distributional determinant
 $-{\rm det}D^2u\,dx$ is a nonnegative Radon
measure in $\Omega$ and enjoys some quantative lower/upper bounds.

\item[(iii)] If $H\in C^1(\rr^2)$, then for all $\alpha>\frac12-\tau_H(0)$, we have
  $$\mbox{$\langle D [H(Du )]^\alpha ,D_p H(Du  )\rangle=0  $ almost everywhere in $\Omega$}.
$$
  \end{enumerate}

 The idea of their proofs  is as follows.
  \begin{enumerate}
 \item[$\bullet$]
 When $H\in C^\fz(\rr^2)$   we observe a fundamental structural identify
 for  the Aronsson operator $\mathscr A_H$,
 and a divergence formula of determinant  matching with  $\mathscr A_H$  perfectly.
 The two identity allow  us to approximate absolute minimizers $u$ for $H$
 via  $e^{\frac1\ez H}$-harmonic functions $u^\ez$ in suitable sense,
 and also establish analogue properties of (i)\&(ii) for $u^\ez$ uniformly in $\ez>0$, from which we conclude (i)-(iii) for $u$.
  \item[$\bullet$]When $H\in C^0(\rr^2)$ (or $C^1(\rr^2)$), we approximate $H$ via smooth $H^\dz$ satisfying (H1)\&(H2) uniformly in $\dz\in(0,1]$, and then approximate absolute minimizers $u$ for $H$
 via   absolute minimizers $ u^\dz$ for $H^\dz$ in suitable sense, which allows us to conclude
 (i)-(iii) for $u$ from those for  $  u^\dz$.
 \end{enumerate}

%
%
\end{minipage}
\end{center}

 \tableofcontents

\section{Introduction}

Let    $H\in C^0(\rn)$ be a   Hamilton function which  is convex and coercive ($\liminf_{p\to\fz}H(p)=\fz$).
Aronsson   1960's initiated   the study of  minimization problems for   $L^\infty$-functional  $${\mathcal F}_{H}(u,\Omega)=\esup_{x\in\Omega}H(Du (x) ), \quad \Omega\subset\rn, u\in W^{1,\fz}_\loc(\Omega);$$
 see \cite{a1,a2,a3,a4}.
It  turns out that the absolute minimizer introduced by Aronsson  is
 the correct notion of minimizers for such  $L^\infty$-functionals.
A function $u\in W^{1,\infty}_{\loc}(\Omega)$ is an absolute minimizer  for $H$ in $\Omega$ (write $u\in AM_H(\Omega)$ for short) if
$${\mathcal F}_{H}(u,V)\le {\mathcal F}_{H}(v,V)\quad \mbox{whenever $V\Subset\Omega$, $v\in W^{1,\infty}_{\loc}(V)\cap C(\overline V)$ and $u=v$ on $\partial V$}.$$

If $H\in C^1(\rn)$ is convex and coercive,   Aronsson derived the   Euler-Lagrange equations  for  absolute minimizers:
 \begin{equation}\label{eq1.2}
\mathscr A_H[u]:=\langle D[H(Du)], D_p H (Du)\rangle= \sum_{i,j=1}^n
H_{p_i} (Du) H_{p_j} (Du) u_{x_ix_j}     =0\quad\mbox{\rm in}\;\Omega,
\end{equation}
which are highly degenerate nonlinear elliptic equations.
The equations \eqref{eq1.2} are called Aronsson equations in the literature;
in the special case $ \frac12|p|^2$,     \eqref{eq1.2} is  the $\infty$-Laplace equation
 \begin{equation}\label{eq1.3}\Delta_\infty u:=\frac12\langle D|Du|^2,Du\rangle =\sum_{i,j=1}^n
u_{x_i}   u_{x_j}   u_{x_ix_j} =0\quad{\rm in }\ \Omega.\end{equation}
 By Crandall-Lions' theory \cite{cil},
 viscosity solutions to \eqref{eq1.2} and \eqref{eq1.3} are defined;
viscosity solutions to \eqref{eq1.3} are called as $\infty$-harmonic functions.
In the seminar paper \cite{j93}, Jensen    identified  $\infty$-harmonic functions
  with absolute minimizers for $\frac12|p|^2$.
In general,
by Crandall et al \cite{cwy}   and  Yu  \cite{y06}
(see also \cite{bjw,c03,gwy06,bej,acjs}) we know that
 absolute minimizers for $H$  coincide  with viscosity solutions to  Aronsson's equation \eqref{eq1.2}.

The  existence and uniqueness  of
  absolute minimizers for  $H$  (or viscosity solutions to \eqref{eq1.2} when $H\in C^1(\rn)$)  have been well-studied in the literature.
Given any   bounded domain
and continuous boundary,
Jensen   \cite{j93}   obtained the  existence and uniqueness of  $\infty$-harmonic functions; see also \cite{bb,cgw,as,pssw}.
For general convex /coercive $H\in C^0(\rn)$,
we refer to \cite{bjw,acj} for  the existence
 of absolute minimizers. 
Assuming additionally that $H^{-1}(\min H )$  has empty interior,
Armstrong et al  \cite{acjs}  obtained their uniqueness;
see also Jensen et al \cite{jwy} when  $H\in C^2(\mathbb R^n)$, and \cite{acj,cgw} when
$H $ is a Banach norm.

The regularity of  absolute minimizers then becomes the main issue in this field.
Note that, by the definition, absolute minimizers for $H$   are always locally Lipschitz,
and hence differentiable almost everywhere.
The study of their possible  regularity beyond these
attracts a lot of attention in the literature for its theoretic difficulty and also potential applications in other fields.

In the special case $H(p)=\frac12|p|^2$,
 as indicated by Aronsson's infinity harmonic function $x_1^{4/3}-x_2^{4/3}$ in whole $\rn$ (see \cite{a84}),
the  best possible regularity of infinity harmonic functions  is $C^{1,1/3}\cap W^{2,t}_\loc$  with $1\le t<3/2$,
which is expected and also conjectured to be true.
Towards this conjecture, the following important progresses were made.
Crandall-Evans \cite{ce} obtain their linear approximation property.
For  planar $\infty$-harmonic functions $u$,
their interior $C^1$-regularity  was proved by Savin \cite{s05}, the interior $C^{1,\alpha}$-regularity  by Evans-Savin  \cite{es08} and the  boundary $C^1$-regularity by Wang-Yu \cite{wy12}; moreover,
Koch et al \cite{kzz} proved that $|Du|^\alpha\in W^{1,2}_\loc$ for   all $\alpha>0$ with  quantative upper bounds,
which is sharp as $\alpha\to0$,
and also proved that the distributional determinant $-\det D^2u\,dx$ is a nonnegative Radon measure enjoying some lower and upper bounds.
When $n\ge 3$, Evans-Smart \cite{es11a,es11b} obtained the everywhere differentiability of
  $\infty$-harmonic functions.

 Assuming that $H \in C^2(\rn)$ is locally strongly convex (see the end of Section 1.1), 
Wang-Yu \cite{wy}
obtained their linear approximation property,  and  when $n=2$, the interior $C^1$-regularity.
Moreover,
when $H \in C^0(\rn)$ is  strictly convex, 
Yu \cite{y07} (even where some stronger assumptions for $H$ are stated) essentially proved that
absolute minimizers enjoy
 the  linear approximation property.
On the other hand,  under $H\in C^1(\rn)$  Katzourakis \cite{k11} showed that,  to get $  C^1$-regularity of all absolute minimizers, it is necessary to assume that $H$ is not a constant in any line segment.

For general $H\in C^0(\rn)$, we  plan to understand
the regularity of absolute minimizers
in a series of papers.
Precisely, in \cite{fwz},  if $H\in C^0(\rn)$ is convex and coercive,
  a regularity criteria is established:   $H$ is not a constant in any line segment
if and only if   absolute minimizers for $H$ have linear approximation property at each point,
moreover  when $n=2$, if and only if   absolute minimizers for $H$ has interior $C^1$-regularity.
In  \cite{fmz},  if $n\ge3$ and
$H\in C^0(\rn)$  is   locally strongly convex/concave, we show that absolute minimizers for $H$ are differentiable everywhere.
In the current paper, if $n=2$ and
$H\in C^0(\rr^2)$  is   locally strongly convex/concave,
we  establish a Sobolev regularity
(involving second order derivatives) for absolute minimizers;   if
$H\in C^1(\rr^2)$ additionally, we prove their determinant are nonpositive Radon measure; see Section 1.1 for details. Moreover, in the forthcoming paper we are going to consider their possible interior $C^{1,\alpha}$-regularity.

 For reader's convenience, we recall that $H\in C^0(\rr^2)$ is locally strongly convex if for any convex set $ U\subset \rr^2$, there exists $\lambda>0$ such that $ H(p)-\frac12\lambda|p|^2 $ convex in $U$.
Similarly,  $H\in C^0(\rr^2)$ is locally strongly concave if for any convex set $ U\subset \rr^2$,
there exists $\Lambda>0$ such that $  \frac12\Lambda|p|^2-H(p) $ is convex in $U$.
Note that  $H(p)\in C^{1,1}(\rn)$ implies that $H$ is locally strongly concave.

 \subsection{Main results}
First we have the following quantative
Sobolev regularity of absolute minimizers.

\begin{thm}\label{THM1.1}
Suppose that $H\in C^0(\rr^2)$ satisfies  the assumptions:
\begin{enumerate}
\item[(H1)] $H$ is   locally strongly convex/concave;
\item[(H2)]  $H(0)=\min_{p\in \rr^2}  H(p)=0$.
\end{enumerate}
Let $\Omega\subset\mathbb R^2$ be any domain.   For any $u\in AM_H(\Omega)$
or,  if $H\in C^1(\rr^2)$ additionally,  for any viscosity solution $u\in C^0(\Omega)$ to  the equation \eqref{eq1.2}, we have   \begin{align*}&\mbox{$[H(Du)]^{\alpha}\in W^{1,2}_{\loc}(\Omega)$ whenever  $\alpha\ge 1/2$, and }\\
 &\mbox{$[H(Du)]^{\alpha}\in W^{1,2}_{\loc}( U)$ whenever  $U\Subset \Omega$ and  $ 1/2-\tau_H( \|H(Du)\|_{L^\fz( U)}) <\alpha<1/2$;}
 \end{align*}
  moreover,  for any  $U\Subset\Omega$ and any $ \alpha>1/2-\tau_H( \|H(Du) \|_{L^\fz(  U)})$, we have
\begin{align}\label{e1.x1}
\int_{V}  |  D[H(Du)]^{\alpha }|^2 \,dx\le&  \frac{C \alpha^2(\alpha+1)}{[\alpha+\tau_H( \|H(Du) \|_{L^\fz ( U)}) -\frac12 ]^2}\left[\frac{   \Lambda_H(  \|H(Du)\|_{L^\fz(U)} ) }{ \lambda_H( \|H(Du)\|_{L^\fz(U)})}\right]^2\nonumber\\
&\times
 \frac{1}{[\dist(V,\partial U)]^2}\int _{U}[H(Du)]^{2\alpha}\,dx
\end{align}
The constant $C$ here is absolute.
\end{thm}

We refer to Section 1.2 for definitions of the  auxiliary functions  $\lambda_H$,  $\Lambda_H$ and $\tau_H$ used above.
Note that $\tau_H:[0,\fz)\to(0,1/2]$ is decreasing and right continuous,  and when $H\in C^2(\rr^2)$, $\tau_H(0)=1/2$.

Next, as a consequence of Theorems \ref{THM1.1} and the $C^{1}$-regularity of absolute minimizers  in Theorem \ref{LEM7.10x},
we   have the following result.

\begin{cor}\label{COR1.3}   Suppose  that $H\in C^0 (\mathbb R^2)$
satisfies  (H1){\rm\&}(H2).
Let $\Omega\subset\mathbb R^2$ be any domain. For any  $u\in AM_H(\Omega)$, or  if $H\in C^1(\rr^2)$ additionally,
 for any viscosity solution $u\in C^0(\Omega)$ to \eqref{eq1.2},
we have  $[H(Du)]^\alpha\in W^{1,2}_\loc(\Omega)$ whenever $\alpha>1/2-\tau_H(0)$.

In particular,   if $H\in C^2 (\mathbb R^2)$, then
$[H(Du)]^\alpha\in W^{1,2}_\loc(\Omega)$ whenever $\alpha>0$;
  if $H\in C^2 (\mathbb R^2)$ and $H^{1/2}$ is convex, then  \eqref{e1.x1} holds  with  $\tau_H\equiv1/2$ whenever $\alpha>0$.
\end{cor}

 Moreover, if $H\in C^1(\rr^2)$ satisfies  (H1){\rm\&}(H2),
 we obtain the following properties of the distributional determinant and the Aronsson equation.
Recall that  for any $v\in C^\fz(U)$, its distributional determinant is given by
$$-\int_U\det D^2v\phi\,dx= \frac12\int_U[- v_{x_i}v_{x_j}\phi_{x_ix_j}+|Dv|^2\phi_{x_ix_i}]\,dx\quad\forall \phi\in C^2_c(U).$$
\begin{thm}\label{THM1.2}
Suppose that $H\in C^1(\rr^2)$ satisfies (H1){\rm\&}(H2).
Let $\Omega\subset\mathbb R^2$ be any domain.
 For any  $u\in AM_H(\Omega)$, equivalently,
  any viscosity solution $u\in C^0(\Omega)$ to \eqref{eq1.2},
   we have the following:
  \begin{enumerate}
  \item[(i)] The distributional determinant $-{\rm det}D^2u\,dx$ is a nonnegative Radon measure satisfying that
  \begin{align*} -\det D^2u \,dx   \ge
  4 \frac{\tau_H( \|H(Du)\|_{L^\fz(U)} )}
 {\Lambda_H( \|H(Du)\|_{L^\fz(U)} )} |D[H(Du)]^{1/2}|^2 \,dx
 \end{align*}
and
 $$\int_{V} -{\rm det}D^2u\,dx \le C\frac1{[\dist(V,\partial U)]^2} \int_U|Du|^2\,dx \quad \mbox{$\forall$  $V\Subset U\Subset\Omega$}.$$
 The constant $C$ here is absolute.

 \item[(ii)]If  $U=\Omega$ and $\alpha\ge 1/2 -\tau_H(0)$, then
 \begin{equation}\label{eq1.9}\langle D[H(Du)]^{\alpha}, D_p H (Du)\rangle=0\quad{\rm   \ almost\ everywere\ in}\   \Omega.
 \end{equation}
  \end{enumerate}
\end{thm}


\subsection{Auxiliary functions $\lambda_H$,  $\Lambda_H$,  $\wz \tau_H$ and
 $\tau_H$  and some remarks}

To characterize the convexity/concavity of $H$  quantatively,  we introduce  two interesting auxiliary functions  $\lambda_H$ and  $\Lambda_H$.  Moreover, we   introduce  auxiliary functions $\tau_H$ used above  and     $\wz \tau_H$ used later, which arise naturally from
 the core identities \eqref{fund1}, \eqref{fund3}\&\eqref{fund5x}
 (see also Lemmas \ref{LEM2.1}\&\ref{LEM2.2} and Theorem \ref{LEM2.3}) below.
For their properties needed in this paper  we refer to Appendix  A.

For $H\in C^0(\rr^2)$  satisfying (H1)\&(H2), we set
 $$\lambda_H(R):=\sup_{\ez>0}\inf\left\{\lambda>0:   H(p )-\frac\lambda2  |p |^2 \mbox{  is  convex  in $H^{-1}([0,R+\ez))$}\right\}\quad \forall R\ge0,$$
  and
  $$\Lambda_H(R):=\inf_{\ez>0}\sup\left\{\Lambda>0:  \frac \Lambda2  |p |^2-H(p ) \mbox{ is  convex  in $H^{-1}([0,R+\ez))$}\right\}\quad \forall R\ge0.$$
By definitions,  $\lambda_H$ is a decreasing right-continuous function in $[0,\fz)$,
   and $\Lambda_H$ is an  increasing and right-continuous function
in $[0,\fz)$.
We always have   $\mbox{$0<\lambda_H\le \Lambda_H<\fz$ in $[0,\fz)$}.$

For $H\in C^2(\rr^2)$ satisfying (H1)\&(H2), we set
\begin{equation} \label{eq1.6}
\wz \tau_H(0):=\frac12  \quad{\rm and}\quad \wz \tau_H (p):= \frac{ H (p)}{\langle [D^2_{pp} H (p)]^{-1} D_p H (p),D_p H (p)\rangle}\quad  \forall p\in \rr^2\setminus\{0\}.\end{equation}
By Lemmas \ref{LEM7.2}, one always has $\wz \tau_H\in C^0(\rr^2)$. We also set
 \begin{equation} \label{eq1.7}\tau_{H }(R) =\inf_{ H(p)\le R}\wz \tau_H (p) \quad\forall R\ge 0.
 \end{equation}
By  Lemma \ref{LEM7.5}, we know    $\tau_H(0)=1/2$ and  $\tau_H\in C^0([0,\fz))$ is decreasing.

For $H\in C^0(\rr^2)$  satisfying (H1)\&(H2), we  define
\begin{equation}\label{eq1.8}\tau_{H }(R): =\sup_{\ez>0}
 \limsup_{\dz\to0}   \tau_{H^\dz} (R+\ez)=\sup_{\ez>0}
 \limsup_{\dz\to0} \inf_{ H^\dz(p)\le R+\ez}\wz \tau_{H^\dz} (p)
  \quad\forall R\ge 0,\end{equation}
where $\{H^\dz\}_{\dz\in(0,1]}$ is a standard smooth approximation of $H$ as given by \eqref{eq7.1} in  Appendix A.
 Since $H^\dz\in C^\fz(\rr^2)$ satisfies (H1)\&(H2) (see Lemma \ref{LEM7.3}),
   both of  $\tau_{H^\dz}$ and $\wz \tau_{H^\dz}$ appeared in \eqref{eq1.8} are given by \eqref{eq1.6}.
 By Remark \ref{REM7.4},  the definitions of
  $\tau_H$  given in \eqref{eq1.7} and \eqref {eq1.8} coincide whenever    $H\in C^2(\rr^2)$.
 By Lemma \ref{LEM7.5}, $\tau_H$ is a decreasing and right-continuous function in $[0,\fz)$, and moreover,
  $$   \frac12\left[\frac {\lambda_H  }  {  \Lambda_H  }\right]^2 \le \tau_H \le \frac12\quad \mbox{ in $[0,\fz)$}.$$

Finally, we give three remarks.

\begin{rem}\rm
In the special case $H(p)=\frac12|p|^2$, 
Theorems \ref{THM1.1}\&\ref{THM1.2} and   Corollary \ref{COR1.3} are exactly the same as
\cite[Theorems 1.1\&1.2]{kzz}. Moreover, Koch et al \cite{kzz} observe that
$|Du|^\alpha\in W^{1,2}_\loc(\Omega)$ is sharp as $\alpha\to0$
since $\log |Dw|\notin W^{1,2}_\loc(\rr^2)$ for the $\fz$-harmonic function $w(x)=x_1^{4/3}-x_2^{4/3}$ in   $\rr^2$.
This also shows that, for general $H$, we can not expect that  $\log [H(Du)+c] \in W^{1,2}_\loc(\Omega)$ where $c>0$ is any constant.
In this sense, $[H(Du)]^\alpha\in W^{1,2}_\loc(\rr^2)$ for $\alpha>0$ given in Corollary \ref{COR1.3} is asymptotic sharp when $\alpha\to0$.
\end{rem}

\begin{rem}\rm\label{REM1.5}  By some necessary modifications of the statements,
we may remove the assumption (H2) from the assumptions in Theorems \ref{THM1.1}\&\ref{THM1.2} and also Corollary \ref{COR1.3}.
Indeed, suppose that $\wz H\in C^k(\rr^2)$ with $k\ge 0$ satisfies (H1).  Thanks to
the locally strongly convexity,   $\wz H$ reaches its minimal $c_{\wz H}$
at a unique point $p^{\wz H}\in\rr^2$.
Write  $  H(p)=\wz H(p+p^{\wz H})-c_{\wz H}\quad\forall p\in\rr^2.$
 Then $ H\in C^k(\rr^2)$ satisfies  (H1)\&(H2).  Set
 $\lambda_{\wz H} =\lambda_H   $, $\Lambda_{\wz H} =\Lambda_H $ and  $\tau_{\wz H}=\tau_H  $,
  where   $ \lambda_H   $, $ \Lambda_H $ and  $ \tau_H  $ are defined in Section 1.2.

  Note that $\wz u\in AM_{\wz H}(\Omega)$ if and only if
  $ u=\wz u-\langle p^{\wz H}, \cdot\rangle \in AM_{  H}(\Omega)$; moreover,
  $ \wz H(D\wz u)-c_{\wz H} =   H(Du) $ and $-\det D^2u=-\det D^2\wz u$.
 In the case $k=0$,     Theorem \ref{THM1.1} holds for $ H$ and $u$ if and only if
it holds for $\wz H$ and $\wz u$ with  all  $H(D  u )$ replaced by $[\wz H(D\wz u)-c_{\wz H}]$.
 In the case $k=1$,    Theorem \ref{THM1.2} holds for $ H$ and $u$ if and only if
   it  holds for $\wz H$ and $\wz u$ with all    $H(D  u )$ replaced by $[\wz H(D\wz u)-c_{\wz H}]$ and  $|Du|^2$ by $|D\wz u-p^{\wz H}|^2$.
 In the case $k=2$,   Corollary \ref{COR1.3} holds for $ H$ and $u$ if and only if
 it holds for $\wz H$ and $\wz u$ with  all   $H(D  u )$ replaced by $[\wz H(D\wz u)-c_{\wz H}]$,
 and $H^{1/2}$ by $[\wz H-c_{\wz H}]^{1/2}$.
\end{rem}

 \begin{rem}\label{REM1.8}\rm
 We only need to prove Theorem  \ref{THM1.1}    when
   $H\in C^0(\rr^2)$ satisfies (H2) and
 \begin{enumerate}
\item[(H1') ]$H$ is strongly convex  and strongly  concave, that is, $0<\lambda_H(\fz) \le \Lambda_\fz(\fz) <\fz,$
    where $\lambda_H(\fz):=\lim_{R\to\fz}\lambda_H(R)$ and $\Lambda_\fz(\fz):=\lim_{R\to\fz}\Lambda_H(R)$.
\end{enumerate}
Indeed, suppose that $H\in C^0(\rr^2)$ satisfies (H1)\&(H2).
Let $\Omega\subset\rr^2$ be any domain and $u\in AM_H(\Omega)$.
To obtain Theorem \ref{THM1.1}, we only need to prove that $[H(Du)]^\alpha\in W^{1,2}_\loc(U)$ and \eqref{e1.x1} hold for any $U\Subset\Omega$ and $\alpha>1/2-\tau_{H}(\|H(Du)\|_{L^\fz(U)})$.
To this end, fix arbitrary $U\Subset   \Omega$.
Let $U\Subset \wz U\Subset \Omega$ and
 note
 $\|H(Du)\|_{L^\fz(\wz U)}<\fz$.
Letting $R:= \|H(Du)\|_{L^\fz(\wz U)}+1 $, by Lemma \ref{LEM7.8}
there exists $\wz H\in C^0(\rr^2)$   satisfying    (H1')\&(H2) such that $\wz H=H$ in $H^{-1}([0, R+1])$,
 $\tau_{H}=\tau_{\wz H} $, $\lambda_{H}=\lambda_{\wz H}$ and
 $\Lambda_{H}=\Lambda_{\wz H}$ in $[0, R+1)$.
Since $u\in AM_{ \wz H} (\wz U)$,
we only need to prove  that $[\wz H(Du)]^\alpha\in W^{1,2}_\loc(U)$ and \eqref{e1.x1} holds with $\wz H$ replaced by $H$ whenever  $\alpha>1/2-\tau_{\wz H}(\|\wz H(Du)\|_{L^\fz(U)})$.

By  Lemma \ref{LEM7.8} and similar reason as above, we only need to prove  Theorem  \ref{THM1.2}  when
   $H\in C^1(\rn)$ satisfies (H1')\&(H2), and Corollary \ref{COR1.3}
 when
   $H\in C^0(\rn)$ satisfies (H1')\&(H2).
 \end{rem}

\subsection{Organization  and ideas of  proofs }

This paper is organized as follows.
In Appendix A  we give several properties of auxiliary functions which we need,
 and in Appendix B we collect some properties of absolute minimizers used in this paper.
Sections 2-4 are devoted to the proofs of  Theorems \ref{THM1.1}\&\ref{THM1.2} when $H\in C^\fz(\rr^2)$
 satisfies  (H1')\&(H2), but some tedious proofs/calculations in Sections 2\&3 are postponed to   Section 7.
 In Sections 5-6, we prove   Theorem  \ref{THM1.1} (resp. \ref{THM1.2})  and
  Corollary \ref{COR1.3} when $H\in C^0(\rr^2)$ (resp. $H\in C^1(\rr^2)$) satisfies  (H1')\&(H2).
Note that from this and Remark \ref{REM1.8}, we conclude
 Theorem     \ref{THM1.1} (resp. \ref{THM1.2}) and Corollary \ref{COR1.3} for general $H\in C^0(\rr^2)$ (resp. $H\in C^1(\rr^2)$) satisfying (H1)\&(H2)

 In the sequel of this paper (except in  the Appendix)
we always assume that $H\in C^0(\rr^2)$ satisfies (H1')\&(H2).
The ideas of the proofs for Theorems \ref{THM1.1}\&\ref{THM1.2} and Corollary \ref{COR1.3} are sketched as below.

\medskip

  {\it Step 1.
 Assuming that $H\in C^\fz(\rr^2)$ satisfies  (H1'){\rm\&}(H2), we prove Theorems \ref{THM1.1}{\rm \&}\ref{THM1.2}}.
 The proofs  are partially
motivated by the approach developed by  \cite{kzz} in the special case $\frac12|p|^2$,
which is based on a structural identity of $\infty$-Laplace operator $\Delta_\fz$ and the well-known divergence formula for determinant; see \eqref{fund1}\&\eqref{fund2} in Section 2.4.
Since $H$ does not have Hilbert structure in general,
  the algebraic and geometric structures of Aronsson operator $\mathscr A_H$ are
much more complicated than those of  $\Delta_\fz$,  
 there are several  essential and also technical difficulties to prove Theorem \ref{THM1.1}\&\ref{THM1.2}. 

The first difficulty  is to understand the   fundamental  structure
of the Aronsson operator  $\mathscr A_H$ and its connection with determinant.
Fortunately,  in Section 2.1 we overcome  this difficulty by building up
a structural identity for $\mathscr A_H$  in plane:
	\begin{align}\label{fund3}
& (-{
		\rm det}D^2v )\langle [D^2_{pp} H (Dv)]^\ast  {D_pH}(Dv) , {D_pH}(Dv) \rangle \\
&\quad\quad =\langle  {D^2_{pp}H}(Dv)  D[H(Dv)] ,D[H(Dv)] \rangle-{\rm div}[ {D_pH} (Dv)] \mathscr A_{H }[v]
\quad {\rm }\forall v\in C^\fz,\nonumber
\end{align}
where  $(D^2_{pp} H)^\ast$   is the adjoint matrix of $D^2_{pp}H$; see Lemma \ref{LEM2.1}. Moreover,
  we derive in Lemma \ref{LEM2.2} a   divergence formula of the determinant:
	\begin{align}\label{fund4}&2(-\det D^2v )[\det D^2_{pp} H (Dv) ]\\
&\quad\quad={\rm div}\Big\{ {D^2_{pp}H}(Dv)   D[H(Dv)]  -{\rm div}[ {D_pH}(Dv)  ]    {D_pH}(Dv)   \Big\} \quad {\rm }\forall v\in C^\fz \nonumber
\end{align}
which turns out to match  with $\mathscr A_H$ in a perfect way. See Section 2.4 for some motivations.

The above identities  \eqref{fund3}\&\eqref{fund4}
  lead  us to consider the  $e^{\frac1{ \ez} H}$-harmonic equations in Sections 2.2\&2.3:
$$
{\rm div}\left[e^{\frac1{ \ez} H(Dv)}D_p H(Dv)\right] =\frac1{\ez}e^{\frac1{ \ez} H(Dv)}\left\{\mathscr A_H[v]+\ez\, {\rm div}[ {D_pH} (Dv)]\right\}=0\ {\rm in}  \  U,
$$
which was originally suggested by Evans \cite{e03}. If $u^\ez\in C^\fz(U)$ is an $e^{\frac1{ \ez} H}$-harmonic function,
from the key \eqref{fund3} and  another  key fact  that $\wz\tau_H(0)=1/2$ and $\wz\tau_H\in C^0(\rr^2)$
we derive  the following useful identity in Theorem \ref{LEM2.3}:
\begin{align}\label{fund5x}
 [-\det D^2u^\epsilon] [\det D ^2_{pp}H(Du^\ez)]
&= 4\wz \tau_H(Du^\ez)  \langle  {D^2_{pp}H}(Du^\ez)  D[H(Du^\ez)]^{1/2} ,D[H(Du^\ez)]^{1/2} \rangle  \\
 &\quad\quad+   \wz \tau_H(Du^\ez)
\frac{  \epsilon( {\rm div}[D_pH (Du^\ez)])^2}{ H(Du^\ez) }\quad\mbox{a.\ e. in $U$.}\nonumber
\end{align}
Note that to get this,  a careful/tedious analysis of the vanishing set of $Du^\ez$ is required.
By this and applying
  \eqref{fund4} for $-\det D^2u^\ez$, we  obtain the following
   quantative estimates in Lemmas \ref{LEM2.6}--\ref{LEM2.8} (whose proofs are postponed to Section 7):
 \begin{enumerate}
 \item[$\bullet$] 
 a quantative
$W^{1,2}(V)$-estimate for $[H(Du^\ez)+\sz]^\alpha$ via $\|H(Du)\|_{L^2(W)}$ and error terms
for any $V\Subset W$ and $\alpha>1/2-\tau_H(\|H(Du)\|_{L^\fz(W)})$, where $\sz=0$ when
$\alpha\ge1/2$ and $\sz>0$ when $\alpha<1/2$;
\item[$\bullet$] a similar $W^{1,2}(V)$-estimate for $e^{\frac1\ez H(Du^\ez)}$;
\item[$\bullet$] 
  an integral flatness estimate for $u^\ez$, that is, for any linear function $F$,
 the  $L^2(\frac12B)$-norm of $\langle D_p H(Du^\ez),Du^\ez-DF\rangle$
 is bounded via $(\frac{u^\ez-F}{r})^2$ for all balls $B\Subset U$.
\end{enumerate}

In Section 3, let $u$ be a viscosity solution to \eqref{eq1.2} in $ \Omega\Subset\rr^2$.
Let  $U\Subset\Omega$ be a smooth domain, for $\ez\in(0,1]$,  consider  the Dirichlet problem
$$
{\rm div}[e^{\frac1{ \ez} H(Dv)}D_pH(Dv)] =0\ {\rm in}  \  U; v|_{\partial U}=u|_{\partial U}.
$$
By the arguments of Evans \cite{e03} and \cite{ey04},     there is a unique solution
  $u^\epsilon\in C^{\fz}(U) \cap C(\overline U)$ to the above Dirichlet problem %
and, $u^\ez\to u$ in   $C(\overline U)$ as $\ez\to0$; see Theorem \ref{THM3.1}.
We next show in Theorem \ref{THM3.2}   that
 \begin{equation}\label{eq1.w1}  \limsup_{\ez\to0}\|H(Du^\ez)\|_{L^\fz(V)}\le  \|H(Du )\|_{L^\fz(U)} \quad \forall   \ V\Subset U, \end{equation}
  which is crucial for us to obtain the range of $\alpha$ in Theorem \ref{THM1.1}\&\ref{THM1.2}.
     To prove this, denote by   $\sz^\ez$ as the normalization of $e^{\frac1\ez H(Du^\ez)}$ and by $\phi$   a cut-off function for $V$.
Using the $W^{1,2}_\loc$-estimates of $e^{\frac1\ez H(Du^\ez)}$ in Lemma \ref{LEM2.6},
we derive some quantative estimates for $ \| \sz^\ez\phi ^{2^{k+1}}\|_{L^{2^{k+1}}(U)} $ in Lemma \ref{LEM3.3}\&\ref{LEM3.4}, whose proofs are postponed to Section 7.
This allows us to prove \eqref{eq1.w1} via a Moser type iteration similar to  that of Evans.
Note that the approach based on maximal principle in \cite{es11b}
  is not enough to get \eqref{eq1.w1}; see Remark \ref{REM3.5}.

  Moreover,  from $W^{1,2}(V)$-estimates of $[H(Du^\ez)]^\alpha$ in Lemma \ref{LEM2.7} and
  the integral flatness of $u^\ez$ in Lemma \ref{LEM2.8},
  we deduce the following uniform estimates and convergence in  Theorems \ref{THM3.5}\&\ref{THM4.2}:
 \begin{enumerate}
\item[$\bullet$]
 when $\alpha\ge 1/2$,  $[H(Du^\ez)]^\alpha\in W^{1,2}(V)$ uniformly in $\ez\in(0,1]$;    when $1/2-\tau_H(\|H(Du)\|_{L^\fz(U)})<\alpha<1/2$,
  $\liminf_{\ez\to0}\|[H(Du^\ez)+\sz]^\alpha\|_{L^2(V)}$ is bounded uniformly in $\sz\in(0,1]$.
\item[$\bullet$] 
when $\alpha\ge 1/2$, 
   $[H(Du^\ez)]^\alpha\to [H(Du)]^\alpha$ in $L^t _\loc(U) $ for all $t\ge1$ and weakly in $W^{1,2}_\loc(U)$ as $\ez\to0$;
    while when $1/2-\tau_H(\|H(Du)\|_{L^\fz(U)})<\alpha<1/2$,  
  $[H(Du^\ez)+\sz]^\alpha\to [H(Du)]^\alpha$ in $L^t _\loc(U) $ for all $t\ge1$ and weakly in $W^{1,2}_\loc(U)$ as $\ez\to0$ and $\sz\to0$ in order.  Moreover,  $u^\ez\to u$ in $W^{1,t}_\loc (U)$ for all $t\ge1$ as $\ez\to0$. 
 \end{enumerate}

From these uniform estimates and  convergence, in Section 4 we conclude Theorems \ref{THM1.1}\&  \ref{THM1.2}.
   We also deduce in Lemma \ref{LEM4.3} an integral flatness of $u$ from the integral flatness of $u^\ez$ in Lemma \ref{LEM2.8}
   and  the local Sobolev convergence  of $u^\ez$ above. 

\medskip
  {\it Step 2. Assuming that $H\in C^0(\rr^2)$ (resp. $H\in C^1(\rr^2)$) satisfies (H1'){\rm\&}(H2), we prove
  Theorem   \ref{THM1.1} (resp. \ref{THM1.2}) and Corollary \ref{COR1.3}.}
    Since
      \eqref{fund3}, \eqref{fund4} and     $e^{\frac1\ez H}$-harmonic functions  are not well-understood in this generality, the approach in Step 1  is not enough to prove Theorems \ref{THM1.1}\&\ref{THM1.2}      New ideas are needed.
    Instead of Evans' approximation, we consider another approximation approach.
   Precisely,
    let $\{H^\dz\}_{\dz\in(0,1]}$ be the smooth approximation of $H$ given in Appendix  A.
 Then $H^\dz\in C^\fz(\rr^2)$ satisfies (H1')\&(H2) uniformly.
 Given any $u\in AM_H(\Omega)$ and $U\Subset\Omega$,  let
 $$\mbox{$u^\dz\in C^0(\overline U)\cap AM_{H^\dz}(U)$ with $u^\dz=u$ on $\partial U$.}$$
In Theorem \ref{THM5.1}
 we show that $u^\dz\to u$ in $C^0(\overline U)$   and moreover,
\begin{equation} \label{eyy} \lim_{\dz\to0}\|H^\dz(Du^\dz)\|_{L^\fz(U)}\le \|H (Du )\|_{L^\fz(U)}.
\end{equation}

 Since Theorem \ref{THM1.1}\&\ref{THM1.2}  and the flatness estimate in Lemma \ref{LEM4.3} hold for $H^\dz$ and $u^\dz$,
using \eqref{eyy} we obtain the following uniform estimates  and convergence in
 Theorems \ref{THM5.2}\& \ref{THM5.3}:
 \begin{enumerate}
\item[$\bullet$]  there is a sequence $\{\dz_j\}_{ j\in\nn}$ which converges to $0$ such that
for  $\alpha>1/2-\tau_H(\|H(Du)\|_{L^\fz(U)})$,
  $[H^{\dz_j}(Du^{\dz_j})]^\alpha\in W^{1,2}_\loc(U)$ uniformly in $j\ge j_\alpha$.

\item[$\bullet$] For any $\alpha>1/2-\tau_H(\|H(Du)\|_{L^\fz(U)})$,  $[H^{\dz_j}(Du^{\dz_j})]^\alpha\to [H(Du)]^\alpha$ in $L^t _\loc(U) $ for all $t\ge1$ and weakly in $W^{1,2}_\loc(U)$ as $j\to \fz$.
    If $H\in C^1(\rr^2)$ additionally,
   $u^{\dz_j}\to u$ in $W^{1,t}_\loc (U)$  in $L^t_\loc(U)$ for all $t\ge1$ as $j\to\fz$.
\end{enumerate}

From these uniform estimates and convergence, and Lemma \ref{LEM7.8}, in Section 6 we conclude  Theorems  \ref{THM1.1}\&\ref{THM1.2}. We also conclude
Corollary \ref{COR1.3}   from Theorem \ref{LEM7.10x} and Theorem \ref{THM1.1}.
%

\subsection{Some conventions}
In this paper, let $C^0(E)$ be the set of continuous functions in a set $E\subset\rr^2$.
For $\alpha\in(0,1]$,  denote by $C^ {0,\alpha}(K)$  the set of all
 $\alpha$-order H\"older functions in a compact set $ K\subset\rr^2$,  and by
 $C^ {0,\alpha}(U)$  all functions in an open set $U\subset\rn$ which  belong  to
 $C^ {0,\alpha}(K)$ for any compact set $K\subset U$.
 For $k\ge1$, $C^k(U)$ is the set of all functions whose $k$-order derivative in $C^0(U)$, and $C^\fz(U)=\cap_{k\ge1} C^k(U)$. For $0\le k\le\fz$, $C^k_c(U)$ consists of all $C^ {k}(U)$-functions with compact supports.
  For $t\in[1,\fz]$ denote by $L^t(U)$ the $t$-th integrable Lebesgue space,
and by $L^t_\loc(U)$  the class of functions which belong to $L^t(V)$ for any $V\Subset U$.
Denote by $W^{1,t}(U)$ (resp. $W^{1,t}_\loc(U)$) the set of all
  functions whose $1$-order distributional derivatives are  in $L^t(U)$ (resp. $L^t_\loc(U)$).
Note that $C^{0,1}(U)=  W^{1,\fz}_\loc(U)$.

In this paper, we write $Dv: =(v_{x_i })_{i=1}^n$  with $ v_{x_i }=\frac{\partial  v}{\partial x_i } $
when $v\in C^1(U)$, and $D^2v: =(v_{x_i x_j})_{i,j=1}^n$   with $ v_{x_ix_j}=\frac{\partial^2 v}{\partial x_i\partial x_j} $   when $v\in C^2(U)$.  When $v\in L^1_\loc(U)$,
 we explain these notions  in  distributional sense.
 We also write $ D_pH=(H_{p_i})_{i=1}^n$ with $ H_{p_i} =
\frac{\partial H }{\partial p_i}  $   when $H\in C^1(\rn)$, and $ D^2_{pp}H=(H_{p_ip_j})_{i,j=1}^n$  and  $H_{p_ip_j} =\frac{\partial^2 H }{\partial p_i\partial p_j}  $     when $H\in C^2(\rn)$.

For two vectors $a=(a_i)_{i=1}^n$ and $b=(b_i)_{i=1}^n$,
write $\langle a, b\rangle=\sum_{i=1}^na_ib_i$,
and by Einstein summation convention we also write $\sum_{i=1}^n a_ib_i$ as $a_ib_i$.
 The notion $V\Subset U$ means that both of $V,U$ are open set, $V$ is bounded and $\overline V \subset U$.
 Moreover, we write $C$ as an absolute constant,
 or an constant  independent of the main parameters when there is no confusion,
 moreover, write  $C(a,b,..)$ as a constant depending on the parameters $a, b,\cdots$. For a measurable set $E\subset $ with $|E|>0$ and we write $\mint-_Ef\,dx=\frac1{|E|}\int_Ef(x)\,dx$.

\section{Structural identities and  apriori estimates   when $H\in C^\fz(\rr^2)$}
Suppose that $H \in C^\fz(\rr^2)$ satisfies  (H1')\&(H2) in this section.

In Section 2.1,  we build up a fundamental structural identity for $\mathscr A_H$ (see Lemma \ref{LEM2.1}),
and  a  divergence formula of $-\det D^2v$ matching with $\mathscr A_H$ perfectly (see Lemma \ref{LEM2.2}).
See Section 2.4 for some ideas/motivations to find the  two identities.

In Sections 2.2\&2.3,
applying Lemma \ref{LEM2.1}, we conclude   a   key identity connecting $-\det D^2u^\ez$, ${\,\rm div\,}[D_pH(Du^\ez)]$ and $D[H(Du^\ez)]^{1/2}$ via  $\wz\tau_H$  for any $e^{\frac1\ez H}$-harmonic function  $u^\ez\in C^\fz(U)$; see Theorem \ref{LEM2.3}.
Via this and Lemma \ref{LEM2.2}  we derive  $W^{1,2}_\loc$-estimates of
$[H(Du^\ez)+\sz]^\alpha$  and $e^{\frac1\ez H(Du^\ez)}$ (see Lemmas \ref{LEM2.6}\&\ref{LEM2.7}),
and also an integral flatness estimates for $u^\ez$ (see Lemma \ref{LEM2.8}).

\subsection{Two structural identities for Aronsson's operator}
Suppose that  $H\in C^\fz(\rn)$ is convex.   We have the following fundamental structural identity for $\mathscr A_H$.
For any matrix $A=[a_{ij}]_{i,j=1}^2$, denote by $A^\ast$  its adjoint matrix, that is,
	$$A^\ast=\left[
	\begin{array}{cc}
	a_{22}&\ -a_{21} \\
	-a_{12}&\ a_{11}\\
	\end{array}
	\right].$$
\begin{lem}\label{LEM2.1} Let $U\subset \rr^2$ be a domain. For any $v \in C^\fz(U)$,  we have
	\begin{align*}
&\langle  {D^2_{pp}H}(Dv)  D[H(Dv)] ,D[H(Dv)] \rangle-{\rm div}[ {D_pH} (Dv)] \mathscr A_{H }[v]\\
&\quad\quad\quad\quad\quad= (-{
		\rm det}D^2v )\langle [D^2_{pp} H(Dv)]^\ast  {D_pH}(Dv) , {D_pH}(Dv) \rangle\quad {\rm in}\ U.
\end{align*}
\end{lem}

\begin{proof}
At each $x\in U$, write
\begin{align*}
			 \langle D^2_{pp}H(Dv)  D[H(Dv)] ,D[H(Dv)] \rangle&=v_{x_lx_j}H _{p_j}(Dv) H _{p_lp_i}(Dv) v_{x_ix_k}H _{p_k}(Dv) \\
			&=H _{p_1p_1}(Dv) v_{x_1x_j} H _{p_j}(Dv) v_{x_1x_k}H _{p_k}(Dv)\\
 &\quad+H _{p_2p_2}(Dv) v_{x_2x_j}H _{p_j}(Dv) v_{x_2x_k}H _{p_k}(Dv)\\
&\quad +2H _{p_1p_2}(Dv) v_{x_1x_j}
			H _{p_j}(Dv) v_{x_2x_k}H _{p_k}(Dv).
\end{align*}
Noting that  $(v_{x_1x_2})^2-v_{x_1x_1}v_{x_2x_2}=-\det D^2v$ and
$$\mathscr A_H[v]= v_{x_1x_1}[H _{p_1}(Dv) ]^2+2H _{p_1}(Dv)H _{p_2}(Dv)  v_{x_1x_2}+v_{x_2x_2}[H _{p_2}(Dv) ]^2,$$
we have
\begin{align*}
&  v_{x_1x_j} H _{p_j}(Dv) v_{x_1x_k}H _{p_k}(Dv)\\
&\quad=  [v_{x_1x_1}H _{p_1}(Dv) ]^2+[v_{x_1x_2}H _{p_2}(Dv) ]^2+2v_{x_1x_1}v_{x_1x_2}H _{p_1}(Dv) H _{p_2}(Dv) \\
&\quad=  v_{x_1x_1} \Big\{v_{x_1x_1}[H _{p_1}(Dv) ]^2+2H _{p_1}(Dv)H _{p_2}(Dv)  v_{x_1x_2}+v_{x_2x_2}[H _{p_2}(Dv) ]^2\Big\}\\
	&\quad\quad-v_{x_1x_1}v_{x_2x_2}
			[H _{p_2}(Dv) ]^2
+   [v_{x_1x_2}H _{p_2}(Dv) ]^2\\
&\quad=    v_{x_1x_1}\mathscr A_H[v]  +(-\det D^2v)
			[H _{p_2}(Dv) ]^2.
\end{align*}
Similarly, we have
\begin{align*}
&  v_{x_2x_j} H _{p_j}(Dv) v_{x_2x_k}H _{p_k}(Dv) =   v_{x_2x_2}\mathscr A_H[v]  +(-\det D^2v)
			[H _{p_1}(Dv)]^2.
\end{align*}
Moreover,
\begin{align*}
 &  v_{x_1x_j}
			H _{p_j}(Dv) v_{x_2x_k}H _{p_k}(Dv)\\
&\quad=  [v_{x_1x_1}H _{p_1}(Dv) +v_{x_1x_2}H _{p_2}(Dv) ][v_{x_2x_1}H _{p_1}(Dv) +v_{x_2x_2}H _{p_2}(Dv) ]\\
& \quad= v_{x_1x_2}\Big\{v_{x_1x_1}[H _{p_1}(Dv)]^2+v_{x_2x_2}[H _{p_2}(Dv)]^2+2v_{x_1x_2}H _{p_1}(Dv) H _{p_2}(Dv)  \Big\}\\
			&\quad\quad -  [v_{x_1x_2}]^2H _{p_1}(Dv) H _{p_2}(Dv) +  v_{x_1x_1}v_{x_2x_2}H _{p_1}(Dv) H _{p_2}(Dv) \\
& \quad= v_{x_1x_2}\mathscr A_H[v]-(-\det D^2v) H _{p_1}(Dv) H _{p_2}(Dv).
\end{align*}
Combining them together, we obtain
\begin{align*}
			&\langle D^2_{pp}H(Dv)  D[H(Dv)] ,D[H(Dv)] \rangle\\
&\quad=H _{p_1p_1}(Dv) v_{x_1x_1}\mathscr A_H[v]  +(-\det D^2v)H _{p_1p_1}(Dv)
			[H _{p_2}(Dv) ]^2\\
 &\quad\quad+H _{p_2p_2}(Dv) v_{x_2x_2}\mathscr A_H[v]  +(-\det D^2v)H _{p_2p_2}(Dv)
			[H _{p_2}(Dv) ]^2\\
&\quad\quad +2H _{p_1p_2}(Dv) v_{x_1x_2}\mathscr A_H[v]-(-\det D^2v) 2H _{p_1p_2}(Dv)H _{p_1}(Dv) H _{p_2}(Dv).
	\end{align*}
Noting that
	$${\rm div}[D_p H (Dv)]=H _{p_1p_1}(Dv)v_{x_1x_1}+H _{p_2p_2}(Dv)v_{x_2x_2}+2H _{p_1p_1}(Dv)v_{x_1x_2},$$
   and 	\begin{align*}
 &\langle [D^2_{pp} H (Dv)]^\ast  {D_pH}(Dv) , {D_pH}(Dv) \rangle\\
 &\quad=H _{p_1p_1}(Dv) [H _{p_2}(Dv) ]^2 +H _{p_2p_2}(Dv) [H _{p_1}(Dv) ]^2 - 2H _{p_1}(Dv) H _{p_2}(Dv) H _{p_1p_1}(Dv),
\end{align*}
we conclude
 \begin{align*}
			&\langle D^2_{pp}H(Dv)  D[H(Dv)] ,D[H(Dv)] \rangle\\
			   			&\quad={\rm div}[D_p H (Dv)]{\mathscr A}_{H }[v]+(-\det D^2v)\langle [D^2_{pp} H (Dv)]^\ast  {D_pH}(Dv) , {D_pH}(Dv) \rangle
	\end{align*}
as  desired.  This completes the proof of Lemma\ref{LEM2.1}.
\end{proof}

To get a divergence formula of $-\det D^2v $  matching with  $\mathscr A_H$, we
write the left hand side of the identity in Lemma \ref{LEM2.1} as
\begin{align*}
&\langle  {D^2_{pp}H}(Dv)   D[H(Dv)]  ,D[H(Dv) ] \rangle-{\rm div}[ {D_pH}(Dv)  ] \mathscr A_{H }[v]\\
&\quad=
 \langle  {D^2_{pp}H}(Dv)   D[H(Dv)]  -{\rm div}[ {D_pH}(Dv)  ]    {D_pH}(Dv)  , D[H (Dv)] \rangle.
 \end{align*}
 Then the divergence of  the vector fields $$ {D^2_{pp}H}(Dv)   D[H(Dv)]  -{\rm div}[ {D_pH}(Dv)  ]    {D_pH}(Dv)$$ gives a divergence formula of $-\det D^2v $ as below.
\begin{lem}\label{LEM2.2} Let $U\subset \rr^2$ be a domain. For any $v \in C^\fz(U)$,  we have
\begin{align*}&2(-\det D^2v ) [\det  D^2_{pp}  H(Dv) ]
 ={\rm div}\Big\{ {D^2_{pp}H}(Dv)   D[H(Dv)]  - {\rm div}[ {D_pH}(Dv)  ]    {D_pH}(Dv)   \Big\}
 \quad{\rm in}\ U.\end{align*}
\end{lem}
\begin{proof} By a direct calculation, at each $x\in U$  we have
\begin{align*}
J&:={\rm div}\Big\{ {D^2_{pp}H}(Dv)  D[H (Dv)]-{\rm div}[ {D_pH}(Dv) ] {D_pH} (Dv)   \Big \}\\
& =  [H _{p_ip_m}(Dv)v_{x_mx_j}H_{p_j}(Dv)]_{x_i}-[H_{p_i}(Dv)H _{p_mp_j}(Dv)v_{x_mx_j}]_{x_i} \\
& =  [H _{p_ip_m}(Dv) H_{p_j}(Dv)]_{x_i}v_{x_mx_j}-[H_{p_i}(Dv)H _{p_mp_j}(Dv)]_{x_i}v_{x_mx_j} \\
& =  H _{p_ip_m}(Dv)[H_{p_j}(Dv)]_{x_i} v_{x_mx_j}-[H_{p_i}(Dv)]_{x_i}H _{p_mp_j}(Dv)v_{x_mx_j}  \\
&\quad +H_{p_j}(Dv)[H _{p_ip_m}(Dv)]_{x_i}v_{x_mx_j} -  H_{p_i}(Dv)[H _{p_mp_j}(Dv)]_{x_i}v_{x_mx_j} .
\end{align*}
Note that
 \begin{align*}
H_{p_i}(Dv)[H _{p_mp_j}(Dv)]_{x_i}v_{x_mx_j} &= H_{p_i}(Dv) v_{x_lx_i}H _{p_mp_jp_l}(Dv) v_{x_mx_j}=H_{p_j}(Dv)[H _{p_ip_m}(Dv)]_{x_i}v_{x_mx_j}.
\end{align*}
Write  \begin{align*}
H _{ p_ip_m}(Dv)[H_{p_j}(Dv)]_{x_i}v_{x_jx_m}
 &= H _{ p_ip_m}(Dv)v_{x_ix_l}H _{p_lp_j}(Dv)v_{x_jx_m} \\
&=  H _{p_1p_i}(Dv)v_{x_ix_1}H _{p_1p_j}(Dv)v_{x_jx_1}
+2H _{p_1p_i}(Dv)v_{x_ix_2}H _{p_2p_j}(Dv)v_{x_jx_1}\\
&\quad+
H _{p_2p_i}(Dv)v_{x_ix_2}H _{p_2p_j}(Dv)v_{x_jx_2}
\end{align*}
and similarly,
\begin{align*}-  [H_{p_i}(Dv)]_{x_i}H _{p_mp_j}(Dv)v_{x_mx_j}
&=- H _{p_1p_s}(Dv)v_{x_sx_1}H _{p_1p_j}(Dv)v_{x_j x_1}-2H _{p_1p_s}(Dv)v_{x_sx_1}H _{p_2p_j}(Dv)v_{x_j x_2}\\
&\quad - H _{p_2p_s}(Dv)v_{x_sx_2}H _{p_2p_j}(Dv)v_{x_j x_2}.\end{align*}
One gets
 \begin{align*}J
= 2H _{p_1p_i}(Dv)v_{x_ix_2}H _{p_2p_j}(Dv)v_{x_jx_1}-2H _{p_1p_s}(Dv)v_{x_sx_1}H _{p_2p_j}(Dv)v_{x_j x_2}.
\end{align*}
Since
\begin{align*}& 2H _{p_1p_i}(Dv)v_{x_ix_2}H _{p_2p_j}(Dv)v_{x_jx_1}\\
&\quad=  2H _{p_1p_1}(Dv)v_{x_1x_2}H _{p_2p_1}(Dv)v_{x_1x_1}+ 2H _{p_1p_2}(Dv)v_{x_2x_2}H _{p_2p_1}(Dv)v_{x_1x_1}\\
&\quad\quad +2H _{p_1p_1}(Dv)v_{x_1x_2}H _{p_2p_2}(Dv)v_{x_2x_1}+  2H _{p_1p_2}(Dv)v_{x_2x_2}H _{p_2p_2}(Dv)v_{x_2x_1}
\end{align*}
and \begin{align*}
&-2H _{p_1p_s}(Dv)v_{x_sx_1}H _{p_2p_j}(Dv)v_{x_j x_2}\\
&\quad= -2H _{p_1p_1}(Dv)v_{x_1x_1}H _{p_2p_1}(Dv)v_{x_1 x_2}- 2  H _{p_1p_1}(Dv)v_{x_1x_1}H _{p_2p_2}(Dv)v_{x_2 x_2}\\
&\quad\quad- 2H _{p_1p_2}(Dv)v_{x_2x_1}H _{p_2p_1}(Dv)v_{x_1 x_2}- 2H _{p_1p_2}(Dv)v_{x_2x_1}H _{p_2p_2}(Dv)v_{x_2 x_2} , \end{align*}
we conclude that
\begin{align*}
J&=     2H _{p_1p_2}(Dv)v_{x_2x_2}H _{p_2p_1}(Dv)v_{x_1x_1}  + 2H _{p_1p_1}(Dv)v_{x_1x_2}H _{p_2p_2}(Dv)v_{x_2x_1}\\
&\quad-2  H _{p_1p_1}(Dv)v_{x_1x_1}H _{p_2p_2}(Dv)v_{x_2 x_2} - 2H _{p_1p_2}(Dv)v_{x_2x_1}H _{p_2p_1}(Dv)v_{x_1 x_2}\\
&=2[(v_{x_1x_2})^2-v_{x_1x_1}v_{x_2x_2} ]\Big\{ H _{p_1p_1}(Dv)H _{p_2p_2}(Dv)-[H _{p_1p_2}(Dv)]^2\Big\}\\
& =2(-\det D^2v) \det D^2_{pp} H(Dv)
\end{align*}
as desired.  This completes the proof of Lemma\ref{LEM2.2}.
\end{proof}

\subsection{An identity  for $-\det D^2u^\ez$ 
}

Suppose   $H\in C^\fz(\rn)$ satisfies  (H1')\&(H2).   Let $U\subset\rn$ be a domain.
For any $\epsilon \in(0,1]$  let $u^\epsilon\in C^\infty(U)$ be a  solution to the equation
\begin{equation}\label{eq2.1}
 {\mathscr A}_{H }[u^\epsilon ]+\epsilon \,{\rm{div}}[D_p H (Du^\epsilon) ] =0\quad{\rm in}\ U,\end{equation}
or equivalently,    the $e^{\frac1\ez H}$-harmonic equation
\begin{equation}\label{eq2.1x}
{\rm div}\left[e^{\frac1{ \ez} H(Du^\ez)}D_p H(Du^\ez)\right]= \frac1\ez e^{\frac1{ \ez} H(Du^\ez)}\left\{{\mathscr A}_{H }[u^\epsilon ]+\epsilon \,{\rm{div}}[D_p H (Du^\epsilon) ]\right\}=0\quad{\rm in}\ U.\end{equation}

By Lemma \ref{LEM2.1} we have the following results.
Recall that $\wz \tau_H$ is defined in \eqref{eq1.7}. The fact that $\wz\tau_H(0)=1/2$ and  $\wz\tau_H\in C^0([0,\fz))$ given by Lemma \ref{LEM7.2} plays a important role here.

\begin{thm}\label{LEM2.3} We have $[H(Du^\ez)]^{1/2}\in C^{0,1}(U)$, and
\begin{align}\label{eq2.y1}
[-\det D^2u^\epsilon] [\det D ^2_{pp}H(Du^\ez)]
&= 4\wz \tau_H(Du^\ez)  \langle  {D^2_{pp}H}(Du^\ez)  D[H(Du^\ez)]^{1/2} ,D[H(Du^\ez)]^{1/2} \rangle  \\
&\quad+   \wz \tau_H(Du^\ez)
\frac{  \epsilon( {\rm div}[D_pH (Du^\ez)])^2}{ H(Du^\ez) } \nonumber
\end{align}
for all $  x\in U$  at which  $  [H  (Du^\ez )]^{1/2} $  is differentiable.
 In particular,  $-\det D^2u^\epsilon\ge0$ in $ U $.

\end{thm}
\begin{rem}\rm
Note that
$$\epsilon |\,{\rm{div}}[D_pH(Du^\epsilon) ]|= |{\mathscr A}_{H }[u^\epsilon ]|
\le |D^2u^\ez|| {D_pH}(Du^\ez)|^2\le \frac{[\Lambda_H(\fz)]^2}{2\lambda_H(\fz)}|D^2u^\ez| H(Du^\ez),$$
where we use Lemma \ref{LEM7.1} in the last inequality.
If $ {D_pH}(Du(\bar x))=0$ for some $\bar x\in U$,
  for $\beta\in(0,4)$  we always let
$$\frac{({\rm{div}}[D_pH(Du^\epsilon) ])^2}{[H(Du^\ez)]^{\beta/2}}=\frac{({\rm{div}}[D_pH(Du^\epsilon) ])^2}{| {D_pH}(Du^\ez)|^\beta}=0\quad\mbox{at $ \bar x$.}$$
\end{rem}
\begin{proof} Since $ H  ^{1/2}\in C^{0,1}(U)$ (see Lemma \ref{LEM7.1} (iv)) and $u^\ez\in C^\fz(U)$,
we have
$[H(Du^\ez)]^{1/2}\in C^{0,1}(U)$.
   By Rademacher's Theorem, $[H(Du^\ez)]^{1/2}$ is differentiable almost everywhere in $U$.

    Let $\bar x\in U$ and   assume that   $[H(Du^\ez)]^{1/2}$ is differentiable at $\bar x$.
  If $D_pH(Du^\epsilon(\bar x))\neq0$, by considering the positivity of the matrice $D^2_{pp} H $ and   $(D^2_{pp} H)^\ast$,
 we have
 $$-\det D^2u^\epsilon = \frac{\langle  {D^2_{pp}H}(Du^\ez)  D[H(Du^\ez)] ,D[H(Du^\ez)] \rangle+ \epsilon[ {\rm div}(D_p (Du^\ez))]^2}{\langle [D^2_{pp} H (Du^\epsilon)]^\ast  {D_pH} (Du^\ez), {D_pH}(Du^\ez) \rangle }.$$
 at $\bar x$, which, together with
 $$\mbox{$(D^2_{pp} H) ^\ast =(\det D^2_{pp} H) (D^2_{pp} H )^{-1}$
 and $D[H(Du^\ez)]= 2 [H(Du^\ez)]^{1/2}D[H(Du^\ez)]^{1/2}$} $$
 and the definition of $\wz\tau_H$,   gives the desired identity \eqref{eq2.y1}.

  Assume that $D_p H (Du^\epsilon(\bar x))=0$ below.  By
the convexity of $H$, we have
$$0=
H (0)\ge H (Du^\epsilon(\bar x))+ \langle D_p H (Du^\epsilon(\bar x)),0-Du^\epsilon(\bar x)\rangle =H (Du^\epsilon(\bar x))\ge0,$$
which implies $H (Du^\epsilon(\bar x))=0$ and hence $ Du^\epsilon(\bar x) =0$.

If $D[H(Du^\ez)]^{1/2}(\bar x) =0$, then
    $[H(Du^\ez(   x))]^{1/2}= o(|x-\bar x|)$
    and hence $$|Du^\ez(x)|^2 \le  \frac2{\lambda_H(\fz)}   H(Du^\ez(  x))  = o(|x-\bar x|^2 ).$$
    This implies that   $Du^\ez(x)= o(|x-\bar x| )$ and hence  $D^2u^\ez(\bar x)=0$.
    So, $-\det D^2u^\ez(\bar x)=0$ as desired.

      If $D[H(Du^\ez)]^{1/2}(\bar x)\ne 0$, then
      $$[H(Du^\ez(  x))]^{1/2}=\langle D[H(Du^\ez)]^{1/2}(\bar x),x-\bar x\rangle+  o(|x-\bar x|).$$
          Write $$ \vec b :  =D^2_{pp}H(0) D[H(Du^\ez)]^{1/2}(\bar x).$$
For $x=\bar x+t\vec b  $, we have
 \begin{align*}  H(Du^\ez(  \bar x+t \vec b))
  &=\left\{t\langle D[H(Du^\ez)]^{1/2}(\bar x),\vec b\rangle+  o(t)\right\}^2
  =t^2\langle D[H(Du^\ez)]^{1/2}(\bar x),\vec b\rangle^2+  o(t^2).
   \end{align*}

  On the other hand, since
 $$ Du^\ez(  x)= D^2u^\ez(\bar x) (x-\bar x)+o(|x-\bar x|)$$
 and
 $$H(p)=\frac12\langle D^2_{pp}H(0)p,p\rangle+o(|p|^2),$$
 we have
 \begin{align*}
  H(Du^\ez( \bar x+t \vec b))
 =
 H(tD^2u^\ez(\bar x) \vec b+o(t))
 &= \frac{t^2}2\langle D^2_{pp}H(0)D^2u^\ez(\bar x) \vec b, D^2u^\ez(\bar x) \vec b\rangle  +o(t^2).
 \end{align*}
 Thus,
 \begin{align*}
   2\langle D[H(Du^\ez)]^{1/2}(\bar x),\vec b\rangle^2
 &=  \langle D^2_{pp}H(0)D^2u^\ez(\bar x)\vec b, D^2u^\ez(\bar x) \vec b\rangle
 \end{align*}
 that is,
 \begin{align*}
 2\langle D[H(Du^\ez)]^{1/2}(\bar x),\vec b\rangle
 &=  \frac{\langle D^2_{pp}H(0)D^2u^\ez(\bar x) \vec b, D^2u^\ez(\bar x)\vec b\rangle} {  \langle[ D^2_{pp}H(0)] ^{-1} \vec b,\vec b\rangle}.
 \end{align*}
 Since $(D^2_{pp}H(0))^{-1}=[D^2_{pp}H(0)]^\ast(\det D ^2_{pp}H)^{-1}$, we obtain
  \begin{align*}
    2\langle D[H(Du^\ez)]^{1/2}(\bar x),\vec b\rangle
  &=\frac{\langle D^2_{pp}H(0)D^2u^\ez(\bar x)  \vec b , D^2u^\ez(\bar x)  \vec b\rangle} {  \langle [D^2_{pp}H(0)]^\ast \vec b,\vec b\rangle}\det D^2_{pp}H(0).
   \end{align*}
Writing $h_{ij}=H_{p_ip_j}(0)$, $ a_{ij}=u^\ez_{x_ix_j} (\bar x)$ and $ \vec b^T=(b_1,b_2)$, we have
   \begin{align*} \frac{\langle D^2_{pp}H(0)D^2u^\ez(\bar x)  \vec b, D^2u^\ez(\bar x)  \vec b\rangle} {  \langle [D^2_{pp}H(0)]^\ast \vec b,\vec b\rangle}  &=\frac{b_ia_{ik}h_{ks}a_{sj}b_{j}} {b_1b_1h_{22}-2b_1b_2h_{12}+b_2b_2h_{11}} . \end{align*}
   Note that
  $$a_{ij}h_{ij}={\,\rm div\,}[D_p H(Du^\ez)](\bar x)=\mathscr A_H[u^\ez](\bar x)=0$$
  that is,
  $$ 2a_{12}h_{12}+ a_{11}h_{11}+a_{22}h_{22}=0.$$
  Therefore, a direct calculation leads to that
      \begin{align*}
   &(b_1b_1h_{22}-2b_1b_2h_{12}+b_2b_2h_{11})(a_{12}a_{21}-a_{11}a_{22}) \\
   &=b_1b_1h_{22}a_{12}a_{21}-b_1b_1h_{22}a_{11}a_{22}\\
   &\quad- 2b_1b_2h_{12}a_{12}a_{21}+2b_1b_2h_{12}a_{11}a_{22}\\
   &\quad+b_2b_2h_{11}a_{12}a_{21}-b_2b_2h_{11}a_{11}a_{22}\\
   &= b_1b_1h_{22}a_{12}a_{21}+ 2b_1b_2h_{12}a_{11}a_{22}+ b_2b_2h_{11}a_{12}a_{21}\\
   &\quad+b_1b_1 a_{11} (2a_{12}h_{12}+a_{11}h_{11}) + b_1b_2 a_{21} (a_{11}h_{11}+a_{22}h_{22})
   +b_2b_2 a_{22} (2a_{12}h_{12}+a_{22}h_{22})\\
   &\quad +b_1b_2a_{12}[a_{11}h_{11}+a_{22}h_{22}+2h_{12}a_{12}]\\
   &= h_{11}(a_{1i}b_ia_{1s}b_s) +h_{22}(a_{2i}b_ia_{2s}b_s) +2h_{12}(a_{1i}b_ia_{2s}b_s)\\
   &= b_ia_{ik}h_{ks}a_{sj}b_{j}.
      \end{align*}
      Thus
         \begin{align*} \frac{\langle D^2_{pp}H(0)D^2u^\ez(\bar x)  \vec b, D^2u^\ez(\bar x)  \vec b\rangle} {  \langle [D^2_{pp}H(0)]^\ast b,\vec b\rangle}  &= -\det D^2u^\ez(\bar x) , \end{align*}
         that is,
          \begin{align*}
   2\langle D[H(Du^\ez)]^{1/2}(\bar x),D^2_{pp}H(0)  D[H(Du^\ez)]^{1/2}(\bar x)\rangle
 &=  [-\det D^2u^\ez(\bar x)] \det D^2_{pp}H(0) .
 \end{align*}
 Noting   $\wz \tau_H(0)= 1/2$, we arrive at the desired identity.
 This completes the proof of Theorem \ref{LEM2.3}.
\end{proof}

\subsection{$W^{1,2}_\loc$-estimates of $e^{\frac1\ez H(Du^\ez)}$ and
$[H(Du^\ez)]^\alpha$  and a  flatness of $u^\ez$
}

Suppose that  $H\in C^\fz(\rn)$ satisfies  (H1')\&(H2).   Let $U\subset\rn$ be any domain.
For any $\ez\in(0,1]$ let $u^\ez\in C^\fz(U)$ be a solution to the equation \eqref{eq2.1}.

Define the functional
\begin{align}\label{FUN2.2} I(u^\ez,\phi)=-2\int_U  \det D^2u^\ez   \det[D^2_{pp}H(Du^\ez) ]\vz\,d x
\quad\forall \vz \in C^0_c(U).
\end{align}
By Theorem \ref{LEM2.3}, write
\begin{align}\label{FUN2.3}I(u^\ez,\phi)
 &=8\int_U \wz \tau_H(Du^\ez)  \langle  {D^2_{pp}H}(Du^\ez)  D[H(Du^\ez)]^{1/2} ,D[H(Du^\ez)]^{1/2} \rangle\vz \,d x\nonumber\\
 &\quad\quad+2\epsilon\int_U \wz \tau_H(Du^\ez)[ {\rm div}(D_pH (Du^\ez))]^2  [ H(Du^\ez)]^{-1} \vz \,d x \quad\forall \vz \in C^0_c(U)
\end{align}
Via integration by parts and Lemma \ref{LEM2.2}, we also have
\begin{equation}\label{FUN2.4}
I(v,\phi)=-\int_U \langle  {D^2_{pp}H}(Dv)   D[H(Dv)]  -{\rm div}[ {D_pH}(Dv)  ]    {D_pH}(Dv)  ,  D\vz\rangle \,d x
\quad\forall \vz \in W^{1,2}_c(U).
\end{equation}

Taking $\varphi  = e^{\frac t\ez  H (Du^\epsilon )} \phi^{2 mt }$ in the functional  $I(u^\ez,\vz)$, using \eqref{FUN2.3} and \eqref{FUN2.4} we obtain the following result, whose proof is postponed to Section 6.

\begin{lem}\label{LEM2.6}For any $W\subset U$, $m\ge1$, $t\ge1$, and $\phi\in C^\infty_c(W)$, we have
\begin{align*}
&\frac1\ez\int_U | D[H(Du^\ez)]|^2   \phi^{2 mt } e^{\frac t\ez H(Du^\ez)} \,d x+
 \frac1{\Lambda_H(\fz)}\int_U({\rm div} [{D_pH}(Du^\ez)] )^2 \phi^{2 mt } e^{\frac t\ez H(Du^\ez)} \,d x\\
&\quad\le 8m^2 \left[\frac{ \Lambda_{H}(\fz)}{\lambda_H(\fz) }\right]^2\int_U  H(Du^\ez) |D\phi|^2\phi^{2 mt -2} e^{\frac t\ez H(Du^\ez)}\,d x.
\end{align*}
\end{lem}

Taking $\vz=[H(Du^\ez)+\sz]^{2\alpha-1}\phi^2$ in \eqref{FUN2.3} and \eqref{FUN2.4} for some cut-off function $\phi\in C^\fz(U)$  and $\sz>0$, using \eqref{FUN2.3} and \eqref{FUN2.4} we obtain the following result, whose proof is postponed to Section 6.
We write $\tau_H(\fz):=\lim_{R\to\fz}\tau_H(R)$ and note that
$$\frac12\ge \tau_H(\fz)\ge \limsup_{R\to\fz} \frac{[\lambda_H(R)]^2}{2[\Lambda_H(R)]^2}\ge   \frac{[\lambda_H(\fz)]^2}{2[\Lambda_H(\fz)]^2}.$$

\begin{lem}\label{LEM2.7}
(i)  Let $\alpha\ge 1/2$. For any $W\Subset U$ and $\phi\in C_c^\fz(W)$ we have
 \begin{align*}
     &  \int_{U} |D[H(Du^\ez) ]^\alpha|^ 2\phi^2 \,dx
  +\frac{\ez\alpha^2}{\Lambda_H(\fz)} \int_U
  \frac{({\rm div}[ {D_pH}(Du^\ez) ])^2  }{   [H(Du^\ez)]^{2-2\alpha}}\phi^2 \,dx\\
    &\quad\le  \frac{ C\alpha^2(\alpha+1)}{[\alpha+\tau_H(\|H(Du^\ez)\|_{L^\fz( W)})-\frac12]^2} \left[\frac{ \Lambda_{H}(\|H(Du^\ez)\|_{L^\fz(W)}) }{\lambda_H(\|H(Du^\ez)\|_{L^\fz(W)}) }\right]^2
 \int_U       [ H(Du^\ez) ]^{2\alpha  }[|D\phi|^2+|D^2\phi||\phi|]  \,dx\\
    &\quad\quad+\frac{C\ez\alpha^2|2\alpha-1|^2}{ [\alpha+\tau_H(\fz) -\frac12]^2}
     \left[\frac{ \Lambda_{H}(\fz) }{\lambda_H(\fz) }\right]^2
 \int_U     [ H(Du^\ez) ]^{2\alpha-1 } |D\phi|^2   \,dx.\end{align*}

 (ii)  For any  $W \Subset U$ and $\sz\in(0,1]$,
 if  $\frac12-\tau( \|H(Du^\ez)\|_{L^\fz( W)})   < \alpha< \frac12$, then for any $\phi\in C_c^\fz(W)$ we have
  \begin{align*}
&\int_{U} | D[H(Du^\ez)+\sigma]^\alpha|^2 \phi^2
     \,dx +\frac{\epsilon\alpha^2}{\Lambda_H(\fz)}\int_{U}
     \frac{( {\rm div}[D_pH (Du^\ez)])^2}
     {[H (Du^\epsilon)+\sigma]^{2-2\alpha }} \phi^2 \,dx\\
   &\quad\le   \frac{ C\alpha^2(\alpha+1)}{[\alpha+\tau_H(\|H(Du^\ez)\|_{L^\fz( W)})-\frac12]^2}
  \left[\frac{ \Lambda_{H}(\|H(Du^\ez)\|_{L^\fz(W)}) }{\lambda_H(\|H(Du^\ez)\|_{L^\fz(W)}) }\right]^2
  \\
    &\quad\quad\quad\times
   \int_{U}        [ H(Du^\ez) +\sigma]^{2\alpha  } [|D\phi|^2+|D^2\phi||\phi|] \,dx\\
&\quad\quad+    \frac{C\ez\alpha^2|2\alpha-1|^2 }{ [\alpha+\tau_H(\|H(Du^\ez)\|_{L^\fz( W)}) -\frac12]^2}
 \left[\frac{ \Lambda_{H}(\fz) }{\lambda_H(\fz) }\right]^2
 \int_{U}       [ H(Du^\ez) +\sigma]^{2\alpha-1 } |D\phi|^2  \,dx.    \end{align*}

\end{lem}

 Taking $\varphi  = \frac12(u^\ez-F) \phi^{4}$ in the functional  $I(u^\ez,\vz)$ for any linear function $F$, using \eqref{FUN2.3} and \eqref{FUN2.4} we obtain the following result, whose proof is postponed to Section 6.

\begin{lem}\label{LEM2.8}
For any $B=B(\bar x,r)\Subset U$ and linear function $F$, we have
\begin{align*}
  &\mint-_{\frac12B }\langle D_pH(Du^\ez),Du^\epsilon-DF\rangle ^2 \,d x\\
  &\quad\le C \Lambda_H(\fz)
  \left[\int_{\frac34B }|D[H(Du^\ez)]|^2  \,d x\right]^{1/2}\\
&\quad\quad\quad\times\left[\mint-_{B } \left((|Du^\epsilon|+|DF|)^2\frac{(u^\epsilon-F)^2}{r^2} + \frac{(u^\epsilon-F)^4}{r^4} \right)\,d x\right]^{1/2}\\
&\quad\quad + C\frac{ [\Lambda_{H}(\fz)]^2 }{\lambda_H(\fz) }\left[\mint-_{\frac34B  }[ H (Du^\ez) ]^2 \,d x\right]^{1/2}\left[\mint-_{\frac34B  }  \frac{(u^\epsilon-F)^4}{r^4}     \,d x\right]^{1/2}\\
&\quad\quad+C\ez
\left[\int_{\frac34B  } ({\rm div}[ {D_pH}(Du^\ez)])^2 \,d x\right]^{1/2} \left[\mint-_{\frac34B  }\frac{(u^\epsilon-F)^2}{r^2}  \,d x\right]^{1/2}.
  \end{align*}
\end{lem}

\subsection{Some remarks}

In this section, we will give some motivations to find Lemmas \ref{LEM2.1}\&\ref{LEM2.2}.
Recall that in the case   $H(p)=\frac12|p|^2$,
  Theorems \ref{THM1.1}\&\ref{THM1.2} were proved by Koch et al \cite{kzz}. 
The key step therein is to prove Lemmas \ref{LEM2.7}\&\ref{LEM2.8} for smooth $e^{\frac1{2\ez}|\cdot|^2}$-harmonic function $u^\ez$---smooth solution to the $e^{\frac1{2\ez}|\cdot|^2}$-harmonic equation
$${\,\rm div\,}(e^{\frac1{2\ez}|Dv|^2})=\frac1{2\ez}e^{\frac1{2\ez}|Dv|^2}( \Delta_\fz v+\ez\Delta v)=0\mbox{ in $U$,} $$
  equivalently,
 $\Delta_\fz v+\ez\Delta v=0\mbox{ in $U$}.$
Indeed, from  Lemmas \ref{LEM2.7}\&\ref{LEM2.8} they derived
 the  convergence of $u^\ez$ in $ W^{1,t}_\loc(U)$ for all $t\ge1$ to
 the given planar $\fz$-harmonic function $u$, and hence conclude the $|Du|^\alpha\in W^{1,2}_\loc(U)$.

To prove Lemmas \ref{LEM2.7}\&\ref{LEM2.8},  Koch et al \cite{kzz} built up   the following structural identity    for $\infty$-Laplace operator $\Delta_\fz$,
\begin{equation}\label{fund1}
|D^2vDv|^2-
  \Delta v \Delta_\fz v  =(-2\det D^2v)  |Dv|^2   \quad \mbox{   $\forall v\in C^\fz $}
  \end{equation}
and also use   the well-known
divergence formula of determinant
\begin{equation}\label{fund2}
-\det D^2v={\,\rm div\,}( D^2vDv-\Delta v Dv)  \quad \mbox{   $\forall v\in C^\fz$}.
\end{equation}

%

Viewing  \eqref{fund1} and \eqref{fund2} and their proofs in \cite{kzz}, with some modifications one does get
the following  identity for $\mathscr A_H$  and divergence formula of determinant involving $H$:
\begin{equation}
\label{fund5}| D[H(Dv)]|^2-(\Delta v){\mathscr A}_{H }[v]=(-\det D^2v)|D_pH(Dv)|^2 \quad \forall\,v\in C^\fz\end{equation}
and
\begin{equation}
\label{fund6}{\rm div}\{ D[H(Dv)]-\Delta v D_pH(Dv)\}=(-\det D^2v)[H_{p_ip_i}(Dv)] \quad \forall\,v\in C^\fz.\end{equation}
However, as clarified below  the two identities are
 not enough to prove Theorems \ref{THM1.1}\&\ref{THM1.2}.

\begin{rem}\rm
Note that \eqref{fund5} and \eqref{fund6} suggest  us to consider approximating equation
\begin{equation}
\label{fund8}{\mathscr A}_{H }[v ]+\ez \Delta v =0\quad {\rm in}\quad  U.\end{equation}
Let $v^\ez$ be any smooth solution.
The key point to obtain Theorems \ref{THM1.1}\&\ref{THM1.2} is to obtain
$W^{1,2}_\loc$-estimates of  $[H(Dv^\ez)+\sz]^\alpha$, and some flatness estimate  of $v^\ez$,
that is, some analogues of  Lemma \ref{LEM2.7}\&\ref{LEM2.8}.
But  \eqref{fund5} and \eqref{fund6}   are not enough to prove these  estimates.
Indeed, to obtain $W^{1,2}_\loc$-estimates of  $[H(Dv^\ez)+\sz]^\alpha$,
thanks to  the divergence formula \eqref{fund6}
it is natural to  consider the functional
 $$F(v^\ez,\varphi)=-\int_U\langle D^2v^\ez D_pH(Dv^\ez)-\Delta v^\ez D_pH(Dv^\ez), D\varphi\rangle\,dx,\quad \forall \varphi
\in W^{1,2}_\loc (U).$$
Letting $\varphi= H(Dv^\ez) \phi^2$ with $\phi\in C^\fz_c(U)$ in this functional,
 via integration by parts and Young's inequality, from
$\int_U\Delta v^\ez  \langle  D_pH(Dv^\ez), D[H(Dv^\ez)\phi^2]\rangle   \,dx
$
 the following   arises
$$ \int_U   \langle D^2_{pp}H(Dv^\ez)D^2v^\ez Dv^\ez, D\phi^2\rangle H(Dv^\ez) \,dx+\int_U   \langle D^2_{pp}H(Dv^\ez)D^2v^\ez Dv^\ez, D [H(Dv^\ez)]\rangle \phi^2 \,dx. $$
Note that, here $D^2_{pp}H(Dv^\ez)D^2v^\ez Dv^\ez$ appears.
 On there other hand, via \eqref{fund5}, \eqref{fund6} and ${\mathscr A}_{H }[v^\ez ]=-\ez \Delta v^\ez$, one has
 \begin{align*}F(v^\ez,\varphi)& = \int_U(-\det D^2v^\ez) [H_{p_ip_i}(Dv^\ez) ] H(Dv^\ez)\phi^2\,dx\\
 & = \int_U [| D[H(Dv^\ez)]|^2+\ez (\Delta v^\ez)^2] \frac{[H_{p_ip_i}(Dv^\ez) ]}{|D_pH(Dv^\ez)|^2} H(Dv^\ez)\phi^2\,dx.
 \end{align*}
It is impossible for us to   estimate the above integrations involving $D^2_{pp}H(Dv^\ez)D^2v^\ez Dv^\ez$
 via this quantatity and some  integration over $H(Dv^\ez)|D\phi|^2$.
 Therefore, we can not get $W^{1,2}_\loc$-estimates of  $H(Dv^\ez)$.
  Similarly, we can not prove  analogues of Lemmas \ref{LEM2.7}\&\ref{LEM2.8} for $v^\ez$.
\end{rem}

For the above reasons,   some essential new ideas/observations are   needed
 to find some fundamental  structural identity for $\mathscr A_H$ and
 divergence formula of determinant matching with $\mathscr A_H$ perfectly.
There are several hints helping us.
Firstly,  in the above calculation, once we integrate by parts,
$D^2_{pp}H(Dv )$  and also $D^2_{pp} D[H(Dv )]$ appears naturally, so they should be involved.  Secondly,
 by Evans' approximation, it is natural consider  $e^{\frac1\ez H}$-harmonic equation \eqref{eq2.1x},
which suggests to consider,  instead of $\Delta$, the elliptic operator $\,{\rm{div}}[D_p H (Du^\epsilon) ]$.
 Finally,  for smooth viscosity solution $u$ to \eqref{eq1.2}, consider  the linearization   of $\mathscr A_H[u]=0$ as Evans-Smart \cite{es11a,es11b} did:
\begin{equation*}
\mathscr L^u_H  (\phi): =-\langle D^2\phi D_p H(Du ),D_p H(Du )\rangle -\langle D^2_{pp}H(Du)D[H(Du)], D\phi\rangle.
\end{equation*}
We see that   $D^2_{pp}H(Du)D[H(Du)]$ arises therein naturally.
The above observations lead  to focus on the relations among
$\mbox{$\mathscr A_H[v], D^2_{pp}H(Dv), D[H(Dv)], \,{\rm{div}}[D_p H (Dv) ]$ and  $-\det D^2v$, }$
and also consider the nonnegativity of the
integral $\int_U\mathscr L^u_H(\phi)\phi\,dx $ for  $\phi\in C^\fz_c(U)$.
 With tedious calculations we find  Lemmas \ref{LEM2.1}\&\ref{LEM2.2} finally,
which allow us to establish Theorem \ref{LEM2.3} and Lemmas \ref{LEM2.6} \ref{LEM2.7}\&\ref{LEM2.8}.

The final remark is about the approximating equations \eqref{fund8}.
\begin{rem}\rm
When $n\ge3$, to get the  everywhere differentiability of absolute minimizers for $H\in C^0(\rn)$ satisfying (H1)\&(H2), the approximation via $e^{\frac1\ez H}$-harmonic functions is not enough.
But using
the approximation given by equations \eqref{fund8} and modifying the adjoint approach by \cite{es11a} we do obtain
everywhere differentiability of absolute minimizer
 when $n\ge2$.
See  \cite{fmz} for details.
\end{rem}

\section{ Sobolev  approximation via $e^{\frac1\ez H}$-harmonic functions $u^\ez$
when    $H\in C^\fz(\rr^2)$}

In this section, we assume that $H\in C^\fz(\rn)$ satisfies  (H1')\&(H2).
Let $\Omega\subset \mathbb R^2$ be any   domain, and
 $u\in  AM_H(\Omega)$, equivalently, $u\in C^0(\Omega)$ be a   viscosity solution to \eqref{eq1.2}.
Let  $U\Subset \Omega$ be any smooth domain.
For  ${\epsilon}\in(0,1]$, we consider  the Dirichlet problem
\begin{equation}\label{eq3.1}
{\rm{div}}\left[e^{\frac1\epsilon H (Dv)} {D_pH}(Dv) \right]
=0\quad{\rm in}\ U;\quad v|_{\partial U}=u|_{\partial U}.
\end{equation}

 Section 3.1 gives the existence/uniqueness of  solutions $u^\ez\in C^\fz(U)\cap C^0(\overline U)$
 to    \eqref{eq3.1}, and also  the convergence
 $u^\ez \to u$ in $C^{0,\gz}(\overline U)$  for any $\gz\in[0,1)$ as $\ez\to0$; see  Theorem \ref{THM3.1}.

In Section 3.2, by Lemma \ref{LEM2.6} and a Moser type iteration,
we  have $\lim_{\ez\to0}\|H(Du^\ez)\|_{L^\fz(V)}\le   \|H(Du )\|_{L^\fz(U)}$; see Theorem \ref{THM3.2}.

In Section 3.3, via  Theorem \ref{THM3.2} and Lemma \ref{LEM2.7}, for any $V\Subset U$, when $\alpha\ge 1/2$  we  show that  $[H(Du^\ez)]^\alpha\in W^{1,2} (V)$  uniformly in $\ez\in(0,1]$;   when $1/2-\tau_H(\|H(Du)\|_{L^\fz(U)})<\alpha<1/2$,
   $\liminf_{\ez\to0}\|[H(Du^\ez)+\sz]^\alpha\|_{L^2(V)}$ is bounded uniformly in $\sz\in(0,1]$.
  See Theorem \ref{THM3.5} for details.

 In Section 3.4,
using Lemma \ref{LEM2.8},
 we prove  that   $u^\ez \to u$ in $W^{1,t}_\loc(U)$ for all $t\ge1$ as $\ez\to 0$;
when $\alpha\ge1/2$, $[H (Du^\epsilon )]^{\alpha}\rightarrow [H(Du)]^{\alpha}$ in $L^t_{\loc}(U)$ for all $t\in [1,\infty)$ and weakly in $W^{1,2}_{\loc}(U)$ as $\ez\to 0$;
when $\frac12-\tau_H (\|H(Du)\|_{L^\fz( U)}) < \alpha<\frac12$,
for each $\sz\in(0,1]$ we have
$[H (Du^\epsilon )+\sz]^{\alpha}\rightarrow [H(Du) ]^{\alpha}$  as $\ez\to 0$ and $\sz\to0$ in order.
See Theorem \ref{THM4.2}.

\subsection{Existence and $C^{0,\gz}(\overline U)$-convergence of  $u^\ez\in C^\fz(U)$}

We have the following result.
\begin{thm}\label{THM3.1}
\begin{enumerate}
\item[(i)] For any $\epsilon>0 $,   there exists  a unique
solution $u^\epsilon \in C^\infty(U)\cap C^0(\overline  U)$ to the equation \eqref{eq3.1} with $u^\ez|_{\partial U}=u|_{\partial U}$.
\item[(ii)]   We have
$$\sup_{\ez\in(0,1]}
\|u^\ez\|_{C^0(\overline U)}\le   \|u \|_{C^0(\partial U)} .$$
\item[(iii)] For any $t\in(2,\fz)$, we have
\begin{equation}\label{eq3.2}
\sup_{0< \ez  \le 2/t  } \|D u^\epsilon\|_{L^t(U)}\le |U|^{1/t} \left[\frac{2}{\lambda_H(\fz)}\right]^{1/2}e^{\frac12\|H (Du)\|_{L^\fz(U)}} .
\end{equation}

\item[(iv)] For any $\gz\in(0,1)$ we have
$
u^\epsilon \rightarrow u
$ in $C^{0,\gz}(\overline U)$ as $\ez\to0$.
\end{enumerate}
\end{thm}
\begin{proof}

(i) Consider the minimization problem of the functional of exponential growth
$$c_{\epsilon}:=\inf\left\{\Gamma_{\epsilon}(v):=\int_Ue^{\frac1\epsilon H (Dw)}\,d x\Big|v\in \rho_{\epsilon} \right\},$$
where $\rho_{\epsilon}$ is the set of admissible functions of the functional $\Gamma_{\epsilon}$ defined by
$$\rho_{\epsilon}:=\left\{w\in W^{1,1}(U)\Big|\int_Ue^{\frac1\epsilon H (Dw)}\,d x<\infty, w-u\in W^{1,1}_0(U)\right\},$$
Since $u\in \rho_\ez$,  we know  $\rho_{\epsilon}\neq\emptyset$. Let $\{u_m\}^\infty_{m=1}\subset \rho_{\epsilon}$ be a minimizing sequence.  Without loss of generality, we may assume that there exists $u^\epsilon \in\rho_{\epsilon}$
such that $u_m\rightarrow u^\epsilon $ uniformly on $U$, and $Du_m\rightharpoonup Du^\epsilon $ weakly in $L^t(U)$ for any $1\le t <\infty$. Since $H (p)$ is uniformly convex, by the lower semi-continuity we
have that
\begin{align*}
\Gamma_{\epsilon}(u^\epsilon )&=\int_Ue^{\frac1\epsilon H (Du^\epsilon )}\,d x=\sum^\infty_{s=1}\int_U\frac{[\frac1\epsilon H (Du^\epsilon )]^s}{s!}\,d x\\
&\le \liminf_{m\rightarrow\infty}\sum^\infty_{s=1}\int_U\frac{[\frac1\epsilon H (Du_m)]^s}{s!}\,d x=\liminf_{m\rightarrow\infty}\int_Ue^{\frac1\epsilon H (Du_m)}\,d x=\liminf_{m\rightarrow\infty}\Gamma_{\epsilon}(u_m)= c_{\epsilon}.
\end{align*}
Hence $c_{\epsilon}=\Gamma_{\epsilon}(u^\epsilon )$ and $u^\epsilon $ is a minimizer of $\Gamma_{\epsilon}$ over the set $\rho_{\epsilon}$. Direct calculations imply
that the Euler-Lagrange equation of $u^\epsilon $ is \eqref{eq3.1}. The uniqueness of $u^\epsilon $ follows from the
maximum principle that is applicable of \eqref{eq3.1}. The smoothness of $u^\epsilon $ follows from the theory of quasi-linear uniformly elliptic equations; see for example \cite{gt}.

(ii) Lemma 3.1 (ii) follows from the the maximum principle. 

(iii) Let $t\in(2,\fz)$. Noting   $  t<1/\ez$, by $x  \le e^{x  }   $ for
$x\ge 0$ and the H\"older inequality, we have
\begin{align*}
\left[\mint{-}_U [H (Du^\epsilon )]^t\,d x\right]^{1/t}&\le\left[\mint{-}_U e^{tH ( Du^\epsilon )}\,d x\right]^{1/t}\le\left[\mint{-}_U e^{\frac1\epsilon H ( Du^\epsilon )}\,d x\right]^{ \epsilon}.
\end{align*}
Since $u^\ez$  is the minimizer of  $\Gamma_{\epsilon}(v)$ with the same boundary as $u$, we have
\begin{align}\label{eq3.xx1}\mint{-}_U e^{\frac1\epsilon H (Du^\epsilon )}\,d x\le\mint{-}_U e^{\frac1\epsilon H (Du)}\,d x\le e^{\frac1\ez \|H (Du)\|_{L^\fz(U)}}.\end{align}

Considering $H(p)\ge \lambda_H(\fz)|p|^2/2$, we conclude that
\begin{equation}\label{eq3.3}
\left[\bint_U |Du^\epsilon |^{2t}\,d x\right]^{1/2t}\le \left[\frac{2}{\lambda_H(\fz)}\right]^{1/2}e^{\frac12\|H (Du)\|_{L^\fz(U)}},
\end{equation}
by multiplying $|U|^{1/2t}$ in both sides, which gives \eqref{eq3.2}.

(iv)   Let $t=2/(1-\gamma)$.
For   $0<\ez< 2/t=1-\gamma$, we have   $  u^\ez \in W^{1,t}(U)$ uniformly in
$\ez\in(0,1-\gz)$. The Sobolev imbedding yields that $  u^\ez \in C^{0,\gz}(\overline U)$ uniformly in
$\ez\in(0,1-\gz)$.
 By Arzela-Ascolli' theorem, we know that $u^\ez$, up to some subsequence,
 converges   to some function $\hat u$ in $C^0(\overline U)$ and hence in $C^{0,\gz}(\overline U)$.
  Thanks to the viscosity theory (see \cite{CIL}), $\hat u$ is also a viscosity solution to \eqref{eq1.2}.
  Noting   $\hat u=u$ in $\partial U$, by Lemmas \ref{LEM7.9}\&\ref{LEM7.10}, we have $\hat u=u$ in $U$.
  This completes the proof of Theorem \ref{THM3.1}.
\end{proof}

\subsection{Uniform   $L^{ \fz}_\loc(U)$-estimates of $H(Du^\ez)$}
Moreover, we establish an $L^\fz_\loc(U)$-estimates for $H(Du^\ez)$ which is uniformly in $\ez\in(0,1]$.

 \begin{thm}\label{THM3.2} For any smooth domain $V\Subset U$, we have
$$  \|H(Du^\epsilon) \|_{L^\fz(V)}\le C(H,U,V)\ez^{1/2} +[1+C(H,U,V)\ez]\|H(Du)\|_{L^\fz(U)} \quad \forall \ez\in(0,1] .$$
 \end{thm}

 Recall that
 when $H(x,p)$ is a periodic Hamiltonian in $x$, that is, $H(x +\zz,p)=H(p)$ for all $x\in\rn$,
for smooth $e^{\frac1\ez H(x,\cdot)}$-harmonic function $v^\ez$,
using a Moser type iteration argument for the normalization of $e^{\frac1\ez H(x,Dv^\ez)}$,
 Evans  \cite{e03} obtained the $W^{1,\fz}([0,1]^n)$-estimates of $ H(x,Du^\ez)$ uniformly in $\ez\in(0,1]$.

To prove  Theorem \ref{THM3.2}, for $\ez\in(0,1)$ we write $$\sz^\ez= \frac{e^{\frac1\ez H(Du^\ez)}}{\int_U e^{\frac1\ez H(Du^\ez)}\,dx}.$$
But note that the iteration of $\sz^\ez$ itself  does not give Theorem \ref{THM3.2} since we donot have some
nice  $L^t$-estimates of $\sz^\ez$.
Instead,  we consider  $\sz^\ez\phi^{k}$ with $k\ge1$. In this section, for any given $V\Subset U$,  we always fix a $\phi=\phi_V\in C^\infty_c(U)$ satisfying
\begin{align}\label{eq3.4} \phi=1\,{\rm in}\ V,\
 0\le\phi\le 1 \,{\rm in}\ U,  \ |D\phi|\le  \frac C{{\rm dist}(V,\partial U)}\
  {\rm and} \  |D^2\phi|\le  \frac C {[{\rm dist}(V,\partial U)]^2} \ {\rm in}\ U.
 \end{align}
The following auxiliary Lemmas \ref{LEM3.4}\&\ref{LEM3.3} allow us to prove
  Theorem \ref{THM3.2} by   borrowing  some ideas from Evans \cite{e03}---a Moser type iteration.

\begin{lem}\label{LEM3.3}
Let $V\Subset U$ and $\phi$ be as in \eqref{eq3.4}.
For any  $\beta>1$  we have
$$\|\sigma ^{\epsilon}\phi^4\|_{L^\beta(U)} \le C(H,U,V,\beta) \frac1 {\epsilon}[1+\|H(Du)\|_{L^\fz(U)}] .$$
\end{lem}

\begin{lem} \label{LEM3.4}Let $V\Subset U$ and $\phi$ be as in \eqref{eq3.4}.
For any $t>1$, $m\ge 2$ and $\theta>1$, we have
\begin{equation}
\|\sigma ^{\epsilon}\phi^{2m}\|_{L^{t\theta^2}(U)}\le
\left[C(H,U,V,\theta) \frac1 {\epsilon}[1+\|H(Du)\|_{L^\fz(U)}] \right]^{1/t} (m t)^{2/t}
\|\sigma^\epsilon\phi^m\|_{L^{t\theta}(U)}.
\end{equation}
\end{lem}

The proofs of Lemmas \ref{LEM3.4}\&\ref{LEM3.3} rely on Lemma \ref{LEM2.6} and Sobolev's imbedding,
 and  are postponed  to Section 7.2.   Now we are ready to prove Theorem \ref{THM3.2}.

\begin{proof}[Proof of Theorem \ref{THM3.2}]
Let $V\Subset U$ and $\phi$ be as in \eqref{eq3.4}.
Let $\theta=2$,  $t=\theta^{k-1}$ and $m=\theta^{k}$  in Lemma \ref{LEM3.4}.
Write  $q_{k+1}:=\theta^k$, and noting $q_{k+1}=t\theta^2$,  $q_{k}:=t\theta$ and   $m^2t^2\le 2^{6k}$,  we obtain
\begin{equation*}
\begin{split}
\|  \sigma ^{\epsilon}\phi^{q_{k+1}}\|_{L^{q_{k+1}}(U)}
&\le
\left\{C(H,U,V) \frac1\ez  [1+\|H(Du)\|_{L^\fz(U)}] \right\}^{2^{- k+1}}  2^{6k2^{-k+1}}\|  \sigma ^{\epsilon}\phi^{q_{k }}\|_{L^{q_{k }}(U)}.
\end{split}
\end{equation*}
Therefore,
\begin{align*}
 \|  \sigma ^{\epsilon}\phi^{q_{k+1}}\|_{L^{q_{k+1}}(U)}
&\le\left \{C(H,U,V) \frac1{\epsilon}[1+  \|H(Du)\|_{L^\fz(U)}]  \right\}^{\sum_{j=2}^k 2^{-j+1}}   {2}^{\sum_{j=2}^k{ 6j }2^{-j+1} }
  \|\sigma^\epsilon\phi^{4}\|_{L^4(U)}\\
 &\le C(H,U,V) \frac1{\epsilon}[1+  \|H(Du)\|_{L^\fz(U)}]
 \|\sigma^\epsilon\phi^{4} \|_{L^4(U)}.
\end{align*}
 Letting  $k\to\infty$, noting $\phi=1$ in $V$ and applying Lemma \ref{LEM3.3} we have
\begin{align*}
 \|\sigma^\epsilon\|_{L^\infty(V)}
&\le C(H,U, V) \frac1{\epsilon}[1+  \|H(Du)\|_{L^\fz(U)}]
 \|\sigma^\epsilon\phi^{4} \|_{L^4(U)}\le C(H,U,V)\frac1{\epsilon^2}[1+  \|H(Du)\|_{L^\fz(U)}]^2 .
\end{align*}
This, together with \eqref{eq3.xx1},
 yields
\begin{align*}
 \|e^{\frac1\ez H(Du^\epsilon )} \|_{L^\fz(V)}
&\le C(H,U,V)\frac1{\epsilon^2}[1+  \|H(Du)\|_{L^\fz(U)}]^2 e^{\frac1 \ez\|H(Du)\|_{L^\infty(U)}}.
\end{align*}
Hence,
\begin{align*}
 \|H(Du^\epsilon ) \|_{L^\fz(V)}
 &\le    C(H,U,V) \ez\left\{ \ln\frac2{\epsilon }+\ln [1+  \|H(Du)\|_{L^\fz(U)}]\right \}+  \ez \ln e^{\frac1 \ez\|H(Du)\|_{L^\infty(U)}},
\end{align*}
from which we conclude
\begin{align*}  \|H(Du^\epsilon ) \|_{L^\fz(V)}
&\le  C(H,U,V)\{\ez^{1/2}+\ez \ln [1+  \|H(Du)\|_{L^\fz(U)]   }\}+  \|H(Du)\|_{L^\infty(U)}\\
&\le  C(H,U,V) \ez^{1/2} +[1+ C(H,U,V)\ez]  \|H(Du)\|_{L^\infty(U)}.
\end{align*}
 This completes the proof of Theorem \ref{THM3.2}.
\end{proof}
\begin{rem}\label{REM3.5}\rm
(i) Note that Theorem \ref{THM3.2} gives
   $\lim_{\ez\to0}\| H(Du^\ez) \|_{L^\fz(V)}\le  \|H(Du)\|_{L^\fz(U)}$ for all $V\Subset U$,
 which is necessary for us to get the range $\alpha>1/2-\tau_H(\|H(Du)\|_{L^\fz(U)} )$ in  Theorem \ref{THM1.1}
 and also to deduce   quantities $\lambda_H(\|H(Du)\|_{L^\fz(U)} )$,  $\Lambda_H(\|H(Du)\|_{L^\fz(U)} )$ and
 $\tau_H(\|H(Du)\|_{L^\fz(U)} )$ in
 \eqref{e1.x1} from Lemma \ref{LEM2.7}.   See   Section 4 for details.

(ii)
Recall that Evans-Smart \cite[Theorem 2.1]{es11b}
 established a uniform gradient estimate  for $e^{\frac1\ez|\cdot|}$-harmonic functions by
 via  an  approach based on the maximal principle, and linearization of
 the equation $\Delta_\fz u^\ez+\ez\Delta u^\ez=0$.
But, via this approach,
 it is impossible to prove Theorem \ref{THM3.2} or even some
 $L^\fz_\loc(U)$-estimates of $H(Du^\ez)$ depending on $\lambda_H$ and $\Lambda_H$
 under our assumptions (H1')\&(H2).
Indeed,  the 3-order derivative $D^3_{ppp}H (Du^\ez) $ of $H$  occurs naturally in
the linearization (coming from the differentiation) of
$\mathscr A_H[u^\ez]+\ez {\,\rm div\,}[D_p H(Du^\ez) ] =0$.
But we do not have any assumption on $D^3_{ppp}H$.
We also note that to prove Theorem \ref{THM1.1} for
 $H\in C^0(\rr^2)$ satisfying  (H1')\&(H2),  we will consider the smooth  approximation $\{H^\dz\}_{\dz\in (0,1]}$
to $H$ as  in \eqref{eq7.1}, which are only known  satisfying (H1')\&(H2) uniformly in $\dz\in(0,1]$.
\end{rem}

 \subsection{Uniform $W^{1,2}_\loc(U)$-estimates of $[H(Du^\ez)+\sz]^\alpha$ }

We have the following uniform Sobolev estimates.
  \begin{thm} \label{THM3.5} Let $V\Subset W\Subset U$.
\begin{enumerate}
\item[(i)] For  $\alpha\ge 1/2$,  we have
 \begin{equation}\label{eq3.6}
 \sup_{\ez\in(0,1]} \left\| |D[H (Du^\epsilon )]^{\alpha}|+\ez  \frac{|{\,\rm div\,}( {D_pH}(Du^\ez))|}{[H(Du^\ez)]^{1-\alpha } }\right \|_{L^2(V)}<\fz
 \end{equation}
  and
  \begin{align}\label{eq3.7}
\liminf_{\ez\to 0} \|  D[H (Du^\epsilon )]^{\alpha}  \|_{L^2(V)}
&\le  \frac{ C\alpha^2(\alpha+1)}{[\alpha+\tau_H(\|H(Du)\|_{L^\fz( W)})-\frac12]^2}
\left[\frac{ \Lambda_H( \|H(Du) \|_{L^\fz(W)})}{\lambda_H( \|H(Du) \|_{L^\fz(W)})}\right]^2  \nonumber\\
 &\quad\quad \times\frac1{[\dist(V,W^\complement)]^{ 2}}\liminf_{{\epsilon}\rightarrow0}  \int_{W}    [H(Du^\ez)]^{2\alpha }\,dx.
\end{align}

 \item[(ii)] Let  $ 1/2-\tau_H (\|H(Du)\|_{L^\fz( U)}) < \alpha< 1/2$.
  For each $\sz\in(0,1]$, there exists
 $\ez_0=\ez(\sz,H,u,U)\in(0,1]$ such that
 \begin{equation}\label{eq3.8}\sup_{\ez\in(0,\ez_0]}\left\||D[H (Du^\epsilon )+\sz]^{\alpha}|+\ez^{1/2}   \frac{|{\,\rm div\,}[ {D_pH}(Du^\ez)]|} {[H(Du^\ez)+\sz]^{1-\alpha} }\right\|_{L^2(V)}<\fz.\end{equation}
Moreover,
 \begin{equation}\label{eq3.9}\sup_{\sz\in(0,1]}\lim_{\ez\to0}\|D[H (Du^\epsilon )+\sz]^{\alpha}\|_{L^2(V)}<\fz, \end{equation}
and
\begin{align}\label{eq3.10}
&\liminf_{\sz\to0}\liminf_{\ez\to0} \|  D[H (Du^\epsilon )+\sz]^{\alpha}  \|_{L^2(V)}\nonumber\\
&\quad\le  \frac{ C\alpha^2(\alpha+1)}{[\alpha+\tau_H(\|H(Du)\|_{L^\fz( W)})-\frac12]^2}\left[\frac{ \Lambda_H( \|H(Du) \|_{L^\fz(W)})}{\lambda_H( \|H(Du) \|_{L^\fz(W)})}\right]^2 \nonumber\\
 &\quad\quad \times \frac1{[\dist(V,W^\complement)]^{ 2}} \liminf_{\sz\to0}\liminf_{{\epsilon}\rightarrow0}   \int_{W}   [H(Du^\ez)+\sz]^{2\alpha }\,dx.
\end{align}
\end{enumerate}
 \end{thm}

 \begin{proof}
(i) Let $V\Subset W\Subset U$. By Lemma \ref{LEM2.7} with a suitable cut-off function $\phi$  we have
 \begin{align*}
     &  \int_{V}  |D[H(Du^\ez) ]^\alpha|^2 \,dx  +\frac{\ez \alpha^2}{\lambda_H(\fz)} \int_{V}
  \frac{({\rm div}[ {D_pH}(Du^\ez) ])^2 }{    [H(Du^\ez)]^{2-2\alpha}} \,dx\\
    &\quad\le  \frac{ C\alpha^2(\alpha+1)}{[\alpha+\tau_H(\|H(Du^\ez)\|_{L^\fz( W)})-\frac12]^2} \left[\frac{ \Lambda_H( \|H(Du^\ez)\|_{L^\fz(W)})}{\lambda_H( \|H(Du^\ez) \|_{L^\fz(W)})}\right]^2 \frac1{[\dist(V,W^\complement)]^2}
 \int _W     [ H(Du^\ez) ]^{2\alpha  }  \,dx\\
   &\quad\quad + C\ez\frac{\alpha^2|2\alpha-1|^2}{ [\alpha+\tau_H(\fz) -\frac12]^2} \left[\frac{ \Lambda_H(\fz)}{\lambda_H(\fz)}\right]^2 \frac1{[\dist(V,W^\complement)]^2}
 \int _{ W}     [ H(Du^\ez) ]^{2\alpha-1 } \,dx\\
 &=J_1 +\ez J_2 .
   \end{align*}

 By Theorem \ref{THM3.2},
and $\alpha\ge1/2$ we know that
\begin{align*}&\sup_{\ez\in[0,1]}[ J_1+J_2]
 \le  \frac{C(H,  W,U)\alpha^4}{ [\alpha+\tau_H(\fz) -\frac12]^2} \left[\frac{ \Lambda_H(\fz)}{\lambda_H(\fz)}\right]^2   \frac{|W|}{[\dist(V,W^\complement)]^2} [ 1+\|H(Du) \|_{L^\fz(U)}]^{2\alpha}<\fz,
\end{align*}
which  together with $\lambda_H(\|H(Du^\ez)\|_{L^\fz(V)})\ge \lambda_H(\fz)>0$ gives \eqref{eq3.6}.

 Moreover,  the boundedness of $J_2$ implies that $\ez J_2\to0$  as $\ez\to0$.

Noting $\lim_{\ez\to0}\|H(Du^\ez)\|_{L^\fz(W)}\le \|H(Du)\|_{L^\fz(U)}$ as given in Theorem \ref{THM3.2},
and using the right-continuity of $\lambda_H, \Lambda_H$ and $\tau_H$, we have
\begin{equation}\label{eq3.15}
\liminf_{\ez\to0}\lambda_H(\|H(Du^\ez)\|_{L^\fz(W)})\ge \lambda_H(   \|H(Du) \|_{L^\fz(U)}   ),
\end{equation}
 \begin{equation}\label{eq3.16}\limsup_{\ez\to0}\Lambda_H(\|H(Du^\ez)\|_{L^\fz(W)})\le \Lambda_H(   \|H(Du) \|_{L^\fz(U)}   ),\end{equation}
   and
\begin{equation}\label{eq3.17}
\liminf_{\ez\to0}\tau_H(\|H(Du^\ez)\|_{L^\fz(W)})\ge \tau_H(   \|H(Du) \|_{L^\fz(U)}   ).\end{equation}
By these inequalities, we obtain
 \begin{align*}
   \liminf_{\ez\to0}\int_{V}  |D[H(Du^\ez) ]^\alpha|^2 \,dx
 &\le \liminf_{\ez\to0}J_1 \\
 &\le  \frac{ C\alpha^2(\alpha+1)}
 {[\alpha+\tau_H(\|H(Du )\|_{L^\fz( U)})-\frac12]^2} \left[\frac{ \Lambda_H( \|H(Du )\|_{L^\fz(U)})}{\lambda_H( \|H(Du ) \|_{L^\fz(U)})}\right]^2\\
 &\quad\times \frac1{[\dist(V,W^\complement)]^2}
\liminf_{\ez\to0} \int _W     [ H(Du^\ez) ]^{2\alpha  }  \,dx
   \end{align*}
 that is, \eqref{eq3.7} holds.

{\it Proof of (ii).}
Let $V\Subset W\Subset U$.
 For any $\sz\in(0,1]$, by Lemma \ref{LEM2.7} (ii)  we have
  \begin{align*}
     &  \int_{V}  |D[H(Du^\ez) ]^\alpha|^2 \,dx  +\frac{\ez \alpha^2}{\lambda_H(\fz)} \int_{V}
  ({\rm div}[ {D_pH}(Du^\ez) ])^2     [H(Du^\ez)]^{2\alpha-2} \,dx\\
    &\ \le\frac{ C\alpha^2(\alpha+1)}{[\alpha+\tau_H(\|H(Du^\ez)\|_{L^\fz( W)})-\frac12]^2}
 \left[\frac{ \Lambda_H( \|H(Du^\ez)\|_{L^\fz(W)})}{\lambda_H( \|H(Du^\ez) \|_{L^\fz(W)})}\right]^2 \frac1{[\dist(V,W^\complement)]^2}\int_W    [ H(Du^\ez) +\sz]^{2\alpha  }  \,dx\\
    &\quad\quad+\frac{C\ez\alpha^2|2\alpha-1|^2}{ [\alpha+\tau_H(\|H(Du^\ez)\|_{L^\fz(W)}) -\frac12]^2}\left[\frac{ \Lambda_H( \fz)}{\lambda_H(\fz)}\right]^2\frac1{[\dist(V,W^\complement)]^2}\int_W   [ H(Du^\ez) +\sz]^{2\alpha-1 } \,dx\\
 &\quad=J_1(\sz )+ \ez J_2(\sz).
   \end{align*}

    By Theorem \ref{THM3.2}, and by the decrease and right-continuity of $\tau_H$,
    there exists $\ez_0\in(0,1]$ depending on $\sz,W,U,H$ such that for all $\ez\in(0,\ez_0]$  we have
    \begin{align*}
    &\alpha+\tau_H(  \|H(Du^\ez) \|_{L^\fz(W)} )-\frac12\\
    & \quad
    \ge \alpha+\tau_H([1+C(H,W,U) \ez_0] \|H(Du) \|_{L^\fz(U)}+C(H,W,U) \ez_0^{1/2})-\frac12\\
    &\quad\ge   \frac12[
    \alpha+\tau_H(  \|H(Du) \|_{L^\fz(U)})-\frac12].
    \end{align*}
    By this and  Theorem \ref{THM3.2} again    we have
   \begin{align*}   \sup_{\ez\in(0,\ez_0]}[J_1(\sz)+J_2(\sz)] &\le C(H,V,W,U)
   \frac{ \alpha^4   }{ [\alpha+\tau_H(\|H(Du)\|_{L^\fz( U)}) -\frac12]^2 }
    \left[\frac{ \Lambda_H( \fz)}{\lambda_H( \fz)}\right]^2\\
    &\quad\times   \frac1{\sz} \frac{|W|}{[\dist(V,W^\complement)]^2} [ \|H(Du)\|_{L^\fz(U)}+1]^{2\alpha }<\fz,
  \end{align*}
   which   implies  \eqref{eq3.8}.

Moreover,   $\sup_{\ez\in(0,\ez_0]}J_2(\sz)<\fz$  implies that  $\ez J_2(\sz) \to0$ as $\ez\to0$.

By \eqref{eq3.15}, \eqref{eq3.16} and \eqref{eq3.17} we have
\begin{align*} \liminf_{\ez\to0} J_1(\sz) &\le
\frac{C \alpha^2 (\alpha+1) }{ [\alpha+\tau_H(\|H(Du)\|_{L^\fz( U)}) -\frac12]^2 }
    \left[\frac{ \Lambda_H( \|H(Du)\|_{L^\fz( U)})}{\lambda_H( \|H(Du)\|_{L^\fz( U)})}\right]^2\\
    &\quad\quad\times  \frac1{[\dist(V,W^\complement)]^2}
  \liminf_{\ez\to0}\int_W[H(Du^\ez)+\sz]^{2\alpha}\,dx.
  \end{align*}
 By Theorem \ref{THM3.2} again, we have
  \begin{align*}
    \sup_{\sz\in(0,1]}\liminf_{\ez\to0}\|D[H(Du^\ez)+\sz]^\alpha\|_{L^2(V)}^2
    &\le  \sup_{\sz\in(0,1]}\liminf_{\ez\to0} J_1(\sz)<\fz,
   \end{align*}
which gives  \eqref{eq3.9}, and moreover,
          \begin{align*}
    \liminf_{\sz\to0}\liminf_{\ez\to0}\|D[H(Du^\ez)+\sz]^\alpha\|_{L^2(V)}^2
    &\le  \liminf_{\sz\to0}\liminf_{\ez\to0} J_1(\sz),
   \end{align*}
 which gives \eqref{eq3.10}  as desired. This completes the proof of Theorem \ref{THM3.5}.
  \end{proof}

\subsection{$L^t_\loc(U)$-convergence of $[H(Du^\ez)+\sz]^\alpha$ and $W^{1,t}_\loc(U)$-convergence of $u^\ez$}

The following convergence are crucial to prove Theorem \ref{THM1.1}\&\ref{THM1.2}.

 \begin{thm} \label{THM4.2}
\begin{enumerate}
\item[(i)] For $\alpha\ge 1/2$, we have $[H (Du^\epsilon )]^{\alpha}\rightarrow [H(Du)]^{\alpha}$ in $L^t_{\loc}(U)$ for all $t\in [1,\infty)$ and weakly in $W^{1,2}_{\loc}(U)$ as $\ez\to0$.

\item[(ii)] We have $u^\ez\to u$ in $W^{1,t}_\loc (U)$ and $D_p H(Du^\ez)\rightarrow D_p H(Du)$ in $L^t_\loc(U)$ for all $t\ge1$ as $\ez\to0$.

 \item[(iii)] For   $ 1/2-\tau_H (\|H(Du)\|_{L^\fz( U)}) < \alpha< 1/2$,
we have $[H (Du^\epsilon )+\sz]^{\alpha}\rightarrow [H(Du)+\sz]^{\alpha}$  as $\ez\to 0$,
and
$[H (Du  )+\sz]^{\alpha}\rightarrow [H(Du)]^{\alpha}$  as $\sz\to 0$,   in $L^t_\loc (U)$ for all $t\in [1,\infty)$ and weakly in $W^{1,2}_\loc (U)$.
\end{enumerate}
 \end{thm}

 \begin{proof}[Proof of Theorem \ref{THM4.2}. ]

\noindent {\it Proof of   (i).} By Theorem \ref{THM3.5}, $[H (Du^\epsilon )]^{\alpha}\in  W^{1,2}_\loc (U)$ uniformly in $\ez\in(0,1]$.
By the weak compactness of $W^{1,2}_{\loc}(U)$,
 we know that, up to some subsequence,
 $[H (Du^\epsilon )]^{\alpha}\rightarrow f^{(\alpha)}$ in $ L^t_\loc (U)$ with $ t\in[1,\infty)$
and weakly in $W^{1,2}_\loc (U)$ as $\ez\to0$.
Thus,  the proof of  Theorem \ref{THM4.2} (i) is reduced to proving
 $f^{(\alpha)}
=[H(Du)]^{\alpha}$  almost everywhere in $U$.
Moreover, it suffices to
show that $    H(Du)=f^{(1)}$ almost everywhere in $U$.
Indeed, if this holds, then
 $$f^{(\alpha)}=\lim_{\ez\to0} [H(Du^\ez)]^{\alpha}= [\lim_{\ez\to0} H(Du^\ez)]^{\alpha}= [f^{(1)}]^{\alpha}
=[H(Du)]^{\alpha}$$ almost everywhere in $U$ as desired.

Below we show that $f^{(1)}= H(Du)$ almost everywhere in $U$.
Assume that $u$ is
differentiable at $\bar x\in U$,  and also assume that $\bar x$ is a Lebesgue point for $f^{(1)}$ and $H(Du)$.
 The set of such points are dense in $U$.
For any $\zeta\in (0,1)$,  there exists $r_{\zeta,\bar x}\in (0,{\rm dist}(\bar x, \partial U)/8)$  such that for any $r\in (0,r_{\zeta,\bar x})$, we have
$$\sup_{B(\bar x,2r)}\frac{|u(x)-u(\bar x)-\langle Du(\bar x), x-\bar x\rangle|}{r}\le\zeta.$$
By the Theorem \ref{THM3.1} (iii), for arbitrary $r\in(0,r_{\zeta,\bar x})$, there exists ${\epsilon}_{\zeta,\bar x,r}\in(0,1)$ such that for all
${\epsilon}\in(0,{\epsilon}_{\zeta,\bar x,r})$, we have
$$\sup_{B(\bar x,2r)}\frac{|u^\epsilon (x)-u^\epsilon (\bar x)-\langle Du(\bar x), x-\bar x\rangle|}{r}\le2\zeta.$$
Letting $F(x)=u^\epsilon (\bar x)+\langle Du(\bar x),x-\bar x\rangle$ in Lemma \ref{LEM2.8} we have
\begin{align*}
  &\mint-_{B(\bar x,r) }\langle D_pH(Du^\ez),Du^\epsilon-DF\rangle ^2 \,d x\\
  &\quad\le C \Lambda_H(\fz)
  \left[\int_{ B(\bar x,2r)}|D[H(Du^\ez)]|^2  \,d x\right]^{1/2}\\
&\quad\quad\times\left[\mint-_{B(\bar x,2r) } \left((|Du^\epsilon|+|DF|)^2\frac{(u^\epsilon-F)^2}{r^2} + \frac{(u^\epsilon-F)^4}{r^4} \right)\,d x\right]^{1/2}\\
&\quad +C \frac{ [\Lambda_{H}(\fz)]^2 }{\lambda_H(\fz) }\left[\mint-_{B(\bar x,2r) }[ H (Du^\ez) ]^2 \,d x\right]^{1/2}\left[\mint-_{B(\bar x,2r)  }  \frac{(u^\epsilon-F)^4}{r^4}     \,d x\right]^{1/2}\\
&\quad+
C\left[\int_{B(\bar x,2r)  }\ez^2({\rm div}[ {D_pH}(Du^\ez)])^2 \,d x\right]^{1/2} \left[\mint-_{B(\bar x,2r)  }\frac{(u^\epsilon-F)^2}{r^2}  \,d x\right]^{1/2}\\
& \quad :=J_1+J_2+J_3.
  \end{align*}
    By Theorem \ref{THM3.2}, we have
\begin{align*}
  J_2&\le C(H,U,\bar x) \frac{ [\Lambda_{H}(\fz)]^2 }{\lambda_H(\fz) }[1+\|H(Du)\|_{L^\fz(U)}]\zeta^2:=L_1\zeta^2,
  \end{align*}
  where $L_1$ is independent of $r$ and $\zeta$.
    We have
  \begin{align*}
  J_3\le C\sup_{\ez\in(0,1)}
  \left\|\ez{\rm div}[ {D_pH}(Du^\ez)] \right\|_{L^2(B(\bar x,\frac12 \dist({\bar x},\partial U))}\zeta:=L_3\zeta,
  \end{align*}
  where $L_3<\fz$  by \eqref{eq3.6} and $L_3$ is independent of $r$ and $\zeta$.
 By  Lemma \ref{LEM7.1} and $\lambda_H\ge\lambda_H(\fz)$, and Theorem \ref{THM3.2}   we have
\begin{align}\label{eq4.2}
L_0&:=\sup_{\ez\in(0,1]}\|Du^\ez\| _{L^\fz(B(\bar x,\frac12 \dist({\bar x},\partial U)))}
+ |Du(\bar x) | \\
& \le \frac {\sqrt 2} {[\lambda_H(\fz)]^{1/2}}
\|H(Du^\ez)\|^{1/2}_{L^\fz(B(\bar x,\frac12 \dist({\bar x},\partial U)))}+ |Du(\bar x) |\nonumber\\
 & \le \frac {\sqrt 2} {[\lambda_H(\fz)]^{1/2}}C(H,U,  \bar x )[1 +
\|H(Du )\|^{1/2}_{L^\fz(U))}]+ |Du(\bar x) |  <\fz\nonumber,
\end{align}
where $L_0$ is independent of $r$ and $\zeta$.
We then have
\begin{align*}
  J_1&\le C\sup_{\ez\in(0,1)}
  \|D[H(Du^\ez)]\|_{L^2(B(\bar x,\frac12 \dist({\bar x},\partial U)))} \Lambda_{H}(\fz)  [L_0+1]\zeta  :=L_2\zeta,
  \end{align*}
  where $L_2<\fz$ by \eqref{eq3.6} and $L_2$ is independent of $r$ and $\zeta$.
Combining all estimates together, we obtain
\begin{align}\label{eq4.3}
&\mint{-}_{B(\bar x,r)}(\langle D_pH (Du^\epsilon ),Du^\epsilon -Du(\bar x)\rangle)^2\,d x\le [L_1+L_2+L_3]\zeta  .
  \end{align}

Moreover, by $\lambda_H\ge \lambda_H(\fz) $ we have
\begin{equation*}
\langle D_p H (Du^\epsilon (x))-D_p H (Du(\bar x)),Du^\epsilon (x)-Du(\bar x)\rangle\ge \lambda_H(\fz)|Du^\epsilon (x)-Du(\bar x)|^2\quad\forall x\in U.
\end{equation*}
Thus
\begin{align}\label{eq4.4}
&\lambda_H(\fz)\mint{-}_{B(\bar x,r)}|Du^\epsilon (x)-Du(\bar x)|^2\,d x\\
&\quad\le  \mint{-}_{B(\bar x,r)} \langle D_p H (Du^\epsilon ),Du^\epsilon -Du(\bar x)\rangle  \,d x-  \mint{-}_{B(\bar x,r)}\langle D_p H (Du (\bar x) ),Du^\epsilon -Du(\bar x)\rangle\,d x\nonumber\\
&\quad\le [(L_1+L_2+L_3)\zeta ]^{1/2}-   \mint{-}_{B(\bar x,r)}\langle D_p H (Du (\bar x) ),Du^\epsilon -Du(\bar x)\rangle\,d x\nonumber.
  \end{align}
Note that for any  $x\in B(\bar x,r) $ we have
$$
 |H (Du^\epsilon (x))-H (Du(\bar x))| =|D_p H(\theta Du^\epsilon (x)+  (1-\theta) Du(\bar x))||Du^\epsilon (x)-Du(\bar x)|
$$
for some $\theta\in [0,1]$.
Since
$$| \theta Du^\epsilon (x)+  (1-\theta) Du(\bar x)|\le L_0,$$
we have
$$|D_p H(\theta Du^\epsilon (x)+  (1-\theta) Du(\bar x))|\le   \Lambda_H(\fz) L_0 .$$
Hence,
\begin{equation}\label{eq4.5}
|H (Du^\epsilon (x))-H (Du(\bar x))|
\le  \Lambda_H(\fz)  L_0  |Du^\epsilon (x)-Du(\bar x)|.
\end{equation}
From \eqref{eq4.4} and \eqref{eq4.5} we have
\begin{align*}
&\frac{\lambda_H(\fz)}{[\Lambda_H(\fz) L_0]^2}\mint{-}_{B(\bar x,r)}|  H (Du^\epsilon ) - H(Du(\bar x))|^2\,d x \\
&\quad\le [(L_1+L_2+L_3)\zeta]^{1/2}-  \mint{-}_{B(\bar x,r)}\langle D_p H (Du (\bar x) ),Du^\epsilon -Du(\bar x)\rangle\,d x
\end{align*}
Since  $H (Du^\epsilon )\rightarrow f^{(1)} $ in $L^2 (B(\bar x,r))$ as $\ez\to0$,
we obtain
\begin{align*}
&\frac{\lambda_H(\fz)}{[\Lambda_H(\fz)L_0]^2}\mint{-}_{B(\bar x,r)}|f^{(1)}- H(Du(\bar x))|^2\,d x\\
&\quad\le [(L_1+L_2+L_3)\zeta]^{1/2}- \liminf_{\ez\to0} \mint{-}_{B(\bar x,r)}\langle D_p H (Du (\bar x) ),Du^\epsilon -Du(\bar x)\rangle\,d x.
\end{align*}
Noting $u^\ez \to u$ uniformly and hence $Du^\ez\rightharpoonup Du$ weakly in $L^2_\loc(U)$, we obtain
\begin{align*}
&\frac{\lambda_H(\fz)}{[\Lambda_H(\fz)L_0]^2}\mint{-}_{B(\bar x,r)}|f^{(1)}- H(Du(\bar x))|^2\,d x \\
&\quad\le  [(L_1+L_2+L_3)\zeta]^{1/2}-  \mint{-}_{B(\bar x,r)}\langle D_pH (Du (\bar x) ),Du  -Du(\bar x)\rangle\,d x.
\end{align*}
Since $L_1,L_2,L_3$ are independent of $r$ and $\zeta$, and $\bar x$ is a Lebesgue point of $Du$,
letting $r\to0$ and $\zeta\to0$ in order, we  conclude $f^{(1)}(\bar x) = H(Du(\bar x))$ as desired.

\medskip
\noindent {\it Proof of   (ii).}
Note that
$$H(Du^\ez)-H(Du)\ge \langle D_p H(Du),Du^\ez-Du\rangle+\frac{\lambda_H(\fz)}2|Du^\ez-Du|^2
$$
 almost everywhere in $U$.
By $H(Du^\ez)\rightarrow H(Du)$ in $L^1_{\loc}(U)$ and using $Du^\ez\rightharpoonup Du$
weakly in $L^2_{\loc}(U)$ as given in Theorem \ref{THM4.2} (i), for any $V\Subset U$ we obtain
\begin{align*}
&\liminf_{\ez\rightarrow 0}\frac{\lambda_H(\fz)}2\int_V|Du^\ez-Du|^2\,dx\\
& \le \liminf_{\ez\rightarrow 0}\int_V\langle D_p H(Du),Du-Du^\ez\rangle\,dx+ \liminf_{\ez\rightarrow 0}\int_V|H(Du^\ez)-H(Du)|\,dx =0,
\end{align*}
that is,  $Du^\ez\to Du$ in $L^2_\loc(U)$ and hence almost everywhere in $U$ as $\ez\to0$. Since $Du^\ez\in L^\fz_\loc(U)$ locally uniformly in $\ez\in(0,1]$, we further have $Du^\ez\to Du$ in $L^t_\loc(U)$ for all $t\ge 1$ as $\ez\to0$.

Observe that
\begin{align*}|D_p H(Du^\ez)-D_p H(Du)|&\le|[D^2_{pp} H(\theta Du^\ez+(1-\theta)Du)](Du^\ez-Du)|
 \le \Lambda_H(\fz)|Du^\ez-Du |
  \end{align*}
  for some $\theta\in[0,1]$.  We further have
   $D_p H(Du^\ez)\rightarrow D_p H(Du)$ in $L^t_{\loc}(U)$ for all
 $t\ge 1$.

 \medskip

\noindent {\it Proof of   (iii).}
For any $\sz\in(0,1]$,
by Theorem \ref{THM3.5} (ii)  we know that
 $[H (Du^\epsilon )+\sz]^{\alpha}\in W^{1,2}_\loc(U)$ uniformly in $\ez\in(0,\ez_0]$.
By the weak compactness of $W^{1,2}_\loc(U)$,  we know that, up to some subsequence,
 $[H (Du^\epsilon )+\sz]^{\alpha}\rightarrow f_\sz^{(\alpha)}$ in $ L^t_\loc (U)$ with $ t\in[1,\infty)$
and weakly in $W^{1,2}_\loc(U)$ as $\ez\to0$.
Note that by Theorem \ref{THM4.2} (i), we have
 $    H (Du^\epsilon)\to H(Du) $  in $L^t_\loc(U)$ for $t\in[1,\fz)$, and hence almsot everywhere in $U$, as $\ez\to0$. Thus
  $$f^{(\alpha)}_\sz=\lim_{\ez\to0} [H(Du^\ez)+\sz]^{\alpha}= [\lim_{\ez\to0} H(Du^\ez)+\sz]^{\alpha}
=[H(Du)+\sz]^{\alpha}$$ almost everywhere in $U$.
 Therefore $[H (Du^\epsilon )+\sz]^{\alpha}\rightarrow [H(Du )+\sz]^{\alpha}$ in $L^t _\loc(U)$
 for all $t\in [1,\infty)$ and weakly in $W^{1,2}  _\loc(U)$ as $\ez\to0$.

By Theorem \ref{THM3.5} (ii) again  we know that
 $[H (Du  )+\sz]^{\alpha}\in W^{1,2}_\loc(U)$ uniformly in $\sz\in(0,1]$.
Since $[H (Du  )+\sz]^{\alpha}\to  [H (Du  ) ]^{\alpha}$ in $L^t_\loc(U)$ for $t\in[1,\fz)$ as $\sz\to0$,
by the weak compactness of $W^{1,2}_\loc(U)$,
 we know that
 $[H (Du  )+\sz]^{\alpha}\rightarrow [H (Du  ) ]^{\alpha}$ in $ L^t_\loc (U)$ with $ t\in[1,\infty)$
and weakly in $W^{1,2}_\loc(U)$ as $\sz\to0$.

  This completes the proof of Theorem \ref{THM4.2}.
\end{proof}

\section{Proofs of Theorems \ref{THM1.1}\&\ref{THM1.2} and  a flatness of $u$  when $H \in C^\fz(\rr^2)$ }

In this section, we assume that $H\in C^\fz(\rn)$ satisfies  (H1')\&(H2).
Let $\Omega\subset \mathbb R^2$ be     any domain, and  $u\in  AM_H(\Omega)$, equivalently, $u\in C^0(\Omega)$ be a   viscosity solution to \eqref{eq1.2}.

First, we derive Theorems \ref{THM1.1}\&\ref{THM1.2} from Theorem \ref{THM4.2} and Theorem \ref{THM3.5}.

\begin{proof}[Proof of Theorem \ref{THM1.1}  when $H\in C^\fz(\rr^2)$ satisfies (H1'){\rm\&}(H2).]
Let    $U\Subset  \Omega$ be an arbitrary domain and
   $ \alpha>\frac12-\tau_H( \|H(Du) \|_{L^\fz(  U)}).$
   Note that up to considering any smooth domain $\wz U\subset U$, where
   $\tau_H( \|H(Du) \|_{L^\fz(  U)})\le \tau_H( \|H(Du) \|_{L^\fz( \wz U)})$,
   we may assume that $U$ is smooth.
   By Theorem \ref{THM4.2}, we already know that $[H(Du)]^\alpha\in W^{1,2}_\loc(U )$.
 It then suffices to prove that
 \begin{align}\label{eq4.6}
 \|D[H(Du)]^{\alpha}\|_{L^2(V)}^2&
 \le\frac{C\alpha^2(\alpha+1)}{[ \alpha+\tau_H(\|H(Du) \|_{L^\fz( U)})-\frac12]^2}
 \left[\frac{ \Lambda_H( \|H(Du) \|_{L^\fz(  U)})}{\lambda_H( \|H(Du) \|_{L^\fz(  U)})}\right]^2  \\
 &\quad \times  \frac1{[\dist(V,W^\complement)]^2}\int_{W}   [H(Du )]^{2\alpha }\,dx,\nonumber
 \end{align}
whenever $V\Subset W\Subset U$ with $\dist(V,W^{\complement})=\frac12\dist(W,U^{\complement})$.
For any $\ez\in(0,1)$ let $u^\ez\in C^\fz(U)\cap C^0(\overline U)$ be the solution to \eqref{eq3.1}.

If $\alpha\ge 1/2$, by Theorem \ref{THM4.2} (i) we  know that $[H (Du^\epsilon )]^{\alpha}\rightarrow [H(Du)]^{\alpha}$
weakly in $W^{1,2}_{\loc}(U)$.  Hence
$$\|D[H(Du)]^{\alpha}\|_{L^2(V)}^2
 \le \liminf_{{\epsilon}\rightarrow0}\|D[H (Du^\epsilon )]^\alpha\|_{L^2(V)}^2.$$
Note that
  $H(Du^\ez)\to H(Du)$  in $L^t( W)$ for $ t\ge1$ as $\ez\to0$, as given in
  Theorem \ref{THM4.2} (ii).
  Then \eqref{eq4.6} follows from  \eqref{eq3.7}.
  Moreover, by  the arbitrariness of $U$, we know that $[H(Du)]^\alpha\in W^{1,2}_\loc(\Omega )$.

 If $1/2-\tau_H(\|H(Du)\|_{L^\fz(U)})<\alpha<1/2$,
  by Theorem \ref{THM4.2} (iii)  we  know that  for each $\sz>0$,
  $[H (Du^\epsilon )+\sz]^{\alpha}\rightarrow [H(Du)+\sz]^{\alpha}$
weakly in $W^{1,2}_{\loc}(U)$ as $\ez\to0$, and that $[H (Du  )+\sz]^{\alpha}\rightarrow [H(Du) ]^{\alpha}$
weakly in $W^{1,2}_{\loc}(U)$ as $\sz\to0$. Thus
 \begin{align*}\|D[H(Du)]^{\alpha}\|_{L^2(V)}^2
 &\le \liminf_{{\epsilon}\rightarrow0}\|D[H (Du^\epsilon )+\sz]^\alpha\|_{L^2(U)}^2
   \le \liminf_{\sz\to0}\liminf_{{\epsilon}\rightarrow0}\|D[H (Du^\epsilon )+\sz]^\alpha\|_{L^2(U)}^2.
\end{align*}
Note that
  $[H(Du^\ez)+\sz]\to [H(Du)+\sz]$  in $L^t( U)$ for $ t\ge1$ as $\ez\to0$, as given in
  Theorem \ref{THM4.2} (ii).
 Then \eqref{eq4.6} follows from  \eqref{eq3.10}, as desired.
 This completes the proof of Theorem \ref{THM1.1} when $H\in C^\fz(\rr^2)$ satisfies (H1')\&(H2).
\end{proof}

\begin{proof}[Proof of Theorem 1.2    when $H\in C^\fz(\rr^2)$ satisfies (H1'){\rm\&}(H2).]

{\it Proof of (i).} Given any $U\Subset \Omega$, for $\ez\in(0,1)$ let $u^\ez\in C^\fz(U)\cap C^0(\overline U)$ be the solution to \eqref{eq3.1}. Since $Du^\ez\rightarrow Du$ in $L^t_\loc(U)$ for any $t\ge 1$ as given in Theorem \ref{THM4.2} (ii), we
have
\begin{align*}\int_U-\det D^2u \phi\,dx &=\frac 1 2\int_U [u _{x_i}u _{x_j}\phi_{x_ix_j}+|Du |^2\phi_{x_ix_i}]\,dx\\
&=\frac 1 2\lim_{\ez\rightarrow 0}\int_U [u^\ez_{x_i}u^\ez_{x_j}\phi_{x_ix_j}+
|Du^\ez|^2\phi_{x_ix_i}]\,dx\\
&=\lim_{\ez\rightarrow 0}\int_U-\det D^2u^\ez \phi\,dx\quad\forall \phi\in C_c^\fz(U)\end{align*}
Write  $V={\rm\,supp\,}\phi \Subset U$.
By Theorem \ref{LEM2.3}  we obtain
\begin{align*} -\det D^2u^\ez
 &  \ge 4\frac{\wz \tau_H(\|H(Du^\ez)\|_{L^\fz(V)})}{\det D ^2_{pp}H(Du^\ez) }  \langle  {D^2_{pp}H}(Du^\ez)  D[H(Du^\ez)]^{1/2} ,D[H(Du^\ez)]^{1/2} \rangle\\
 &\ge 4 \frac{ \tau_H(\|H(Du^\ez)\|_{L^\fz(V)}) }
 {\Lambda_H(\|H(Du^\ez)\|_{L^\fz(V)}) }|D[H(Du^\ez)]^{1/2}|^2 \quad{\rm in}\ V.\end{align*}
Note that by Theorem \ref{THM3.2} and the monotonicity and
 right-hand continuity of $\Lambda_H,\tau_H$, we have
$$ \lim_{\ez\to0}\frac{ \tau_H(\|H(Du^\ez)\|_{L^\fz(V)}) }
 {\Lambda_H(\|H(Du^\ez)\|_{L^\fz(V)}) }\ge \frac{ \tau_H(\|H(Du )\|_{L^\fz(U)}) }
 {\Lambda_H(\|H(Du) \|_{L^\fz(U)}) }.$$
Thus, applying Theorem \ref{THM4.2} we get
 \begin{align*}\int_U-\det D^2u \phi\,dx\ge
  4
 \frac{\tau_H( \|H(Du)\|_{L^\fz(U)} )}
 {\Lambda_H( \|H(Du)\|_{L^\fz(U)} )}\int_U|D[H(Du)]^{1/2}|^2\phi\,dx\quad\forall\phi\in C^\fz_c(U). \end{align*}
This implies that  $-\det D^2u \,dx$ is a nonnegative Radon measure with
  \begin{align*} -\det D^2u \,dx   \ge
  4 \frac{\tau_H( \|H(Du)\|_{L^\fz(U)} )}
 {\Lambda_H( \|H(Du)\|_{L^\fz(U)} )} |D[H(Du)]^{1/2}|^2 \,dx.
 \end{align*}

 Finally, for any $V\Subset U$, let $\phi$ be as in \eqref{eq3.4}.
We have
\begin{align*}  \int_{ V} - \det D^2u  \,d x&\le \int_{U} - \det D^2u   \phi \,d x\\
&=\frac12\int_U [u_{x_i}u_{x_j}\phi_{x_ix_j}+
|Du|^2\phi_{x_ix_i}]\,dx  \le C\frac1{[\dist (V,U^\complement)]^2}\int_{U}  | Du  |^2  \,d x
\end{align*}
 as desired.

 {\it Proof of (ii).}    It suffices to prove that
\begin{equation}\label{eq4.7}\int_{U}\langle D[H(Du)]^{\alpha},D_p H(Du)\rangle\phi\,d x =0\quad\forall\,\mbox{ $U \subset \Omega$ and $\phi\in C^\fz_c(U)$.}
\end{equation}
 Given any $U\Subset \Omega$, for $\ez\in(0,1)$ let $u^\ez\in C^\fz(U)\cap C^0(\overline U)$ be the solution to \eqref{eq3.1}.

{\it Case $\alpha\ge 1/2 $.}
By Theorem \ref{THM4.2} we have
$$\mbox{$D_p H (Du^\epsilon )\to D_p H(Du) $ in $L^2_\loc(U)$ and
$D[H (Du^\epsilon )]^{\alpha}\to D[H(Du)]^{\alpha}$  weakly in $L^2_\loc(U)$ as $\ez\to0$.}$$
Therefore, for any $\phi\in C^\infty_c(U)$, we have
 \begin{align*}
 \int_{U}\langle D[H(Du)]^{\alpha},D_p H(Du)\rangle\phi\,d x &=\lim_{{\epsilon}\rightarrow0}
 \int_{U}\langle D[H (Du^\epsilon )]^{\alpha},D_p H (Du^\epsilon )\rangle\phi\,d x .
\end{align*}
Since $D[H (Du^\epsilon )+\sigma]^{\alpha}\rightharpoonup D[H (Du^\epsilon )]^{\alpha}$ weakly in $L^2_\loc (U)$ as $\sigma\rightarrow 0$, we have
 \begin{align*}
 \int_{U}\langle D[H(Du)]^{\alpha},D_p H(Du)\rangle\phi\,d x
& =\lim_{{\epsilon}\rightarrow0}\lim_{\sigma\rightarrow0}
\int_{U}\alpha[H (Du^\epsilon )+\sigma]^{\alpha-1}\langle D[H(Du^\ez)] , {D_pH}(Du^\ez) \rangle\phi\,d x\\
& =-\lim_{{\epsilon}\rightarrow0}\lim_{\sigma\rightarrow0}\epsilon \int_{U}\alpha[H (Du^\epsilon )+\sz]^{\alpha-1} {\rm{div}}[{D_pH}(Du^\ez)] \phi\,d x,
\end{align*}
where in the last identity, we use ${\mathscr A}_{H }[u^\epsilon ]+\epsilon{\rm{div}}[{D_pH}(Du^\ez)]=0$ in $\Omega$.
 Write $V={\,\rm supp\,}\phi$.  By Theorem \ref{THM3.5} and Theorem \ref{THM3.2},
we have
 \begin{align*}
 &\ez^{1/2}\int_{U}\alpha[H (Du^\epsilon )+\sz]^{\alpha-1} {\rm{div}}[ {D_pH}(Du^\ez)] \phi \,d x\\
  &\le \alpha\ez^{1/2}\|\phi\|_{L^\fz(V)}  \|[H (Du^\epsilon )+\sz]^{\alpha-1/2}\|_{L^\fz(V)}
  \int_{V}[H (Du^\epsilon ) ]^{ -1/2} |{\rm{div}}[ {D_pH}(Du^\ez)] | \,d x\\
 &\le C\alpha\|\phi\|_{L^\fz(V)}[\| H (Du  )  \|^{\alpha-1/2}_{L^\fz(U)}+1]\ez^{1/2}\left[\int_{V}[H (Du^\epsilon ) ]^{ -1} ({\rm{div}}[ {D_pH}(Du^\ez)])^2 \,d x\right]^{1/2},
  \end{align*}
  which is uniformly bounded in $\ez\in(0,1]$.
 This gives
 $$\int_{U}\langle D[H(Du)]^{\alpha},D_p H(Du)\rangle\phi\,d x=0$$
   as desired.

{\it Case $1/2-\tau_H(\|H(Du)\|_{L^\fz(U)})<\alpha< 1/2 $. }
 By Theorem \ref{THM4.2} (iii),  we have
 $D[H (Du  )+\sz]^{\alpha} \to D[H (Du  ) ]^{\alpha}$ weakly in $L^2_\loc(U)$ as $ \sz\to0$;
 and  for each $\sz>0$, we have $D[H (Du^\ez  )+\sz]^{\alpha} \to D[H (Du  ) +\sz]^{\alpha}$ weakly in $L^2_\loc(U)$ as $ \ez\to0$. By Theorem \ref{THM4.2} (ii), $D_p H (Du^\epsilon )\to D_p H (Du  )$ in $L^2_\loc(U)$ as $ \ez\to0$.
Thus
 \begin{align*}
 \int_{U}\langle D[H(Du)]^{\alpha},D_p H(Du)\rangle\phi\,d x &=\lim_{\sz\rightarrow0}
 \int_{U}\langle D[H (Du  )+\sz]^{\alpha},D_p H (Du  )\rangle\phi\,d x \\
 &=\lim_{\sz\rightarrow0}\lim_{\ez\to0}
 \int_{U}\langle D[H (Du^\ez  )+\sz]^{\alpha},D_p H (Du^\epsilon )\rangle\phi\,d x .
\end{align*}

 Write $V={\,\rm supp\,}\phi$. Using ${\mathscr A}_{H }[u^\epsilon ]+\epsilon{\rm{div}}[{D_pH}(Du^\ez)]=0$ in $\Omega$,
 we obtain
 \begin{align*}
 &\left|\int_{U}\langle D[H(Du)]^{\alpha},D_p H(Du)\rangle\phi\,d x \right|\\
&\quad= \lim_{\sigma\rightarrow0}\lim_{{\epsilon}\rightarrow0}\epsilon  \left|\int_{U}\alpha[H (Du^\epsilon )+\sz]^{\alpha-1} {\rm{div}}[ {D_pH}(Du^\ez)] \phi\,d x\right|\\
&\quad\le \lim_{\sigma\rightarrow0}\lim_{{\epsilon}\rightarrow0} \|\phi\|_{L^\fz(V)} \ez^{1/2}\ez^{1/2}\left[\int_{V}\alpha[H (Du^\epsilon )+\sz]^{2\alpha-2} ({\rm{div}}[ {D_pH}(Du^\ez)])^2 \,d x\right]^{1/2},
\end{align*}
 by  Theorem \ref{THM3.5}, which equals to $0$.
 This gives \eqref{eq4.7} as desired.
  This completes the proof of Theorem \ref{THM1.2} when $H\in C^\fz(\rr^2)$ satisfies (H1')\&(H2).
\end{proof}

Next, we obtain an integral flatness estimate of $u$ which will be used later.

\begin{lem}\label{LEM4.3}
For  any $ B\Subset 2B\Subset \Omega$ and linear function $F$, we have
\begin{align*}
  &\mint-_{\frac12B }\langle D_pH(Du ),Du -DF\rangle ^2 \,d x\\
  &\quad\le
   C\frac{[\Lambda_H (\fz)]^2}{ \lambda_H (\fz) } \|H(Du)\|_{L^\fz(B)}
    \left[\mint-_{ B } \left((|Du |+|DF|)^2\frac{(u -F)^2}{r^2} + \frac{(u -F)^4}{r^4} \right)\,d x\right]^{1/2}.
\end{align*}
\end{lem}

 \begin{proof}[Proof of Lemma \ref{LEM4.3}.]

 Let $U=B\Subset\Omega$, and for any $\ez\in(0,1]$, let $u^\ez\in C^\fz(U)\cap C^0(U)$
 be the solution to  \eqref{eq3.1}.
By Theorem \ref{THM4.2}, we have $u^\ez\to u$ and
$ D_pH(Du^\ez)\to  {D_pH}(Du^\ez)$ in $L^t_\loc(U)$ for all $t\ge1$.
 Thus
\begin{align*}
    &\mint-_{\frac12B }\langle  D_pH(Du ),Du -DF\rangle ^2 \,d x=\lim_{\ez\to0}  \mint-_{\frac12B }\langle D_pH(Du^\ez),Du^\epsilon-DF\rangle ^2 \,d x .\end{align*}
   By Lemma \ref{LEM2.8}, we have
   \begin{align*}
    &\mint-_{\frac12B }\langle D_pH(Du ),Du -DF\rangle ^2 \,d x\\
  &\quad\le C \Lambda_H (\fz)\liminf_{\ez\to0} \left[\int_{\frac34B } | D[H(Du^\ez)]|^2  \,d x\right]^{1/2}\\
&\quad\quad\quad\times\left[\mint-_{\frac34B } \left((|Du^\epsilon|+|DF|)^2\frac{(u^\epsilon-F)^2}{r^2} + \frac{(u^\epsilon-F)^4}{r^4} \right)\,d x\right]^{1/2}\\
&\quad\quad +\liminf_{\ez\to0} \frac{[\Lambda_H (\fz)]^2}{\lambda_H (\fz)}\left[\mint-_{\frac34B }   [H (Du^\ez)]^2 \,d x\right]^{1/2}\left[\mint-_{\frac34B }  \frac{(u^\epsilon-F)^4}{r^4}     \,d x\right]^{1/2}\\
&\quad\quad+\liminf_{\ez\to0}
\left[\int_{\frac34B }\ez^2({\rm div} [{D_pH}(Du^\ez)])^2 \,d x\right]^{1/2} \left[\mint-_{\frac34 B }\frac{(u^\epsilon-F)^2}{r^2}  \,d x\right]^{1/2}\\
&\quad=:J_1+J_2 +J_3.
  \end{align*}
  Applying  \eqref{eq3.6} to $V=\frac34B$, we know that
$\ez^{1/2} \|{\rm div} {D_pH}(Du^\ez)   \|_{L^2(\frac34B)}$ is bounded uniformly in $\ez\in(0,1]$,
and hence, by Theorem \ref{THM3.2},   $J_3=0$.
By Theorem \ref{THM4.2} (i) we have
 \begin{align*}
 J_2&\le  C \frac{[\Lambda_H (\fz)]^2}{\lambda_H (\fz)} \|H(Du)\| _{L^\fz(B)}
 \left[\mint-_{ B }   \frac{(u -F)^4}{r^4}  \,d x\right]^{1/2}
\end{align*}
Applying  \eqref{eq3.7} to $V=\frac34B$ and $W=\frac45B$,
 by  Theorem \ref{THM4.2} (i)    we have
 \begin{align*}\liminf_{\ez\to0}   \int_{\frac34B } | D[H(Du^\ez)]|^2  \,d x  & \le C\left[\frac{\Lambda_H (\|H(Du )\|_{L^\fz(B)})}{\lambda_H (\|H(Du )\|_{L^\fz(B)})}\right]^2
\mint-_{\frac45B }[H(Du)]^2\,dx\\
&\le  C\left[\frac{\Lambda_H (\fz)}{\lambda_H (\fz)}\right]^2
 \|H(Du)\|^2_{L^\fz(B)}.\end{align*}
Thus by  $u^\ez\to u$   in $W^{1,t}_\loc(U)$ for all $t\ge1$ as $\ez\to 0$ again, we have
 \begin{align*}
 J_1&\le  C \frac{[\Lambda_H (\fz)]^2}{\lambda_H (\fz)} \|H(Du)\| _{L^\fz(B)}
 \left[\mint-_{ B } \left((|Du |+|DF|)^2\frac{(u -F)^2}{r^2} + \frac{(u -F)^4}{r^4} \right)\,d x\right]^{1/2}.
\end{align*}
This  completes the proof of Lemma \ref{LEM4.3}.
  \end{proof}

\section{Sobolev  approximation via $u^\dz\in AM_{H^\dz}$ when $H\in C^0(\rr^2)$ (or $  C^1(\rr^2)$) }

In this section, we assume that $H\in C^0(\rr^2)$ satisfies  (H1')\&(H2). Let $\{H^\dz\}_{\dz\in (0,1]}$ be the smooth approximation of   $H$ as given in Appendix A.
Let $\Omega\subset\rr^2$ be any domain and $u\in AM_H(\Omega)$.
Let   $U\Subset \Omega$ be any domain.
For any $\dz\in(0,1]$, let $$\mbox{$u^\dz\in C^0(\overline U)\cap AM_{H^\dz}(U)$
with $u^\dz=u$ on $\partial U$.}$$
Note that  $ H^\dz\in C^\fz(\rr^2) $ satisfies   (H1')\&(H2),
as proved in Section 4, Theorems \ref{THM1.1}\&\ref{THM1.2} and Lemma \ref{LEM4.3} hold for $H^\dz$ and $u^\dz$ in $U$.

In Section 5.1, we  prove that
 $u^\dz\to u$ in $C^{0} (\overline U)$ as $\dz\to 0$  and
  $\lim_{\dz\to0}\|H^\dz(Du^\dz)\|_{L^\fz(U)}\le\|H (Du )\|_{L^\fz(U)} $; see Theorem \ref{THM5.1}.

In Section 5.2,   we show that for any $\alpha\ge1/2$ and some $\dz_0\in(0,1)$,  one has
$[H^{\dz }(Du^{\dz })]^\alpha\in W^{1,2}_\loc (U)$ uniformly in $\dz\in(0,\dz_0]$;
and that   there is a
 sequence $\{\dz_j\}_{j\in\nn}$   which converges to $0$ such that  for any $ 1/2-\tau_H(\|H(Du)\|_{L^\fz(U)})<\alpha<1/2$
 and some $j_\alpha\in\nn$, one has
$[H^{\dz_j}(Du^{\dz_j})]^\alpha\in W^{1,2}_\loc (U)$ uniformly in $j\ge j_\alpha$; see Lemma 5.2 and Theorem \ref{THM5.2}.

In Section 5.3, by   the flatness estimate of $u^\dz$ in Lemma \ref{LEM4.3},
when $\alpha\ge1/2$ we obtain $[H^{\dz }(Du^{\dz })]^\alpha\to [H(Du)]^\alpha$ in $L^t_\loc(U)$ and weakly $W^{1,2}_\loc(U)$ for any $t\ge1$  as $\dz\to0$;
when $ 1/2-\tau_H(\|H(Du)\|_{L^\fz(U)})<\alpha<1/2$ we have
$[H^{\dz_j }(Du^{\dz_j })]^\alpha\to [H(Du)]^\alpha$ in $L^t_\loc(U)$ and weakly $W^{1,2}_\loc(U)$ for any $t\ge1$ as $j\to\fz$;
If $H\in C^1(\rr^2)$ additionally, we also have
$u^{\dz }\to u$ in $W^{1,t}_\loc(U)$ and $D_pH^\dz(Du^\dz)\to D_pH(Du)$ in $L^t_\loc(U)$ for
any $t\ge1$ as $\dz\to0$; see Theorem \ref{THM5.3}.

\subsection{Uniform $L^\fz $-estimates of $H^\dz(Du^\dz)$  }

We prove   the following result.

\begin{thm}\label{THM5.1}
We have
$$\| H^\dz(D u^\dz) \|_{L^\fz ( U )} \le   \frac 1{2}\Lambda_H(\fz) \|u\|^2_{C^{0,1}(\partial U)} $$
and
\begin{equation}\label{eq5.x1}   \limsup_{\dz\to0}\|H^\dz (Du^\dz)\|_{L^\fz(U)}\le \|H (Du )\|_{L^\fz(U)}.
\end{equation}
Moreover, $ u^\dz\to u $ in $C^0(\overline U)$.
\end{thm}

\begin{proof}
{\it Step 1.} 
For any $\dz\in (0,1]$,   applying Lemma \ref{LEM7.14} to $H^\dz$ and $u^\dz$, we have
 $u^\dz\in C^{0,1} (\overline U)$
  and
  \begin{equation}\label{eq5.1} \|H^\dz(D u^\dz)\|_{L^\fz(  U )} \le \sup_{|p|\le L}H^\dz(p)
 \end{equation}
 where    $L=\|u\|_{C^{0,1}(\partial U)}$.
 By Lemma \ref{LEM7.3} (i) and Lemma \ref{LEM7.1} (iii), we have
 $$\sup_{|p|\le L}H^\dz(p) \le  \frac12\Lambda_H(\fz) \sup_{|p|\le L} |p|^2\le \frac12\Lambda_H(\fz)L^2
 \quad \mbox{ and $H^\dz(D u^\dz)\ge \frac12 \lambda_H(\fz) |D u^\dz|^2$}.$$
  Thus,
 $$\| H^\dz(D u^\dz) \|_{L^\fz ( U )} \le   \frac 1{2}\Lambda_H(\fz) \|u\|^2_{C^{0,1}(\partial U)} \quad {\rm and }\quad \| D u^\dz \|_{L^\fz ( U )} \le  \left[\frac {\Lambda_H(\fz)}{\lambda_H(\fz)}\right]^{1/2}L.
$$
By Arzela-Ascolli's Theorem, we know that, up to some subsequence,
$u^\dz\to \hat u$ in $C^{0 }(\overline U)$ as $\dz\to0$ for some $\hat u \in C^{0,1}(\overline U)$.

{\it Step 2.} We show that  $\hat u =u$ in $\overline U$.
Note that $\hat u=u$ in $\partial U$.
To
 get $\hat u =u$ in $U$, thanks to the uniqueness given in Lemma \ref{LEM7.1},
  it suffices to prove that
 $\hat u\in AM_{H}(U)$.
By Lemma \ref{LEM7.11}, we only need to show that $\hat u \in CC_H(U)$.

The proof of $\hat u \in CCB_H(U)$ is similar to that of $\hat u \in CCA_H(U)$, and hence is omitted.
To see $\hat u \in CCA_H(U)$, let $V\Subset U$ be any domain and $x_0\in U\setminus V$ be any point,
 and assume that
$$
\hat u (x) \le C^H_a(x-x_0)+b \quad\forall x\in\partial V
$$
for some  $a\ge 0$. We only need to show that this inequality also holds in $V$.

If $a=0$, then $C^H_a\equiv 0$ by Lemma \ref{LEM7.9} we know that $$
\hat u (x) \le b = C^H_a(x-x_0)+b \quad\forall x\in  V.
$$
Assume that $a>0$.
By Corollary \ref{cone1}, for  any  $\ez>0$  there exists $\dz_{a,\ez}>0$ such that
$$ C^{H}_{(1+\ez)^{-1}\wz a}(x)\le C^{H^{\dz}}_{\wz a}(x)\le  C^{H}_{(1+\ez)\wz a}(x)\quad \forall x\in \rn, \wz a\in[\frac12a,2a].$$
By Lemma \ref{LEM7.13},
    for every $\ez > 0$ and  for all
     $0<\dz<\dz_{a,\ez}$, we have
$$
\hat u (x) \le C^{H^\dz}_{(1+\ez)a}(x-x_0)+b \quad\forall x\in\partial V.
$$
Thus, by $u^\dz\to \hat u$ in $C^0(\overline U)$,  there exists $\wz \dz_{a,\ez}\in(0,\dz_{a,\ez}]$ such that for all
     $0<\dz<\wz \dz_{a,\ez}$,  one has
$$
  u^\dz (x) \le C^{H^\dz}_{(1+\ez)a}(x-x_0)+b+\ez \quad\forall x\in\partial V \mbox{ and hence,
  by $u^\dz\in AM_{H^\dz}(U)$ and Lemma \ref{LEM7.11}, $\forall x\in V$}.
$$
%
Thus
$$
  \hat u (x) \le C^{H}_{ a}(x-x_0)+b+2\ez   \quad\forall x\in  V.
$$
Letting $\ez\to0$, we  obtain
    $$
  \hat u (x) \le C^{H}_{ a}(x-x_0)+b    \quad\forall x\in  V
$$
as desired.

{\it Step 3.} We prove \eqref{eq5.x1}.
We claim that
 \begin{equation}\label{eq5.2} \|H^\dz(Du^\dz)\|_{L^\infty(U )}\le \|H^\dz(Du)\|_{L^\infty(U )}.
  \end{equation}
Indeed, for almost all $x\in \{z\in U, u^\dz(z)=u(z)\}$,
we have $Du^\dz(x)=Du(x)$. So it suffices to  prove  \begin{equation}\label{eq5.x2}\|H^\dz(Du^\dz)\|_{L^\infty( \{z\in U: u^\dz(z)\neq u(z)\} )}\le\|H^\dz(Du)\|_{ L^\infty( \{z\in U: u^\dz(z)\neq u(z)\} )}  . \end{equation}
 For any $\ez>0$, set $V^{+,\ez}:=  \{z\in U, u^\dz(z)>u(z)+\ez\}$ and $V^{-,\ez}:=  \{z\in U, u^\dz(z)<u(z)-\ez\}$.  Then  $V^{\pm,\ez}\Subset U$ and  $u^\dz=u\pm \ez$ on $\partial V^{\pm,\ez}$.
  Since $u^\dz$ is an absolute minimizer for $H^\dz$ in $U$, we have
 $$\|H^\dz(Du^\dz)\|_{L^\infty(V^{\pm,\ez} )}\le\|H^\dz(Du)\|_{L^\infty(V^{\pm,\ez} )}.$$
Sending $\ez \rightarrow 0$, we arrive at the desired \eqref{eq5.x2}.

 Note that \eqref{eq5.x1} follows from \eqref{eq5.x2} and
  \begin{align}\label{eq5.x3}
\limsup_{\dz\to0}\|H^\dz(Du)\|_{L^\infty(U )} \le \|H (Du)\|_{L^\infty(U )} .
\end{align}
To see \eqref{eq5.x3},
note that for any $a> \|H(Du) \|_{  L^\fz(U)}$, by Lemma \ref{LEM7.11} we already have
$$
u (x)-u (y)\le C_{a }^{H  }(x-y) \ \mbox{whenever $x,y$ in some line segment in $U$}. $$
By Corollary \ref{cone1}, there exists $\dz_{a,\ez}\in(0,1]$ such that for all $0<\dz<\dz_{a,\ez}$
one has $ C_{a }^{H } \le C_{a(1+\ez)}^{H^\dz}  $,  and
hence,
$$
u (x)-u (y)\le C_{a(1+\ez)}^{H^\dz}(x-y) \ \mbox{whenever $x,y$ in some line segment in $U$}.$$
  From this and  Lemma \ref{LEM7.11},
we have $\|H^\dz(Du)\|_{L^\infty(U )}\le a(1+\ez)$.
Letting $\dz\to0$ and $\ez\to 0$ in order we conclude \eqref{eq5.1} as desired.
 This completes the proof of Theorem \ref{THM5.1}.
\end{proof}

\subsection{Uniform $W^{1,2}_\loc$-estimates of
 $[H^{\dz_j}(Du^{\dz_j})]^\alpha$   }

 The following lemma comes from the definition of $\tau_H$ and Theorem \ref{THM5.1}.

\begin{lem} There exists a positive sequence $\{\dz_k\}_{k\in\nn}$ and $\{\ez_k\}_{k\in\nn}$, both of which converge  to $0$ as $j\to\fz$, such that for all $k\in\nn$,
 $$\|H^{ \dz_{k} }(Du^{  \dz_{k} })\|_{L^\fz(U)}\le \|H(Du)\|_{L^\fz(U)}+\ez_k $$
and
  \begin{align*}  
    \tau_{H^{   \dz_{k }}}(\|H(Du)\|_{L^\fz(U)}+\ez_k)
    &\ge
  \tau_{H }(\|H(Du)\|_{L^\fz(U)} )-2^{-k-1} \tau_{H}(\|H(Du)\|_{L^\fz(U)}).
  \end{align*}
\end{lem}

\begin{proof}
Note that $\tau_{H}(\|H(Du)\|_{L^\fz(U)})>0$. By \eqref{eq1.8}, for any $k\in\nn$,
there exist $0<\ez_k<2^{-k}$ such that
 $$\tau_{H}(\|H(Du)\|_{L^\fz(U)})\le \limsup_{\dz\to0}\tau_{H^\dz}(\|H(Du)\|_{L^\fz(U)}+\ez_k)
 +2^{-k-2} \tau_{H}(\|H(Du)\|_{L^\fz(U)}).
 $$
Since  the function $r\to \tau_{H^\dz}(r)$ is decreasing in $[0,\fz)$ for any $\dz>0$,
 up to considering $\min_{1\le j\le k}\ez_j$,
 we may assume that $\ez_k$ is decreasing in $k$.
  Let $\{\hat\dz_{k,j}\}_{j\in\nn}$ with $\hat \dz_{k,j}\to0$ as $j\to\fz$ such that for all $j\ge1$ we have
  $$   \lim_{j\to\fz}\tau_{H^{\hat \dz_{k,j}}}(\|H(Du)\|_{L^\fz(U)}+\ez_k)=
  \limsup_{\dz\to0}\tau_{H^\dz}(\|H(Du)\|_{L^\fz(U)}+\ez_k) .$$
There exists $  j_k^\ast$ such that for all $j\ge  j^\ast_k$ we have
    $$    \tau_{H^{\hat \dz_{k,j}}}(\|H(Du)\|_{L^\fz(U)}+\ez_k)\ge
  \limsup_{\dz\to0}\tau_{H^\dz}(\|H(Du)\|_{L^\fz(U)}+\ez_k)-2^{-k-2} \tau_{H}(\|H(Du)\|_{L^\fz(U)}),$$
  and hence
   $$    \tau_{H^{\hat \dz_{k,j}}}(\|H(Du)\|_{L^\fz(U)}+\ez_k)\ge
 \tau_{H }(\|H(Du)\|_{L^\fz(U)} )-2^{-k-1} \tau_{H}(\|H(Du)\|_{L^\fz(U)}) .$$
By \eqref{eq5.x1}, there exists $j_k\ge   j_k^\ast$ such that  for all $j\ge j_k$, one has
 $$\|H^{\hat \dz_{k,j} }(Du^{\hat \dz_{k,j} })\|_{L^\fz(U)}\le \|H(Du)\|_{L^\fz(U)}+\ez_k.$$
Since
the function $r\mapsto\tau_{H^\dz}(r)$ is  decreasing in $[0,\fz)$ for any $\dz>0$,
 we have,  for $j\ge j_k$,
  $$    \tau_{H^{\hat \dz_{k,j}}}(\|H^{\hat \dz_{k,j} }(Du^{\hat \dz_{k,j} })\|_{L^\fz(U)} )\ge
   \tau_{H^{ \hat \dz_{k,j }}}(\|H(Du)\|_{L^\fz(U)}+\ez_k).$$
Set $\dz_k=\hat \dz_{k,j_k}$ for all $k\in\nn $.  The proof of Lemma 5.2 is complete.
\end{proof}

We have the following  uniform $W^{1,2}_\loc$-estimates of $[H^{\dz_j} (Du^{\dz_j})]^{\alpha}$.
 \begin{thm} \label{THM5.2}
 (i) There exists $\dz_0\in(0,1]$ such that for any $\alpha\ge1/2$,  we have
  \begin{equation}\label{eq5.yy1} \sup_{ \dz\in(0,\dz_0]}\|D[H^{\dz } (Du^{\dz })]^{\alpha}\|_{L^{2}(V)}<\fz. \end{equation}
  and \begin{align}\label{eq5.yy2}
 \liminf_{\dz\to0}\|D[H^{\dz }(Du^{\dz })]^{\alpha}\|_{L^2(V)}^2&
 \le\frac{C\alpha^2(\alpha+1)}{[\alpha+\tau_{H }(\|H (Du ) \|_{L^\fz(U)})-\frac12]^2}\left[\frac{ \Lambda_{H }( \|H (Du ) \|_{L^\fz(  U)})}{\lambda_{H }( \|H (Du ) \|_{L^\fz(  U)})}\right]^2\nonumber
 \\ &\quad\times
 \frac1{[\dist(V,\partial W)]^2} \liminf_{\dz\to 0}\int_{W}   [H^{\dz }(Du^{\dz } )]^{2\alpha }\,dx.
 \end{align}
 (ii)
  Let  $\{\dz_j\}_{j\in\nn}$  be as in Lemma 5.2.  For any $1/2-\tau_H(\|H(Du)\|_{L^\fz(U)})<\alpha<1/2$,
 there exists $j_\alpha$ such that for all $V\Subset W\Subset U$, we have
   \begin{equation}\label{eq5.y1} \sup_{ j\ge j_\alpha}\|D[H^{\dz_j} (Du^{\dz_j})]^{\alpha}\|_{L^{2}(V)}<\fz,  \end{equation}
and moreover
\begin{align}\label{eq5.y2}
 \liminf_{j\to\fz}\|D[H^{\dz_j}(Du^{\dz_j})]^{\alpha}\|_{L^2(V)}^2&
 \le\frac{C\alpha^2(\alpha+1)}{[\alpha+\tau_{H }(\|H (Du ) \|_{L^\fz(U)})-\frac12]^2}\left[\frac{ \Lambda_{H }( \|H (Du ) \|_{L^\fz(  U)})}{\lambda_{H }( \|H (Du ) \|_{L^\fz(  U)})}\right]^2
 \\ &\quad\times
 \frac1{[\dist(V,\partial W)]^2} \liminf_{j\to\fz}\int_{W}   [H^{\dz_j}(Du^{\dz_j} )]^{2\alpha }\,dx.\nonumber
 \end{align}
 \end{thm}
 \begin{proof}
{\it Case $1/2-\tau_H(\|H(Du)\|_{L^\fz(U)})<\alpha<1/2$.}
Let $j_\alpha$ as the unique $k\in\nn $ such that
$$2^{-k } \tau_{H}(\|H(Du)\|_{L^\fz(U)})\le \alpha+ \tau_{H}(\|H(Du)\|_{L^\fz(U)})- \frac12 < 2^{-k+1} \tau_{H}(\|H(Du)\|_{L^\fz(U)}).$$
By Lemma 5.2,   the decrease of $\tau_{H^\dz}$ and  $\tau_{H^\dz}\le1/2$,  we have
\begin{align}\label{eq5.x5}
&\alpha +\tau_{H^{\dz_k}}(\|H^{\dz_k}(Du^{\dz_k})\|_{L^\fz(U)} )-\frac12 \\
&\quad \ge \alpha + \tau_{H^{\dz_k}}(\|H(Du)\|_{L^\fz(U)}+\ez_k)-\frac12\nonumber \\
&\quad \ge   \alpha  +\tau_{H}(\|H(Du)\|_{L^\fz(U)})-\frac12 - 2^{-k-1} \tau_{H}(\|H(Du)\|_{L^\fz(U)})\nonumber\\
&\quad \ge \frac12 [\alpha  +\tau_{H}(\|H(Du)\|_{L^\fz(U)})-\frac12].\nonumber
  \end{align}

For $j\ge j_\alpha$, applying  Theorem \ref{THM1.1} for $H^{\dz_j}$ we have
\begin{align*}
 \|D[H^{\dz_j}(Du^{\dz_j})]^{\alpha}\|_{L^2(V)}^2&
 \le\frac{C\alpha^2(\alpha+1)}{[\alpha+\tau_{H^{\dz_j}}(\|H^{\dz_j}(Du^{\dz_j}) \|_{L^\fz( W)})-\frac12]^2}\left[\frac{ \Lambda_{H^{\dz_j}}( \|H^{\dz_j}(Du^{\dz_j}) \|_{L^\fz(  W)})}{\lambda_{H^{\dz_j}}( \|H^{\dz_j}(Du^{\dz_j}) \|_{L^\fz(  W)})}\right]^2
 \\ &\quad\times
 \frac1{[\dist(V,\partial W)]^2}\int_{W}   [H^{\dz_j}(Du^{\dz_j} )]^{2\alpha }\,dx \quad\mbox{whenever $V\Subset W\Subset U$.}
 \end{align*}

By Lemma \ref{LEM7.1} (i) and Lemma \ref{LEM7.3} (i), we obtain
\begin{equation}\label{eq5.4}
\frac{ \Lambda_{H^{\dz_j}}( \|H^{\dz_j}(Du^{\dz_j}) \|_{L^\fz(  W)})}{\lambda_{H^{\dz_j}}( \|H^{\dz_j}(Du^{\dz_j}) \|_{L^\fz(  W)})} \le \frac{ \Lambda_{H }(\fz)}{\lambda_{H }(\fz)}\quad \forall j\in\nn.
\end{equation}
From this, Theorem \ref{THM5.1} and \eqref{eq5.x5}  it follows that
\begin{align*}
 \|D[H ^{\dz_j}(Du^{\dz_j})]^{\alpha}\|_{L^2(V)}^2
 &\le  \frac{C\alpha^2(\alpha+1)}{[\alpha+\tau_{H }(\|H (Du ) \|_{L^\fz( U)})-\frac12]^2} \left[\frac{ \Lambda_{H }(\fz)  }{\lambda_ H ( \fz)}\right]^2 \frac{|W|}{[\dist(V,\partial W)]^2} \\
 &\quad\times
[\Lambda_H(\fz) \|u\|^2_{C^{0,1}(\partial U)} ]^{2\alpha }\quad\forall j\ge j_\alpha,
 \end{align*}
which gives \eqref{eq5.y1}.

Moreover,  by Lemma \ref{LEM7.3} (iii) and Theorem \ref{THM5.1}, one gets
\begin{align}\label{eq5.xx5}
\limsup_{j\to\fz}\frac{ \Lambda_{H^{\dz_j}}( \|H^{\dz_j}(Du^{\dz_j}) \|_{L^\fz(  W)})}{\lambda_{H^{\dz_j}}( \|H^{\dz_j}(Du^{\dz_j}) \|_{L^\fz(  W)})}
 &\le \frac{ \limsup_{j\to\fz}\Lambda_{H^{\dz_j}}( \|H^{\dz_j}(Du^{\dz_j}) \|_{L^\fz(  W)})}{\liminf_{j\to\fz}\lambda_{H^{\dz_j}}( \|H^{\dz_j}(Du^{\dz_j}) \|_{L^\fz(  W)})} \\
&\le \frac{\Lambda_{H }( \limsup_{j\to\fz}\|H^{\dz_j}(Du^{\dz_j}) \|_{L^\fz(  W)})}{\lambda_{H }( \liminf_{j\to\fz}\|H^{\dz_j}(Du^{\dz_j}) \|_{L^\fz(  W)})}\nonumber\\
&\le  \frac{ \Lambda_{H }(\|H (Du ) \|_{L^\fz(  U)}) }{\lambda_{H }(\|H (Du ) \|_{L^\fz(  U)} )}. \nonumber
\end{align}
Thanks to \eqref{eq5.x5} we obtain
\begin{align*}
 \liminf_{j\to\fz}\|D[H^{\dz_j}(Du^{\dz_j})]^{\alpha}\|_{L^2(V)}^2&
 \le\frac{C\alpha^2(\alpha+1)}{[\alpha+\tau_{H }(\|H (Du ) \|_{L^\fz(U)})-\frac12]^2}\left[\frac{ \Lambda_{H }( \|H (Du ) \|_{L^\fz(  U)})}{\lambda_{H }( \|H (Du ) \|_{L^\fz(  U)})}\right]^2
 \\ &\quad\times
 \frac1{[\dist(V,\partial W)]^2} \liminf_{j\to\fz}\int_{W}   [H^{\dz_j}(Du^{\dz_j} )]^{2\alpha }\,dx,
 \end{align*}
as desired.

{\it Case $\alpha\ge1/2$.}
The proof in this case is similar to but easier than that in the case $\alpha<1/2$.
Indeed, instead of \eqref{eq5.x5}, we easily have
$$\alpha-\tau_{H^{\dz }}(\|H^{\dz }(Du^\dz )\|_{L^\fz(U)})-\frac12\ge
\tau_{H^{\dz }}(\|H^{\dz }(Du^\dz )\|_{L^\fz(U)})
 \ge \frac12\left[\frac{\lambda_H(\fz)}{\Lambda_H(\fz)}\right]^2,$$
moreover, we also have
 $$\limsup_{\dz\to0}[\alpha-\tau_{H^{\dz }}(\|H^{\dz }(Du^\dz )\|_{L^\fz(U)})-\frac12]
  \ge\limsup_{\dz\to0}[\alpha-\tau_{H}(\|H(Du )\|_{L^\fz(U)})-\frac12] .$$
With the aid of this and Theorem \ref{THM5.1}, by some argument  similar  to the  above case $\alpha<1/2$, we are able to get the desired estimates \eqref{eq5.yy1}\&\eqref{eq5.yy2}.
Here we omit  the details.
 The proof of Theorem \ref{THM5.2} is complete.
 \end{proof}

\subsection{$L^t_\loc$-convergence of $[H^{\dz_j}(Du^{\dz_j})]^\alpha$ and
$W^{1,t}_\loc$-convergence of $u^{\dz_j}$ }

In this subsection, we prove the following  Sobolev convergence.
 \begin{thm} \label{THM5.3}
 \begin{enumerate}
\item[(i)]
For any $\alpha\ge  1/2 $, we have
 $[H^{\dz } (Du^{\dz } )]^{\alpha}\rightarrow [H(Du )]^{\alpha}$ in $L^t_{\loc}(U)$ for all $t\in [1,\infty)$ and weakly in $W^{1,2}_{\loc}(U)$ as $\dz\to0$.

\item[(ii)] Let  $\{\dz_j\}_{j\in\nn}$  be   as in Lemma 5.2.   For any $ 1/2-\tau_{H}(\|H(Du)\|_{L^\fz(U)})<\alpha<1/2$, we have
 $[H^{\dz_j} (Du^{\dz_j} )]^{\alpha}\rightarrow [H(Du )]^{\alpha}$ in $L^t_{\loc}(U)$ for all $t\in [1,\infty)$ and weakly in $W^{1,2}_{\loc}(U)$ as $j\to\fz$.

 \item[(iii)] If $H\in C^1(\rr^2)$ additionally,  $ u^{\dz }\to u$ in $W^{1,t}_{\loc}(U)$ and
 $D_pH^{\dz }(Du^{\dz })\to D_pH(Du)$ in $L^t_\loc (U)$ for all $t\in [1,\infty)$ as $\dz\to 0$.

 \end{enumerate}
 \end{thm}
 \begin{proof} We borrow some ideas used in the proof of  Theorem \ref{THM4.2},
but due to some technical differences caused by  $H\in C^0(\rr^2)$ or $H\in C^1(\rr^2)$, we give the details.

  {\it Proof of   (i).}
By Theorem \ref{THM5.2} (i) and the weak compactness of $W^{1,2}_{\loc}(U)$, we know that, up to some subsequence,
 $[H^{\dz } (Du^{\dz } )]^{\alpha}\rightarrow f^{(\alpha)}$ in $ L^t_{\loc}(U)$ with $ t\in[1,\infty)$
and weakly in $W^{1,2}_{\loc}(U)$ as $\dz\to0$.
To prove (i), it suffices to show that $    H(Du)=f^{(1)}$ almost everywhere.
Indeed,   this implies that
  $$f^{(\alpha)}=\lim_{\dz\to0} [H^{\dz }(Du^{\dz })]^{\alpha}=[\lim_{\dz\to0}  H^{\dz }(Du^{\dz })]^{\alpha} = [f^{(1)}]^{\alpha}
=[H(Du)]^{\alpha}$$ almost everywhere as desired.

Below we  prove $    H(Du)=f^{(1)}$ almost everywhere. Assume that $u$ is
differentiable at $\bar x$,  and also assume that $\bar x$ is Lebesgue point of $f^{(1)}$ and $H(Du)$. For any $\zeta\in (0,1)$,  there exists $r_{\zeta,\bar x}\in (0,\frac18{\rm dist}(\bar x,\partial U) )$  such that for any $r\in (0,r_{\zeta,\bar x})$, we have
$$\sup_{B(\bar x,2r)}\frac{|u(x)-u(\bar x)-\langle Du(\bar x), x-\bar x\rangle|}{r}\le\zeta.$$
By  Theorem \ref{THM5.1}, for arbitrary $r\in(0,r_{\zeta,\bar x})$, there exists
$\dz_{\zeta,\bar x,r}\in(0,1)$ such that for all
$0<\dz<\\dz_{\zeta,\bar x,r}$, we have
$$\sup_{B(\bar x,2r)}\frac{|u^{\dz } (x)-u^{\dz } (\bar x)-\langle Du(\bar x), x-\bar x\rangle|}{r}\le2\zeta.$$
Applying  Lemma \ref{LEM4.3} with $ H^{\dz }$, $u^{\dz }$ and
 $F^{\dz } =u^{\dz } (\bar x)+\langle Du(\bar x),\cdot-\bar x\rangle$, we have
\begin{align*}
  &\mint-_{B(\bar x,r) }\langle D_pH^{\dz }(Du^{\dz } ),Du^{\dz } -DF\rangle ^2 \,d x\\
  &\quad\le C\frac{[\Lambda_{H^{\dz }} (\fz)]^2}{ \lambda_{H^{\dz }} (\fz) }
   \|H^{\dz }(Du^{\dz })\|_{L^\fz({B(\bar x,2r)})}    \left[\mint-_{ B(\bar x,2r) } \left((|Du^{\dz } |+|DF|)^2\frac{(u^{\dz } -F)^2}{r^2} + \frac{(u^{\dz } -F)^4}{r^4} \right)\,d x\right]^{1/2}.
 \end{align*}
  By  Lemma \ref{LEM7.1}, $\lambda_{H^{\dz }}\ge\lambda_H(\fz)$
 and Theorem \ref{THM5.1},
\begin{align}\label{eq5.6}
 \|Du^{\dz }\| _{L^\fz(U)}
   \le \frac {\sqrt2} {[\lambda_H (\fz)]^{1/2}}
\|H^{\dz}(Du^{\dz} )\|^{1/2}_{L^\fz(U)}
 & \le \left[\frac {\Lambda_H(\fz)}{\lambda_H(\fz)}\right]^{1/2} \|u\|_{C^{0,1}(\partial U)}
\end{align}
By Theorem \ref{THM5.1} again and  by $\lambda_H^{\dz }\ge \lambda_H(\fz) $  and
$\Lambda_H^{\dz }\le \Lambda_H(\fz)$, for  $\dz\in(0,1]$ we obtain
\begin{align*}
&\mint-_{B(\bar x,r) }\langle D_pH^{\dz }(Du^{\dz } ),Du^{\dz } -DF\rangle ^2 \,d x\\
  &\quad\le  C\frac{[\Lambda_{H } (\fz)]^3}{ \lambda_{H } (\fz) } \|u\|^2_{C^{0,1}(\partial U)}\left(\left[\frac {\Lambda_H(\fz)}{\lambda_H(\fz)}\right]^{1/2} \|u\|_{C^{0,1}(\partial U)} + |Du(\bar x)|+1\right) \zeta  =:L_1\zeta.
 \end{align*}

On the other hand, by $\lambda_H^{\dz }\ge \lambda_H(\fz) $ and the strong convexity of
 $H^{\dz }$, we have
\begin{equation}\label{eq5.x8}
\langle D_p H^{\dz } (Du^{\dz }(x))-D_p H^{\dz } (Du(\bar x)),Du^{\dz } (x)-Du(\bar x)\rangle\ge \lambda_H(\fz)|Du^{\dz }(x)-Du(\bar x)|^2.
\end{equation}
Moreover, for any  $x\in  B(\bar x, 2r)$ we have
$$
 |H^{\dz } (Du^{\dz } (x))-H (Du(\bar x))| =|D_p H^{\dz }(\theta Du^{\dz } (x)+  (1-\theta) Du(\bar x))||Du^{\dz } (x)-Du(\bar x)|
$$
for some $\theta\in [0,1]$.
Since
$$| \theta Du^{\dz } (x)+  (1-\theta) Du(\bar x)|\le \sup \|Du^\dz\|_{L^\fz(U)} +|Du(\bar x)|=:L_0<\fz,$$
by Lemma \ref{LEM7.1} (iv) we have
$$|D_p H^{\dz } (\theta Du^\dz (x)+  (1-\theta) Du(\bar x))|\le
 {  \Lambda_H(\fz) }  L_0,$$
which implies that
\begin{equation}\label{eq5.8}
|H^{\dz } (Du^\dz(x))-H (Du(\bar x))|
\le  \Lambda_H(\fz)   L_0 |Du^{\dz } (x)-Du(\bar x)|.
\end{equation}
By this and \eqref{eq5.x8} one has
\begin{align}
&\frac{\lambda_H(\fz)}{[\Lambda_H(\fz) L_0]^2}\mint{-}_{B(\bar x,r)}|  H^{\dz } (Du^{\dz } ) - H^\dz(Du(\bar x))|^2\,d x \\
&\quad\le \lambda_H(\fz)\mint{-}_{B(\bar x,r)}|Du^{\dz } (x)-Du(\bar x)|^2\,d x\nonumber\\
&\quad\le  \mint{-}_{B(\bar x,r)} \langle D_p H^{\dz } (Du^{\dz_j} ),Du^{\dz } -Du(\bar x)\rangle  \,d x-  \mint{-}_{B(\bar x,r)}\langle D_p H^{\dz } (Du (\bar x) ),Du^{\dz }-Du(\bar x)\rangle\,d x\nonumber\\
&\quad\le [L_1\zeta ]^{1/2}-   \mint{-}_{B(\bar x,r)}\langle D_p H^{\dz } (Du (\bar x) ),Du^{\dz } -Du(\bar x)\rangle\,d x.\nonumber
  \end{align}
Since  $H^{\dz } (Du^{\dz })\rightarrow f^{(1)} $ in
$L^2 (B(\bar x,r))$ and $H^ {\dz }(Du (\bar x))\to H(Du(\bar x))$ as $\dz\to0$,
this yields
\begin{align*}
&\frac{\lambda_H(\fz)}{[\Lambda_H(\fz) L_0]^2}\mint{-}_{B(\bar x,r)}|f^{(1)}- H(Du(\bar x))|^2\,d x\\
&\quad\le [L_1\zeta]^{1/2}- \liminf_{\dz\to0} \mint{-}_{B(\bar x,r)}\langle D_p H^{\dz } (Du (\bar x) ),Du^{\dz } -Du(\bar x)\rangle\,d x.
\end{align*}
Observe that $H\in C^{0,1} (\rr^2)$ implies that
 $\{D_pH^{\dz }(Du(\bar x))\}_{\dz\in(0,1]}$ is bounded,
  and hence, up to some subsequence,
  $D_pH^{\dz }(Du(\bar x))\to \bar p$ for some $\bar p\in\rr^2$ as $\dz\to0$.
By $Du^\dz\rightharpoonup Du$ weakly in $L^2_\loc(U)$ as $\dz\to0$, we obtain
\begin{align*}
&\frac{\lambda_H(\fz)}{[\Lambda_H(\fz) L_0]^2}\mint{-}_{B(\bar x,r)}|f^{(1)}- H(Du(\bar x))|^2\,d x \le  [L_1\zeta]^{1/2}-  \mint{-}_{B(\bar x,r)}\langle \bar p,Du  -Du(\bar x)\rangle\,d x.
\end{align*}
Letting $r\to0$ and $\zeta\to0$ in order, we  conclude that $f^{(1)}(\bar x) = H(Du(\bar x))$ as desired.

\medskip
 \noindent {\it Proof of   (ii).}  Let
$\{\dz_j\}_{j\in\nn}$ be as in Lemma 5.2. By Theorem \ref{THM5.3} (i) we have
$[H^{\dz_j}(Du^{\dz_j})]^{\alpha}\to [H(Du)]^{\alpha}$ almost everywhere.
From this, Theorem \ref{THM5.2} (ii) and the weak compactness of $W^{1,2}_{\loc}(U)$, we know that,
 $[H^{\dz_j} (Du^{\dz_j} )]^{\alpha}\rightarrow [H(Du)]^{\alpha}$ in $ L^t_{\loc}(U)$ with $ t\in[1,\infty)$
and weakly in $W^{1,2}_{\loc}(U)$ as $j\to\fz$.

\medskip
 \noindent {\it Proof of   (iii).}
Since $H^{\dz }\in C^\fz(\rr^2)$ is strongly convex and $\lambda_{H^{\dz }}(\fz)\ge \lambda_H(\fz) $, we have
$$H^{\dz }(Du^{\dz })-H^{\dz}(Du)\ge \langle D_p H^{\dz }(Du),Du^{\dz }-Du\rangle+\frac{\lambda_H(\fz)}{2}|Du^{\dz }-Du|^2$$
 almost everywhere in $U$.
Thus, for any $V\Subset U$,
\begin{align*}
 \frac{\lambda_H(\fz)}{2}\int_V|Du^{\dz}-Du|^2\,dx
 & \le \int_V\langle D_pH^{\dz}(Du),Du-Du^{\dz }\rangle\,dx +\int_V|H^{\dz}(Du^{\dz })-H^{\dz}(Du)|\,dx\\
& \le \int_V\langle D_pH^{\dz}(Du),Du-Du^{\dz }\rangle\,dx + \int_V|H^{\dz }(Du^{\dz })-H (Du)|\,dx\\
&\quad\quad+\int_V|H^{\dz }(Du )-H (Du)|\,dx.
\end{align*}
By Theorem \ref{THM5.3} (i) we have
 $H^{\dz }(Du^{\dz })\to H(Du)$ in $L^t_\loc (U)$ for all $t\ge1$ as $\dz\to0$.
By $H^\dz\to H$ locally uniformly as $\dz\to0$ we have
and $H^\dz(Du)\to H(Du)$ in $L^t_\loc (U)$ for all $t\ge1$ as $\dz\to0$.
Moreover, note that $H\in C^1(\rr^2)$ implies that $D_pH^\dz\to D_pH$ locally uniformly in $\rr^2$ as $\dz\to0$.
Thus $D_pH^\dz(Du)\to D_pH(Du)$  in $L^t_\loc(U)$ for all $t\ge1$ as $\dz\to0$.
By $Du^\dz\to Du$ weakly in $L^2_\loc (U)$, one has
$$\int_V\langle D_pH^\dz(Du),Du-Du^\dz\rangle\,dx\to 0 \quad {\rm as}\ \dz\to0.$$
Therefore we conclude
  $Du^{\dz }\to Du$ in $L^2_\loc (U)$, and hence, by
 \eqref{eq5.6},    in $L^{ t}_\loc (U)$ for any $t\ge1$ as $\dz\to0$.

Finally, write
\begin{align*}|D_p H^\dz(Du^\dz)-D_p H(Du)|&\le|D_p H^\dz(Du^\dz)-D_p H^\dz(Du)|+ |D_p H^\dz (Du)-D_p H(Du)|.
  \end{align*}
Note that    $D_p H^\dz (Du)\to D_p H(Du) $ almost everywhere  as $\dz\to0$.
By Lemma \ref{LEM7.1} (iv) and Theorem \ref{THM5.1}, $D_p H^\dz (Du)\in L^\fz(U)$ uniformly in $\dz\in(0,1]$.
Therefore, by the Lebesgue theorem, we have $D_p H^\dz (Du)\to D_p H(Du) $
 in $L^t_\loc(U)$ for all $t\ge1$ as $\dz\to0$.
Moreover,
by Lemma \ref{LEM7.3}(i), we have
\begin{align*} |D_p H^\dz(Du^\dz)-D_p H^\dz(Du)|
&\le|[D^2_{pp} H^\dz(\theta Du^\dz+(1-\theta)Du)](Du^\dz-Du)|
 \le \Lambda_{H }(\fz)|Du^\dz-Du |
  \end{align*}
  for some $\theta\in[0,1]$.  This, together with
   $Du^{\dz }\to Du$ in $L^t_\loc (U)$  for all $t\ge1$ as $\dz\to0$, yields
   $D_p H^{\dz }(Du^{\dz })\rightarrow D_p H(Du)$ in $L^t_{\loc}(U)$ for all
 $t\ge 1$ as $\dz\to0$. This completes the proof of Theorem \ref{THM5.3}.
\end{proof}

\begin{rem}
\rm When $H\in C^0(\rr^2)$ satisfies (H1')\&(H2), we do not know if Theorem \ref{THM5.3} (ii) holds or not.
Note that in the proof of Theorem \ref{THM5.3} (ii) we do use  $H\in C^1(\rr^2)$   to guarantee
$D_pH^\dz(Du)\to D_pH(Du)$  in $L^t_\loc(U)$ for all $t\ge1$ as $\dz\to0$. Note that $u\in C^{0,1}(U)$.
Under   $H\in C^0(\rr^2)$, we do not know if $\{D_pH^\dz(Du)\}_{\dz\in(0,1]}$ contains a Cauchy sequence in $L^t_\loc(U)$ for all $t\ge1$.
\end{rem}

\section{Proofs of Theorems \ref{THM1.1}\&\ref{THM1.2} and Corollary \ref{COR1.3}}
Considering Remark \ref{REM1.8},
We  only need to prove Theorem  \ref{THM1.1} (resp. \ref{THM1.2}) and Corollary \ref{COR1.3} when
   $H\in C^0(\rr^2)$ (resp. $H\in C^1(\rr^2)$ ) satisfies  (H1')\&(H2).
    We always let $\{H^\dz\}_{\dz\in(0,1]}$ be the smooth approximation to $H$ as in Appendix A.

 Theorem \ref{THM1.1} then follows from Theorem  \ref{THM5.3}\& \ref{THM5.2}.
\begin{proof}
[Proof of Theorem \ref{THM1.1}.]  Suppose that $H\in C^0(\rr^2)$ satisfies (H1')\&(H2).
Let $ \{\dz_j\}_j$ be the sequence given in Lemma 5.2.
By Theorem \ref{THM5.3} (i)\&(ii),
for $\alpha>1/2-\tau_{H}(\|H(Du)\|_{L^\fz(U)})$, one has
$[H^{\dz_j}(Du^{\dz_j})]^{\alpha}\to [H(Du)]^\alpha$ weakly in $W^{1,2}_\loc (U)$ and
in $L^t_\loc (U)$ for all $t\ge 1$ as $j\to\fz$.
This implies that $[  H(Du)]^\alpha\in W^{1,2}_\loc(U)$.
Moreover, applying \eqref{eq5.yy2} and \eqref{eq5.y2} in Theorem \ref{THM5.2}, for all $V\Subset W\Subset U$  with
$\dist(V,\partial W) =\dist(W,\partial U)=\frac12 \dist(V,\partial U)  $,
 we have
\begin{align*}
\|D[H(Du)]^\alpha\|^2_{L^2(V) }&\le
 \liminf_{j\to\fz}\|D[H^\dz(Du^{\dz_j})]^{\alpha}\|_{L^2(V)}^2\\
 &
 \le\frac{C\alpha^2(\alpha+1)}{[\alpha+\tau_{H }(\|H (Du ) \|_{L^\fz(U)})-\frac12]^2}\left[\frac{ \Lambda_{H }( \|H (Du ) \|_{L^\fz(  U)})}{\lambda_{H }( \|H (Du ) \|_{L^\fz(  U)})}\right]^2
 \\ &\quad\times
 \frac1{[\dist(V,\partial W)]^2} \liminf_{j\to\fz}\int_{W}   [H^{\dz_j}(Du^{\dz_j} )]^{2\alpha }\,dx\\
  &
 \le \frac{C\alpha^2(\alpha+1)}{[\alpha+\tau_{H }(\|H (Du ) \|_{L^\fz(U)})-\frac12]^2}\left[\frac{ \Lambda_{H }( \|H (Du ) \|_{L^\fz(  U)})}{\lambda_{H }( \|H (Du ) \|_{L^\fz(  U)})}\right]^2
 \\ &\quad\times
 \frac1{[\dist(V,\partial U)]^2}  \int_{U}   [H (Du  )]^{2\alpha }\,dx.
\end{align*}
Thus \eqref{e1.x1} follows. This completes the proof of Theorem \ref{THM1.1}.
\end{proof}

Next, we conclude   Corollary \ref{COR1.3} from
  Theorem  \ref{THM1.1}, Lemma \ref{LEM7.10x}
 and Lemma \ref{LEM7.5}.

\begin{proof} [Proof of Corollary \ref{COR1.3}.]  Suppose that
$H\in C^0(\rr^2)$   satisfies  (H1')\&(H2).
 We only need to consider
the case $1/2-\tau_H(0)<\alpha<1/2$.

Firstly, we show that
for any $ x\in \Omega$ with $Du(  x)=0$,  there exists
  $ 0<r_{  x,\alpha}<\dist(x,\partial \Omega)/8 $ so that
$|D[H(Du)]^\alpha|\in   L^2 (B( x,r_{  x,\alpha}))$.
Indeed, by the right continuity and monotonicity of $\tau_H$  (see Lemma \ref{LEM7.5}),
we can find $R_\alpha>0$ such that $\alpha>1/2-\tau_H $ in $[0,R_\alpha]$.
Choose $ 0<r_{  x,\alpha}<\dist(x,\partial \Omega)/8 $ so that
$\|H(Du)\|_{L^\fz(B(x,2r_{x,\alpha}))}<R_\alpha$. By Theorem \ref{THM1.1}, we know that
$[H(Du)]^\alpha\in W^{1,2}_\loc(B(x,2r_{x,\alpha}))$ and hence
 $\mbox{$[H(Du)]^\alpha\in W^{1,2} (B(x, r_{x,\alpha}))$}.$

Next, for any $U\Subset\Omega$, note that the set $ U_0=\{z\in \overline  U: Du(z)=0\}$ is compact and covered by
the union of $\{B(x,r_{x,\alpha})\}_{x\in U_0}$.
So we can find  $\{x_j\}_{j=1}^N\subset U_0$   for some $N<\fz$
  so that
 $U_0\subset  \cup_{j=1}^NB(  x_j,r_{  x_j,\alpha}) $. Thus
 $$\mbox{$[H(Du)]^\alpha\in W^{1,2} (\bigcup_{j=1}^NB(  x_j,r_{  x_j,\alpha}))$}.$$
By $u\in C^1(\Omega)$, there is $r_0>0$ such that  the closure of the open set
$U\setminus (\cup_{j=1}^N\overline {B(  x_j, r_{  x_j,\alpha}/2))}$ is contained in the open set $$
\Omega_{>r_0}:= \{x\in \Omega: H(Du(x))>r_0\} .$$
Applying Theorem \ref{THM1.1} to each component of $\Omega_{>r_0}$, we have $|D[H(Du)] |\in L^2_\loc(\Omega_{>r_0}).$
This allows us to get
$$|\alpha [H(Du)]^{\alpha-1} |D[H(Du)]|\le \alpha r_0 ^{\alpha-1}|D[H(Du)] |\in L^2_\loc(\Omega_{>r_0}).$$
Thus, by
\begin{equation}\label{exx} D[H(Du)]^\alpha = \alpha [H(Du)]^{\alpha-1}  D[H(Du)]\quad{\rm in}\ \Omega_{>r_0}\end{equation}
in distributional sense, we know that
  $$\mbox{
       $|D[H(Du)]^\alpha| \le \alpha r_0 ^{\alpha-1}|D[H(Du)] |$ almost everywhere   in $\Omega_{>r_0}$.} $$
Therefore, $$ [H(Du)]^\alpha \in   W^{1,2} (U\setminus (\cup_{j=1}^N\overline {B(  x_j, r_{  x_j,\alpha}/2))}).$$
  We then conclude
$ [H(Du)]^\alpha \in   W^{1,2} (U )$ as desired.

If $H\in C^2(\rr^2)$ and $H^{1/2}$ is convex additionally, by Lemma \ref{LEM7.2} we have
$\tau_H\equiv 1/2$, and hence, by Theorem \ref{THM1.1}, one gets \eqref{e1.x1}  for all $\alpha>0$.
This completes the proof of Corollary \ref{COR1.3}.
\end{proof}

Finally we prove Theorem  \ref{THM1.2} by using Theorem \ref{THM1.1}, Theorem \ref{THM5.2} and Corollary \ref{COR1.3}.
\begin{proof}[Proof of Theorem  \ref{THM1.2}.]
Suppose that $H\in C^1(\rr^2)$ satisfies (H1)\&(H2).
Let $\Omega\subset\rr^2$ be any domain and $u\in AM_H(\Omega)$.
Similarly as in the proof of Theorem \ref{THM1.1},
fix arbitrary $U\Subset  \Omega$.
 Let $U\Subset \wz U\Subset \Omega$ and note
 $\|H(Du)\|_{L^\fz(\wz U)}<\fz$.
Letting $R:= \|H(Du)\|_{L^\fz(\wz U)}+1 $, by Lemma \ref{LEM7.8}
there exists $\wz H\in C^1(\rr^2)$   satisfying    (H1') and
 (H2) such that $\wz H=H$ in $H^{-1}([0, R+1])$,
 $\tau_{H}=\tau_{\wz H} $, $\lambda_{H}=\lambda_{\wz H}$ and
 $\Lambda_{H}=\Lambda_{\wz H}$ in $[0, R+1)$.
Note that $u\in AM_{ \wz H} (\wz U)$.
By abuse of notation, we write $\wz H$ as $H$ below.

Moreover, let $\{H^\dz\}_{\dz\in(0,1]}$ be a smooth approximation to $H$ as in the Appendix A.
For $\dz\in(0,1]$, let $u^\dz\in AM_{H^\dz}(U)$ with $u^\dz=u$ on $\partial U$.
Let $\{\dz_j\}_j$ be the sequence given in Lemma 5.2.

\noindent {\it Proof of   (i).}
By Theorem \ref{THM5.2} (iii), we know that
$Du^{\dz_k}\to Du$ in $L^t_\loc (U)$ for any $t\ge 1$ as $k\to\fz$, we obtain
\begin{align*}
\int_U-\det D^2u \phi\,dx&=\frac 1 2\int_U[u_{x_i}u_{x_j}\phi_{x_ix_j}+|Du|^2\phi_{x_ix_i}]\,dx\\
&=\frac 1 2\lim_{k\rightarrow \fz}\int_U [u^{\dz_k}_{x_i}u^{\dz_k}_{x_j}\phi_{x_ix_j}+
|Du^{\dz_k}|^2\phi_{x_ix_i}]\,dx\\
&=\lim_{k\rightarrow\fz}\int_U-\det D^2u^{\dz_k} \phi\,dx\quad\forall \phi\in C_c^\fz(U).
\end{align*}
Applying Theorem  \ref{THM1.2} (i) to $H^\dz\in C^\fz(\rr^2)$,  and noting
$[H^{\dz_k}(Du^{\dz_k})]^{1/2}\to [H(Du)]^{1/2}$ weakly in $W^{1,2}_\loc (U)$ as given in Theorem \ref{THM5.3} (i)\&(ii),
we have
\begin{align*}\lim_{k\to\fz}\int_U-\det D^2u^{\dz_k} \phi\,dx&\ge
  4\limsup_{k\to\fz}
 \frac{\tau_{H^{\dz_k}}( \|H^\dz(Du^{\dz_k})\|_{L^\fz(U)} )}
 {\Lambda_{H^{\dz_k}}( \|H^{\dz_k}(Du^{\dz_k})\|_{L^\fz(U)} )}\int_U|D[H^{\dz_k}(Du^{\dz_k})]^{1/2}|^2\phi\,dx\\
 &\ge
  4\limsup_{k\to\fz}
 \frac{\tau_{H^{\dz_k}}( \|H^\dz(Du^{\dz_k})\|_{L^\fz(U)} )}
 {\Lambda_{H^{\dz_k}}( \|H^{\dz_k}(Du^{\dz_k})\|_{L^\fz(U)} )}\int_U|D[H (Du )]^{1/2}|^2\phi\,dx
    \end{align*}
for all $0\le \phi\in C^\fz_c (U)$.
 By Lemma 5.2, we have
$$
\lim_{j\to \fz} \tau_{H^{\dz_j}}(\|H^{\dz_j}(Du^{\dz_j})\|_{L^\fz(U)})\ge \tau_{H}( \|H(Du)\|_{L^\fz(U)} ).$$
By  Theorem \ref{THM5.1} again and Lemma \ref{LEM7.3}, using the decrease and right-continuity of $\Lambda_H$, further  we have
 \begin{align*}
\limsup_{k\to\fz}
 \frac{\tau_{H^{\dz_k}}( \|H^{\dz_k}(Du^{\dz_k})\|_{L^\fz(U)} )}
 {\Lambda_{H^{\dz_k}}( \|H^{\dz_k}(Du^{\dz_k})\|_{L^\fz(U)} )}&\ge
 \frac{ \limsup_{k\to \fz}\tau_{H^{\dz_k}}( \|H^{\dz_k}(Du^{\dz_k})\|_{L^\fz(U)} )}
 { \limsup_{k\to \fz}\Lambda_{H^{\dz_k}}( \|H^{\dz_k}(Du^{\dz_k})\|_{L^\fz(U)} )}\ge \frac{\tau_{H}( \|H(Du)\|_{L^\fz(U)} )}
 {\Lambda_{H}( \|H(Du)\|_{L^\fz(U)} )}.
   \end{align*}
   Thus
 \begin{align*}\int_U-\det D^2u \phi\,dx &\ge 4
 \frac{\tau_{H}( \|H(Du)\|_{L^\fz(U)} )}
 {\Lambda_{H}( \|H(Du)\|_{L^\fz(U)} )}\int_U|D[H(Du)]^{1/2}|^2\phi\,dx\ge0 \quad \forall 0\le \phi\in C^\fz_c (U),
  \end{align*}
 which  implies that   $-\det D^2u\, dx$ is a nonnegative Radon measure with
\begin{align*}-\det D^2u\, dx&\ge4
 \frac{\tau_{H}( \|H(Du)\|_{L^\fz(U)} )}
 {\Lambda_{H}( \|H(Du)\|_{L^\fz(U)} )}|D[H(Du)]^{1/2}|^2\,dx.
  \end{align*}
On the other hand, for any ball $V\Subset U\Subset \Omega $, we choose a cut-off function $\phi\in
C^\fz_c(U)$ as in \eqref{eq3.4} to obtain
\begin{align*}
\int_{V}-\det D^2u\phi\,dx&\le \int_U-\det D^2u\phi\,dx\\
&= \frac 1 2\int_U[u_{x_i}u_{x_j}\phi_{x_ix_j}+|Du|^2\phi_{x_ix_i}]\,dx\le C\frac1{[\dist(V,\partial U)]^2}\int_U|Du|^2\,dx.
  \end{align*}
\noindent {\it Proof of   (ii).}
 We divide the proof into 2 steps.

{\it Step 1.}
If $\alpha\ge 1/2$ and $U=\Omega$, or  $\alpha>1/2-\tau_H(\|H(Du)\|_{L^\fz (U)})$ with $U\Subset\Omega$, by Theorem \ref{THM5.2} (iii),   one has that $D_pH^{\dz_k}(Du^{\dz_k})\to D_pH(Du)$ in $L^2_\loc (U)$  and also,
$D[H^{\dz_k}(Du^{\dz_k})]^\alpha\to [DH(Du)]^\alpha$ weakly in $L^2_\loc (U)$ as $k\to\fz$.  Thus
$$
 \int_{ U}\langle D[H(Du)]^{\alpha},D_p H(Du)\rangle\phi\,d x  =\lim_{k\to\fz}
 \int_{U}\langle D[H^{\dz_k} (Du^{\dz_k} )]^{\alpha},D_p H^{\dz_k} (Du^{\dz_k} )\rangle\phi\,d x \quad
$$
for any $ \phi \in C^\fz_c(U)$.
Applying Theorem \ref{THM1.2} (ii) to $H^{\dz_k}$, we have $$\mbox{$\langle D[H^{\dz_k} (Du^{\dz_k} )]^{\alpha},D_p H^{\dz_k} (Du^{\dz_k} )\rangle=0$  almost everywhere in $U$}.$$ Thus,
  $$
 \int_{ U}\langle D[H(Du)]^{\alpha},D_p H(Du)\rangle\phi\,d x  =0
$$
for any $ \phi \in C^\fz_c(U)$, that is,
 $\langle D[H(Du)]^{\alpha},D_p H(Du)\rangle=0$
 almost everywhere in $U$.

 {\it Step 2.}
If $1/2-\tau_H(0)<\alpha< 1/2 $, for any $x\in \Omega$ with $H(Du(x))=0$,
let $R_\alpha>0$ and $r_{x,\alpha}$ be as in the proof of Corollary \ref{COR1.3}.
Since $\alpha>1/2-\tau_H(\|H(Du)\|_{L^\fz(B(x,r_{x,\alpha}))})$, by Step 1  with $U=B(x,r_{x,\alpha})$  we know that
 $$\mbox{$\langle D[H(Du)]^{\alpha},D_p H(Du)\rangle=0$
 almost everywhere in $B(x,r_{x,\alpha})$.  }$$

 For any $U\Subset\Omega$, let $\{x_j\}_{j=1}^N$ be as in the proof of Corollary \ref{COR1.3}.
 One then has
 $$\mbox{$\langle D[H(Du)]^{\alpha},D_p H(Du)\rangle=0$
 almost everywhere in $\bigcup_{i=1}^NB(x_j,r_{x_j,\alpha}) $.}$$
 Let $r_0$ be as in the proof of Corollary \ref{COR1.3}.
By \eqref{exx} and applying Step 1 to $U\cap \Omega_{r_0}$, we know that
$$\mbox{$\langle D[H(Du)]^{\alpha},D_p H(Du)\rangle=0$
 almost everywhere in $U\cap \Omega_{r_0}$.}$$
Since $U=[\cup_{x\in U_0}B(x,r_{x,\alpha}) ]\cup [U\cap \Omega_{r_0}]$, one gets $\langle D[H(Du)]^{\alpha},D_p H(Du)\rangle=0$ almost everywhere in $U$. By the arbitrariness of $U$, we have
 $$\mbox{$\langle D[H(Du)]^{\alpha},D_p H(Du)\rangle=0$ almost everywhere in $\Omega$.}$$
 The proof of Theorem \ref{THM1.2} is complete.
\end{proof}

\section{Proofs of Lemmas  \ref{LEM2.7},   \ref{LEM2.8},    \ref{LEM2.6},   \ref{LEM3.3} \&  \ref{LEM3.4}}

In Section 7.1, we prove  Lemmas \ref{LEM2.7}\&\ref{LEM2.8};
in Section 7.2, we prove Lemmas \ref{LEM2.6},   \ref{LEM3.3}\&\ref{LEM3.4}.
\subsection{Proofs of Lemmas \ref{LEM2.7}\&\ref{LEM2.8}}

\begin{proof}[Proof of Lemma \ref{LEM2.7}]
 Let $\varphi=  [H (Du^\epsilon)+\sigma]^{2\alpha-1} \phi^2$,
 where $\phi\in C^\fz_c(W)$, and $\sigma=0$ if $\alpha\ge 3/2$
 and $\sigma>0$ if $1/2 -\tau_H(\|H(Du^\ez)\|_{L^\fz(W)})<\alpha<3/2$.
 By \eqref{FUN2.3} and
  $$[H (Du^\epsilon)+\sigma]^{2\alpha-1}\ge [H (Du^\epsilon)+\sigma]^{2\alpha-2} H(Du^\ez),$$     we have
  \begin{align*} I(u^\ez,\vz)
 &=8\int_U \wz \tau_H(Du^\ez)  \langle  {D^2_{pp}H}(Du^\ez)  D[H(Du^\ez)]^{1/2} ,D[H(Du^\ez)]^{1/2} \rangle [H (Du^\epsilon)+\sigma]^{2\alpha-1} \phi^2\,d x\nonumber\\
 &\quad\quad+2\epsilon\int_U \wz \tau_H(Du^\ez)( {\rm div}[D_pH (Du^\ez) ])^2  [ H(Du^\ez)]^{-1} [H (Du^\epsilon)+\sigma]^{2\alpha-1} \phi^2,d x\\
 &\ge   \frac2{\alpha^2}\int_U \wz \tau_H(Du^\ez)  \langle  {D^2_{pp}H}(Du^\ez)  D[H(Du^\ez)+\sz]^{\alpha} ,D[H(Du^\ez)+\sz]^{\alpha} \rangle   \phi^2\,d x\nonumber\\
 &\quad\quad+2\epsilon\int_U \wz \tau_H(Du^\ez)( {\rm div}[D_pH (Du^\ez)])^2    [H (Du^\epsilon)+\sigma]^{2\alpha-2} \phi^2\,d x\\
 &\ge \frac2{\alpha^2}\tau_H(\|H(Du^\ez)\|_{L^\fz(W)})K_1+ 2\epsilon\tau_H(\|H(Du^\ez)\|_{L^\fz(W)})K_2,
\end{align*}
where and below, we write
\begin{align}\label{eq6.y1}
K_1&:= \int_U    \langle  {D^2_{pp}H}(Du^\ez)  D[H(Du^\ez)+\sz]^{\alpha} ,D[H(Du^\ez)+\sz]^{\alpha} \rangle   \phi^2\,d x\\
&\ge \lambda_H(\|H(Du^\ez)\|_{L^\fz(W)})\int_U    |D[H(Du^\ez)+\sz]^{\alpha}|^2   \phi^2\,d x\nonumber
\end{align}
and
$$
K_2:= \int_U ( {\rm div}[D_pH (Du^\ez) ])^2    [H (Du^\epsilon)+\sigma]^{2\alpha-2} \phi^2\,d x.$$

 On the other hand, we have $\varphi\in W^{1,2}_c(U)$ and
$$\varphi_{x_i}=(2\alpha-1) \phi^2u^\ez_{x_kx_i}H_{p_k}(Du^\ez)[H(Du^\ez) +\sigma]^{2\alpha-2}+ 2\phi \phi_{x_i}[ H(Du^\ez) +\sigma]^{2\alpha-1}\quad\mbox{in $U$} $$
 for $i=1,2$.
By   \eqref{FUN2.4} and integration by parts, we   write
\begin{align*}
I(u^\ez,\varphi)&=-\int_U \langle  {D^2_{pp}H}(Du^\ez)   D[H(Du^\ez)]  -{\rm div}[ {D_pH}(Du^\ez)  ]    {D_pH}(Du^\ez)  ,  D\varphi\rangle \,d x\\
&=(2\alpha-1)\int_U{\rm div}[ {D_pH}(Du^\ez) ] H_{p_i}(Du^\ez) u^\ez_{x_kx_i}H_{p_k}(Du^\ez)[H(Du^\ez) +\sigma]^{2\alpha-2}\phi^2\,dx\\
  &\quad+2\int_U {\rm div}[ {D_pH}(Du^\ez) ] H_{p_i}(Du^\ez) \phi_{x_i}[ H(Du^\ez) +\sigma]^{2\alpha-1 }\phi\,dx\\
    &\quad-(2\alpha-1)\int_U   H_{p_j}(Du^\ez)H_{p_ip_l}(Du^\ez)u^\epsilon_{x_lx_j} u_{x_kx_i}H_{p_k}(Du^\ez)[H(Du^\ez) +\sigma]^{2\alpha-2}\phi^2\,d x\\
  &\quad-2\int_U   H_{p_j}(Du^\ez)H_{p_ip_l}(Du^\ez)u^\epsilon_{x_lx_j} \phi_{x_i}[ H(Du^\ez) +\sigma]^{2\alpha-1}\phi \,d x\\
&=: J_1+\cdots+J_4.
\end{align*}

Note that
\begin{align*}J_3&=- (2\alpha-1)\int_U \langle D^2_{pp}H(Du^\ez)D[H(Du^\ez)], D[H(Du^\ez)] \rangle
  [H(Du^\ez) +\sigma]^{2\alpha-2}\phi^2 \,dx\\
  &=- (2\alpha-1)\frac1{\alpha^2}\int_U \langle D^2_{pp}H(Du^\ez)D[H(Du^\ez)+\sz]^\alpha, D[H(Du^\ez)+\sz]^\alpha\rangle
  \phi^2   \,dx\\
  &=  - (2\alpha-1)\frac1{\alpha^2}K_1.
  \end{align*}
Since  $$H_{p_i}(Du^\ez)u^\ez_{x_kx_i}H_{p_k}(Du^\ez)=\mathscr A_H[u^\ez]=-\ez\,{\rm div\,}[{D_pH}(Du^\ez)],$$  we have
$$J_1=-\ez(2\alpha-1)\int_U ({\rm div}[ {D_pH}(Du^\ez) ])^2  [H(Du^\ez) +\sigma]^{2\alpha-2} \phi^2 \,dx=-\ez(2\alpha-1)K_2.$$
By Young's inequality we obtain
\begin{align*}
J_4&=
-2  \frac1\alpha \int_U \langle D^2_{pp}H(Du^\ez)D[H(Du^\ez)+\sz]^\alpha, D\phi\rangle    [ H(Du^\ez) +\sigma]^{ \alpha }\phi \,d x\\
&\le \eta\frac1{\alpha^2} K_1    + C \eta^{-1}\Lambda_H(\|H(Du^\ez)\|_{L^\fz(W)})\int_U         [ H(Du^\ez) +\sigma]^{2\alpha } |D\phi|^2\,d x.
  \end{align*}
   Regards of $J_2$, by integration by parts, we write
\begin{align*}J_2=&  -2\int_U   H_{p_j}(Du^\ez) [ H_{p_i}(Du^\ez)\phi \phi_{x_i}[ H(Du^\ez) +\sigma]^{2\alpha-1 }]_{x_j}\,dx\\
 = & -2\int_U   H_{p_j}(Du^\ez)   H _{p_ip_s}(Du^\ez)u^\ez_{x_sx_j} \phi_{x_i}[ H(Du^\ez) +\sigma]^{2\alpha-1 }\phi \,dx\\
 & -2 \int_U   H_{p_j}(Du^\ez)  H_{p_i}(Du^\ez)\phi_{x_j} \phi_{x_i}[ H(Du^\ez) +\sigma]^{2\alpha-1 } \,dx\\
 & -2\int_U   H_{p_j}(Du^\ez)   H_{p_i}(Du^\ez)\phi \phi_{x_ix_j}[ H(Du^\ez) +\sigma]^{2\alpha-1 } \,dx\\
 & -2(2\alpha-1)\int_U   H_{p_j}(Du^\ez)   H_{p_i}(Du^\ez)\phi \phi_{x_i}[ H(Du^\ez) +\sigma]^{2\alpha-2 }H_{p_s}(Du^\ez)u^\ez_{x_sx_j}\,dx\\
 =:&J_{21}+\cdots+J_{24}.
  \end{align*}
Observe that   $J_{21}=J_4 $, $J_{22}\le 0$ and by Lemma \ref{LEM7.1} (iv), we have
  \begin{align*}
  J_{23}&\le 2\int_U   | {D_pH}(Du^\ez) |^2|D^2\phi| |\phi| [ H(Du^\ez) +\sigma]^{2\alpha-1 } \,dx\\
  &\le 4\frac{[\Lambda_H(\|H(Du^\ez)\|_{L^\fz(W)})]^2}{\lambda_H(\|H(Du^\ez)\|_{L^\fz(W)})}
  \int_U   |D^2\phi| |\phi| [ H(Du^\ez) +\sigma]^{2\alpha  } \,dx.
  \end{align*}
Regards of $J_{24}$, by H\"older's inequality and Lemma \ref{LEM7.1} (iv), using
  $\lambda_H(\fz)\le\lambda_H\le \Lambda_H\le \Lambda_H(\fz)$, we get
  \begin{align*}J_{24}&=2\ez (2\alpha-1)\int_U   {\rm div} [ {D_pH}(Du^\ez)] \langle  {D_pH}(Du^\ez),D\phi\rangle  \phi [ H(Du^\ez) +\sigma]^{2\alpha-2 } \,dx\\
    &\le 2\ez\eta K_2+  \ez  C\eta^{-1}|2\alpha-1|^2 \int_U   \langle  {D_pH}(Du^\ez),D\phi\rangle ^2    [ H(Du^\ez) +\sigma]^{2\alpha-2 } \,dx\\
  &=2\ez\eta K_2+  \ez  C\eta^{-1}|2\alpha-1|^2\frac{[\Lambda_H(\fz)]^2}{\lambda_H(\fz)}\int_U   |D\phi| ^2    [ H(Du^\ez) +\sigma]^{2\alpha-1 } \,dx.
  \end{align*}

Combining all estimates together,  we obtain
  \begin{align*}
   &\left[2\tau_H(\|H(Du^\ez)\|_{L^\fz( W)})+(2\alpha-1) -2\eta \right] [\frac1{\alpha^2} K_1+\ez K_2]  \nonumber\\
   &\quad\le C \frac{[\Lambda_H(\|H(Du^\ez)\|_{L^\fz(W)})]^2}{\lambda_H(\|H(Du^\ez)\|_{L^\fz(W)})}
  \int_U     [ H(Du^\ez) +\sigma]^{2\alpha  } |D^2\phi|\phi\,dx\nonumber\\
  &\quad\quad+ C \eta^{-1}\ez \frac{[\Lambda_H(\fz)]^2}{\lambda_H(\fz)}
  \int_U     [ H(Du^\ez) +\sigma]^{2\alpha-1 }    |D\phi| ^2 \,dx\nonumber\\
  &\quad\quad +C \eta^{-1}\Lambda_H(\|H(Du^\ez)\|_{L^\fz(W)})\int_U         [ H(Du^\ez) +\sigma]^{2\alpha }|D\phi|^2\,dx. \end{align*}

Since $\alpha>1/2-\tau_H(\|H(Du^\ez)\|_{L^\fz(W)})$,  choosing $$\eta=\frac14[2\tau_H(\|H(Du^\ez)\|_{L^\fz(W)})+(2\alpha-1)],$$
by \eqref{eq6.y1} and $\lambda_H(\fz)\le \lambda_H\le \Lambda_H(\fz)$ we get
      \begin{align}\label{eq6.x1}
     &  \int_{U}|D[H(Du^\ez)+\sz ]^\alpha|^2\phi^2\,dx+
  \ez\alpha^2 \frac1{\Lambda_H(\fz)}\int_{U} ({\rm div}[ {D_pH}(Du^\ez)  ])^2     [H(Du^\ez)+\sz]^{2\alpha-2} \phi^2 \,dx \\
    &\quad\le \frac{ C\alpha^2(\alpha+1)}{[\alpha+\tau_H(\|H(Du^\ez)\|_{L^\fz( W)})-\frac12]^2} \left[\frac{ \Lambda_H(\|H(Du^\ez)\|_{L^\fz(W)})}{\lambda_H(\|H(Du^\ez)\|_{L^\fz(W)})}\right]^2\nonumber \\
    &\quad\quad\quad\times
    \int_U
  [ H(Du^\ez) +\sigma]^{2\alpha  }[|D\phi| ^2+|D^2\phi||\phi|] \,dx\nonumber\\
  &\quad\quad +\frac{C\ez\alpha^2|2\alpha-1|^2}{ [\alpha+\tau_H(\|H(Du^\ez)\|_{L^\fz( W)}) -\frac12]^2}\left[\frac{\Lambda_H(\fz) }{\lambda_H(\fz)}\right]^2
  \int_U
  [ H(Du^\ez) +\sigma]^{2\alpha-1 } |D\phi| ^2 \,dx.\nonumber
  \end{align}

 In the case $ 1/2-\tau_H(\|H(Du^\ez)\|_{L^\fz(W)})<\alpha< 1/2$,
    \eqref{eq6.x1} gives Lemma \ref{LEM2.7} (ii).

   When $\alpha\ge3/2$, taking $\sz=0$, from \eqref{eq6.x1}
   and  $\tau_H\ge\tau_H(\fz)$,  we conclude  Lemma \ref{LEM2.7} (i).

Assume that $1/2\le \alpha<3/2$. Note that
$ [H(Du^\ez)+\sz]^{2\alpha-1} \in L^\fz_\loc(U)$
and $[H(Du^\ez)+\sz]^{2\alpha } \in L^\fz_\loc(U)$
uniformly in $\sz\in(0,1].$
 By choosing some suitable cut-off function $\phi$, using
 $\lambda_H(\fz)\le \lambda_H\le \Lambda_H\le \Lambda_H(\fz)$ and $\tau_H\ge\tau_H(\fz)$, we know that
 $[H(Du^\ez)+\sz]^\alpha\in W^{1,2}_\loc(U )$ uniformly in $\sz\in(0,1].$
By the weak compactness of $W^{1,2}_\loc (U)$, we further have
 $[H(Du^\ez)+\sz]^\alpha\to [H(Du^\ez)]^\alpha$ weakly in $W^{1,2}_\loc(U)$ as $\sz\to0$ (up to some subsequence), and hence $[H(Du^\ez)]^\alpha \in W^{1,2}(U)$.
 Letting $\sz\to0$ in \eqref{eq6.x1} and
 using    $\tau_H\ge\tau_H(\fz)$, we   obtain  Lemma \ref{LEM2.7} (i).
The proof of Lemma \ref{LEM2.7} is complete.
\end{proof}

\begin{proof}[Proof of Lemma \ref{LEM2.8}]
Without loss of generality, we may assume that $F(x)=cx_2$.
Let $$\varphi=\frac12 (u^\epsilon-cx_2)^2\phi^4  ,\quad \forall\phi\in C^\fz_c (\frac34B),\quad \forall B\Subset U.$$
Then $\varphi\in W^{1,2}_c(U)$ and
$$\varphi_{x_i}=2(u^\epsilon-cx_2)^2\phi^3\phi_{x_i}+(u^\epsilon_{x_i}-c\delta_{2i}
)(u^\epsilon-cx_2)\phi^4.$$
By Theorem \ref{LEM2.3}, \eqref{FUN2.2} and \eqref{FUN2.4} we obtain
\begin{align*}
0\le I(u^\ez,\vz)&= -\int_U H_{p_j}(Du^\ez)H_{p_ip_l}(Du^\ez)u^\epsilon_{x_lx_j}(u^\epsilon_{x_i}-c\delta_{2i})(u^\epsilon-cx_2)\phi^4 \,d x\\
&\quad -2\int_U  H_{p_j}(Du^\ez)H_{p_ip_l}(Du^\ez)u^\epsilon_{x_lx_j}
\phi_{x_i} (u^\epsilon-cx_2)^2\phi^3 \,d x\\
&\quad+\int_U({\rm div}[ {D_pH}(Du^\ez)  ])H_{p_i}(Du^\ez)(u^\epsilon_{x_i}-c\delta_{2i})(u^\epsilon-cx_2)\phi^4\,d x\\
&\quad+2\int_U ({\rm div}[ {D_pH}(Du^\ez)  ])H_{p_i}(Du^\ez)\phi_{x_i}
(u^\epsilon-cx_2)^2\phi^3\,d x\\
& =:J_1+J_2+J_3+J_4.
\end{align*}

  By the H\"{o}lder inequality and Lemma \ref{LEM2.7}, using
  $\Lambda_H\le \Lambda_H(\fz)$,   we have
\begin{align*}
J_1+J_2&\le\left[\int_U\langle  {D^2_{pp}H}(Du^\ez)  D[H(Du^\ez)] ,D[H(Du^\ez)] \rangle\phi^4\,d x\right]^{1/2}\\
&\quad\quad\times\left[\int_U\langle  {D^2_{pp}H}(Du^\ez) (Du^\ez-DF) ,(Du^\ez-DF) \rangle(u^\epsilon-F)^2\phi^4\,d x\right]^{1/2}\\
&\quad+2\left[\int_U\langle  {D^2_{pp}H}(Du^\ez)  D[H(Du^\ez)] ,D[H(Du^\ez)] \rangle\phi^4\,d x\right]^{1/2}\\
&\quad\quad\times \left[\int_U\langle  {D^2_{pp}H}(Du^\ez)  D\phi ,D\phi  \rangle(u^\epsilon-F)^4\phi^2\,d x\right]^{1/2}\\
&\le C \Lambda_H(\fz) \left[\int_U|D[H(Du^\ez)]|^2\phi^4\,d x\right]^{1/2}\\
&\quad\times\left[\int_U \left[(|Du^\epsilon|+|DF|)^2(u^\epsilon-F)^2\phi^4+ |D\phi|^2(u^\epsilon-F)^4\phi^2\right]\,d x\right]^{\frac{1}{2}} .
\end{align*}

By integration by parts, we have
\begin{align*}
J_3&=-\int_UH_{p_m}(Du^\ez)[H_{p_i}(Du^\ez)(u^\epsilon_{x_i}-c\delta_{2i})(u^\epsilon-cx_2)\phi^4]_m\,d x\\
&=-\int_UH_{p_m}(Du^\ez)H_{p_ip_j}(Du^\ez)u^\epsilon_{x_jx_m}(u^\epsilon_{x_i}-c\delta_{2i})(u^\epsilon-cx_2)\phi^4\,d x\\
&\quad-\int_UH_{p_m}(Du^\ez)H_{p_i}(Du^\ez)u^\epsilon_{x_ix_m}(u^\epsilon-cx_2)\phi^4\,d x\\
&\quad-\int_U\langle D_pH(Du^\ez),Du^\epsilon-DF\rangle ^2\phi^4\,d x\\
&\quad-4\int_UH_{p_m}(Du^\ez)H_{p_i}(Du^\ez)(u^\epsilon-cx_2)\phi_{x_m}(u^\epsilon_{x_i}-c\delta_{2i})\phi^3\,d x\\
&=:J_{31}+\cdots+J_{34}
\end{align*}
Note that $J_{31}=J_1$.
Since
$$ H_{p_m}(Du^\ez)H_{p_i}(Du^\ez)u^\epsilon_{x_ix_m} =\mathscr A_H[u^\ez]=-\ez{\,\rm div\,}[ {D_pH}(Du^\ez)],$$
by the H\"{o}lder inequality, we obtain
\begin{align*}
 J_{32}&=\ez\int_U({\,\rm div\,} [{D_pH}(Du^\ez)])(u^\epsilon-cx_2)\phi^4 \,d x\\
&\le \ez
\left[\int_U({{\,\rm div\,} [{D_pH}(Du^\ez)]})^2\phi^4\,d x\right]^{1/2} \left[\int_U(u^\epsilon-F)^2 \phi^4\,d x\right]^{1/2}.
\end{align*}
By  Young's inequality, H\"{o}lder's inequality and Lemma \ref{LEM7.1} (iv), using
  $\Lambda_H\le \Lambda_H(\fz)$ and $\lambda_H\ge \lambda_H(\fz)$,  we have
\begin{align*}
J_{34}
&\le\frac{1}{4}\int \langle {D_pH}(Du^\ez) ,Du^\epsilon-DF\rangle ^2\phi^4\,d x+C\int|D_pH(Du^\ez)|^2|D\phi|^2(u^\epsilon-F)^2\phi^2\,d x\\
&\le-\frac{1}{4}J_{33}+C\left[\int_U| {D_pH}(Du^\ez) |^4|D\phi|^2\phi^2\,d x\right]^{1/2} \left[\int_U(u^\epsilon-F)^4|D\phi|^2\phi^2\,d x\right]^{1/2}\\
&\le-\frac{1}{4}J_{33}+C\frac{[\Lambda _H(\fz)]^2}{\lambda_H(\fz)}
\left[\int_U[H(Du^\ez) ]^2|D\phi|^2\phi^2\,d x\right]^{1/2} \left[\int_U(u^\epsilon-F)^4|D\phi|^2\phi^2\,d x\right]^{1/2}
\end{align*}

By integration by parts, then
\begin{align*}
J_4&=-2\int_U H_{p_m}(Du^\ez)[H_{p_i}(Du^\ez)(u^\epsilon-cx_2)^2\phi_{x_i}\phi^3]_m\,d x\\
&= -2\int_U H_{p_m}(Du^\ez)H_{p_ip_j}(Du^\ez)u^\epsilon_{x_jx_m}
 (u^\epsilon-cx_2)^2 \phi_{x_i}\phi^3\,d x \\
&\quad-4\int_UH _{p_i}(Du^\ez)H_{p_m}(Du^\ez)(u^\epsilon_{x_m}-c\sigma_{2m})(u^\epsilon-cx_2)\phi_{x_i}\phi^3\,d x\\
&\quad-2\int_UH _{p_i}(Du^\ez)H_{p_m}(Du^\ez)(u^\epsilon-cx_2)^2\phi_{x_ix_m}\phi^3\,d x\\
&\quad-6\int_U \langle {D_pH}(Du^\ez) ,D\phi\rangle ^2(u^\epsilon-cx_2)^2\phi^2\,d x\\
&=:J_{41}+\cdots+J_{44}
\end{align*}
Note that $J_{44}\le0$, $J_{41}=J_{2}$  and $J_{34}=J_{42}$.
By H\"older's  inequality and Lemma \ref{LEM7.1}(iv), using
  $\Lambda_H\le \Lambda_H(\fz)$ and $\lambda_H\ge \lambda_H(\fz)$ we have
\begin{align*}
 J_{43}
&\le 2\left[\int_U| {D_pH}(Du^\ez) |^4| D^2\phi| \phi^3\,d x\right]^{1/2}\left[\int_U (u^\epsilon-F)^4|D^2\phi|\phi^3   \,d x\right]^{1/2}\\
&\le 4\frac{[\Lambda_H(\fz)]^2}{\lambda_H(\fz)}\left[\int_U[ H(Du^\ez)]^2| D^2\phi| \phi^3\,d x\right]^{1/2}\left[\int_U (u^\epsilon-F)^4|D^2\phi|\phi^3   \,d x\right]^{1/2}.
\end{align*}

  Combining all estimates together,  we obtain
\begin{align*}
  &\int_U\langle D_pH(Du^\ez),Du^\epsilon-DF\rangle ^2\phi^4\,d x\\
  &\quad\le C \Lambda_H (\fz) \left[\int_U|D[H(Du^\ez)]|^2\phi^4\,d x\right]^{1/2}\\
&\quad\quad\quad\times\left[\int_U \left[(|Du^\epsilon|+|DF|)^2(u^\epsilon-F)^2\phi^4+ |D\phi|^2(u^\epsilon-F)^4\phi^2\right]\,d x\right]^{\frac{1}{2}}\\
&\quad\quad +C\frac{[\Lambda_H(\fz)]^2}{\lambda_H(\fz)}\left[\int_U [H(Du^\ez)]^2[| D \phi|^2\phi^2  + | D^2\phi| \phi^3]\,d x\right]^{1/2}\left[\int_U (u^\epsilon-F)^4|D^2\phi|\phi^3   \,d x\right]^{1/2}\\
&\quad\quad+
\left[\int_U\ez^2({\rm div}[D_pH(Du^\ez)])^2\phi^4\,d x\right]^{1/2} \left[\int_U(u^\epsilon-F)^2 \phi^4\,d x\right]^{1/2}
  \end{align*}
  If  letting
  $\phi\subset C^\fz_c(\frac34B)$ with $\phi= 1$ in $\frac12B$, and
$|D\phi|^2+|D^2\phi|\le \frac C{|B|}$ in $B$, we obtain the desired estimates.
 This completes the proof of Lemma \ref{LEM2.8}.
\end{proof}

\subsection{Proofs of Lemmas \ref{LEM2.6},   \ref{LEM3.3}\&\ref{LEM3.4}}

\begin{proof}[Proof of Lemma \ref{LEM2.6}]
Let $\varphi  = e^{\frac t\ez  H (Du^\epsilon )} \phi^{2 mt }$ for any $ 0\le \phi\in C^\infty_c(U)$.
We have
$$\varphi_{x_i}=\frac t\ez e^{\frac t\ez  H (Du^\epsilon )} H_{p_m}(Du^\ez)
u^\ez_{x_mx_i} \phi^{2 mt } + 2 mt \phi^{2 mt -1} \phi_{x_i} e^{\frac t\ez  H (Du^\epsilon )} .$$
Note that $\varphi\in W^{1,2}_c(U)$.  By \eqref{FUN2.4} and   a
direct calculation, we obtain
\begin{align*}
  I(u^\ez, \varphi)&=\int_U {\rm div}[ {D_pH} (Du^\ez)] H_{p_i}(Du^\ez)\varphi_{x_i}\,dx-\int_U H_{p_j}(Du^\ez) u^\epsilon_{x_lx_j}H_{p_ip_l}(Du^\ez)\varphi_{x_i} \,d x\\
 &=\frac t\ez\int_U {\rm div}[ {D_pH} (Du^\ez)] H_{p_i}(Du^\ez) e^{\frac t\ez  H (Du^\epsilon )} H_{p_m}(Du^\ez)
u^\ez_{x_mx_i} \phi^{2 mt }\,dx\\
&\quad+2 mt\int_U {\rm div}[ {D_pH} (Du^\ez)] H_{p_i}(Du^\ez) \phi^{2 mt -1} \phi_{x_i} e^{\frac t\ez  H (Du^\epsilon )} \,dx\\
&\quad-\frac t\ez\int_U H_{p_j}(Du^\ez) u^\epsilon_{x_lx_j}H_{p_ip_l}(Du^\ez) e^{\frac t\ez  H (Du^\epsilon )} H_{p_m}(Du^\ez)
u^\ez_{x_mx_i} \phi^{2 mt } \,d x\\
&\quad-2 mt\int_U H_{p_j}(Du^\ez) u^\epsilon_{x_lx_j}H_{p_ip_l}(Du^\ez) \phi^{2 mt -1} \phi_{x_i} e^{\frac t\ez  H (Du^\epsilon )} \,d x\\
&=:J_1+\cdots+J_4
\end{align*}
Since $$H_{p_i}(Du^\ez)   H_{p_m}(Du^\ez)
u^\ez_{x_mx_i} =\mathscr A_H[u^\ez]=-\ez {\rm div}[ {D_pH} (Du^\ez) ],$$  we have
$$J_1=- t\int_U ({\rm div}[ {D_pH} (Du^\ez) ])^2    e^{\frac t\ez  H (Du^\epsilon )}
  \phi^{2 mt }\,dx.$$

   By Young's inequality and Lemma \ref{LEM7.1} (iv),   we have
 \begin{align*}
 J_2&=2 mt \int_U {\rm div}[ {D_pH} (Du^\ez)]  \langle  {D_pH}  (Du^\ez),D\phi\rangle  \phi^{2 mt -1}   e^{\frac t\ez  H (Du^\epsilon )} \,dx\\
 &\le -\frac14  J_1 +  (2m)^2t\int_U  \langle  {D_pH}  (Du^\ez),D\phi\rangle ^2  \phi^{2 mt -2}   e^{\frac t\ez  H (Du^\epsilon )} \,dx\\
 &\le -\frac14  J_1 +  (2m)^2t \frac{[\Lambda_H(\fz)]^2}{2\lambda_H(\fz)}\int_U    H(Du^\ez)|D\phi|^2  \phi^{2 mt -2}   e^{\frac t\ez  H (Du^\epsilon )} \,dx
   \end{align*}
  Noting $D[H(Du^\ez)]=  D^2u^\ez {D_pH}(Du^\ez) $ we have
\begin{align*} J_3&=  - \frac t\ez \int_U \langle  {D^2_{pp}H}(Du^\ez) D[H(Du^\ez)] , D[H(Du^\ez)]\rangle e^{\frac t\ez  H (Du^\epsilon )}
 \phi^{2 mt } \,d x\\
 &\le  - \frac t{\ez }\lambda_H(\fz)  \int_U  |D[H(Du^\ez)]|^2   e^{\frac t\ez  H (Du^\epsilon )}
 \phi^{2 mt } \,d x.
  \end{align*}
 By Young's inequality, we have
 \begin{align*}
 J_4&= 2 mt \int_U \langle  {D^2_{pp}H}(Du^\ez) D[H(Du^\ez)] , D\phi \rangle  \phi^{2 mt -1}   e^{\frac t\ez  H (Du^\epsilon )} \,d x\\
 &\le -\frac14 J_3 + (2m)^2t\ez \int_U \langle  {D^2_{pp}H}(Du^\ez) D\phi , D\phi \rangle  \phi^{2 mt -2}   e^{\frac t\ez  H (Du^\epsilon )} \,d x\\
  &\le -\frac14 J_3 + (2m)^2t\ez \Lambda_H(\fz)\int_U  |D\phi|^2  \phi^{2 mt -2}   e^{\frac t\ez  H (Du^\epsilon )} \,d x.
 \end{align*}
Since  $I(u^\ez,\vz)\ge0$ by Theorem \ref{LEM2.3} and \eqref{FUN2.2}, combining above estimates,
we have

\begin{align*}
 & \frac34\frac 1{\ez }\lambda_H(\fz)  \int_U  |D[H(Du^\ez)]|^2   e^{\frac t\ez  H (Du^\epsilon )}
 \phi^{2 mt } \,d x+  \frac34\int_U ({\rm div}[ {D_pH} (Du^\ez) ])^2    e^{\frac t\ez  H (Du^\epsilon )}\\
 &\quad=-\frac34\frac1{t} [J_1+J_3]\le  2(2m)^2  \frac{[\Lambda_H(\fz)]^2}{ \lambda_H(\fz)}\int_U  [1 + H(Du^\ez)]|D\phi|^2  \phi^{2 mt -2}   e^{\frac t\ez  H (Du^\epsilon )} \,dx,
  \end{align*}
which gives
 \begin{align*}
&\int_U\left\{\frac1\ez  | D[H(Du^\ez)]|^2  +\frac1{\Lambda_H(\fz)}[{\rm div} ({D_pH}(Du^\ez)) ]^2\right\}\phi^{2 mt } e^{\frac t\ez H(Du^\ez)} \,d x\\
&\quad\le 8m^2 \left[\frac{ \Lambda_{H}(\fz)}{\lambda_H(\fz) }\right]^2\int_U  [1+H(Du^\ez)] |D\phi|^2\phi^{2 mt -2} e^{\frac t\ez H(Du^\ez)}\,d x
\end{align*}
as desired.  This completes the proof of Lemma \ref{LEM2.6}.
\end{proof}

We prove  Lemmas \ref{LEM3.4}\&\ref{LEM3.3} by using Lemma \ref{LEM2.6} and Sobolev's imbedding, and  borrowing some ideas from \cite[Theorem 5.1]{e03}.

\begin{proof}[Proof of Lemma \ref{LEM3.3}]
For any $\beta>1$, by Sobolev's embedding
 $W^{1,2}_0(U)\hookrightarrow L^{2\beta}(U),$
we have
$$
\|\sigma ^{\epsilon}\phi^4\|_{L^\beta(U)} =\left(\int_U[(\sigma ^{\epsilon})^{1/2}\phi^2]^ { {2\beta}} \,d x\right)^{1/\beta}\le
C(U,\bz ) \int_U|D[(\sigma ^{\epsilon})^{\frac{1}{2}}\phi^2]|^2 \,d x,
$$
where recall that $\phi$ is as in \eqref{eq3.4}.
Noting
\begin{equation}\label{eq6.1}
  |D[(\sigma ^{\epsilon})^{\frac{1}{2}}\phi^2]|^2\le \frac1{ \epsilon ^2} \sigma ^{\epsilon}\phi^4|D[H (Du^\epsilon )]|^2 +4\sigma ^{\epsilon}|D\phi|^2\phi^2,
\end{equation}
by \eqref{eq3.4} and $\int_U\sz^\ez\,dx=1$, we have
\begin{equation}\label{eq6.2}
\|\sigma ^{\epsilon}\phi^4\|_{L^\beta(U)}  \le
C(U,V,\bz) +C(U,\beta)\int_U\frac1{ \epsilon ^2} \sigma ^{\epsilon}\phi^4|D[H(Du^\ez)]|^2 \,d x .
\end{equation}

Applying Lemma \ref{LEM2.6} with $m=2$ and $t=1$, by  \eqref{eq3.4},
we obtain
\begin{equation*}
 \int_U \frac1\ez |D[H(Du^\ez)] |^2\phi^4e^{\frac1\ez H(Du^\ez)}\,d x\le  C(H,U,V)
\int_U   [1+ H  (Du^\epsilon )]    e^{\frac1\ez H(Du^\ez)}\phi^2\,d x,
\end{equation*}
   by $\int_U \sz^\ez\,dz=1$, which implies that
$$
\int_U\frac1{\ez }  |D[H(Du^\ez)] |^2\phi^4\sz^\ez\,d x
 \le  C(H,U,V)+ C(H,U,V)\int_U  H  (Du^\epsilon ) \phi^2\sz^\ez\,d x.
$$
Note that by \eqref{eq3.xx1}, one has
$$\ez \ln \int_Ue^{\frac1\ez H(Du^\ez)}\,dz\le   \ez \ln |U|+    \|H(Du )\|_{L^\fz(U)}\le C(U)[1+ \|H(Du )\|_{L^\fz(U)}].$$
Since $y\le \frac \ez \gz e^{\frac\gz\ez y}$ for all $y\in\rr$ and  by $\int_U \sz^\ez\,dz=1$
we obtain
\begin{align*}
&\int_UH (Du^\epsilon ) \phi^2\sz^\ez\,d x\\
&\quad= \int_U\left[H (Du^\epsilon )-\ez\ln \int_Ue^{\frac1\ez H(Du^\ez)}\,dz\right] \phi^2\sz^\ez\,d x+ \left[\ez\ln \int_Ue^{\frac1\ez H(Du^\ez)}\,dz\right]\int_U \phi^2\sz^\ez\,d x \\
&\quad\le \int_U \frac \ez \gz e^{\frac\gz\ez  \left[H (Du^\epsilon )-\ez\ln \int_Ue^{\frac1\ez H(Du^\ez)}\,dz\right] } \phi^2\sz^\ez\,d x+ C( U,V)[1+\|H(Du)\|_{L^\fz(U)}] \\
&\quad=  \frac \ez \gz \int_U( \sz^\ez)^{1+\gz} \phi^2\,dx+C( U,V)[1+\|H(Du)\|_{L^\fz(U)}].
\end{align*}

Plugging this  in \eqref{eq6.2}    we obtain
\begin{equation}\label{eq6.3}
\|\sigma ^{\epsilon}\phi^4\|_{L^\beta(U)}
\le C(H, U,V,\bz)\left[\frac1 {\epsilon} [1+\|H(Du)\|_{L^\fz(U)}] +  \frac1   \gz \int_U( \sz^\ez)^{1+\gz} \phi^2\,dx\right].
 \end{equation}

Letting $\gamma=(\beta-1)/{2\beta}$ and
noting  $1+\gamma= 1/2+({2\beta-1})/{2\beta}$   we have
\begin{equation*}
\|\sigma ^{\epsilon}\phi^4\|_{L^\beta(U)}
\le C(H,U,V,\beta) \left\{\frac1 {\epsilon}[1+\|H(Du)\|_{L^\fz(U)}]   +\int_U(\sigma ^{\epsilon})^{(3\bz-1)/2\bz}\phi^2\,d x\right\}
\end{equation*}
 By H\"older's inequality,  $\int_U\sz^\ez\,dx=1$ and Young's inequality, we have
\begin{align*}
\int_U(\sigma ^{\epsilon})^{(3\beta-1)/{2\beta} }\phi^2\,d x&\le  \left(\int_U\sigma ^{\epsilon}\,d x\right)^{ ( 2\beta-1)/2\beta } \left(\int_U[(\sigma ^{\epsilon})^{1/2}\phi^2]^{2\beta} \,d x\right)^{1/2\beta}\\
& =\|\sigma ^{\epsilon}\phi^4\|_{L^\beta(U)}^{1/2}   \le \eta \|\sigma ^{\epsilon}\phi^4\|_{L^\beta(U)}+ 4\eta^{-1}
\end{align*}
for any $\eta>0$.
Letting $\eta= [2C(H,U,V,\beta)]^{-1}$, we have
\begin{equation*}
\|\sigma ^{\epsilon}\phi^4\|_{L^\beta(U)}
 \le C(H,U,V,\beta) \frac1 {\epsilon}[1+\|H(Du)\|_{L^\fz(U)}]  +
 \frac{1}{2}\|\sigma ^{\epsilon}\phi^4\|_{L^\beta(U)}
\end{equation*}
from which we conclude the desired estiamtes. This completes the proof of Lemma \ref{LEM3.3}.
\end{proof}

\begin{proof}[Proof of Lemma \ref{LEM3.4}]
For $\beta=\theta^2>1$, by Sobolev's imbedding  we have
\begin{equation}
\begin{split}
\|[\sigma ^{\epsilon}\phi^{2m}]^{t }\|_{L^\beta(U)} =\left(\int_U[(\sigma ^{\epsilon})^{t/2}\phi^{ mt }]^{2\beta}\,d x\right)^{1/\beta}&\le
C(U,\beta) \int_U |D[(\sigma ^{\epsilon})^{t/2}\phi^{ mt }]|^2\,d x.
\end{split}
\end{equation}
By \eqref{eq3.4} we have
\begin{equation*}
\begin{split}
|D[(\sigma ^{\epsilon})^{t/2}\phi^{ mt }]|^2&\le
Ct^2(\sigma^\epsilon)^{t-2}|D\sigma^\epsilon|^2\phi^{2 mt }+C( U,V)m^2t^2\phi^{2 mt -2} (\sigma^\epsilon)^t,
\end{split}
\end{equation*}
and hence
\begin{align*}
\|[\sigma ^{\epsilon}\phi^{2m}]^{t }\|_{L^\beta(U)}
&\le C(U,\beta) t^2\int_U(\sigma^\epsilon)^{t-2}|D\sigma^\epsilon|^2\phi^{2 mt }\,dx+
C(U,V,\beta) m^2t^2\int_U (\sigma^\epsilon)^t\phi^{2 mt -2} \,d x
\end{align*}
Noting $|D\sigma^\epsilon|^2=\frac1{\epsilon^2} (\sigma^\epsilon)^2|D[H(Du^\ez)]|^2$, by Lemma \ref{LEM2.6} and \eqref{eq3.4}    we have
\begin{align*}
 \int_U (\sigma ^{\epsilon})^{t-2}|D\sigma^\epsilon|^2\phi^{2 mt }\,d x
&=\frac1{\epsilon^2 }\int_U (\sigma ^{\epsilon})^t|D[H(Du^\ez)]|^2\phi^{2 mt }\,d x\\
&\le C(H)\frac1{\epsilon  }m^2\int_U(\sigma ^{\epsilon})^t\phi^{2 mt -2}[1+ H(Du^\ez)][|D\phi|^2+|D^2\phi|\phi ]\,d x\\
&\le C(H,U,V)\frac1{\epsilon  }m^2\int_U(\sigma ^{\epsilon})^t\phi^{2 mt -2}[1+ H(Du^\ez)]\,d x.
\end{align*}
Since  \eqref{eq3.xx1} implies
$$\ez \ln \int_Ue^{\frac1\ez H(Du^\ez)}\,dz\le     C(U)[1+ \|H(Du )\|_{L^\fz(U)}],$$ we have
\begin{align*}&\int_U (\sigma ^{\epsilon})^{t-2}|D\sigma^\epsilon|^2\phi^{2 mt }\,d x\\
&\quad\le C(H,U,V )\frac1{\epsilon  }m^2\int_U \phi^{2 mt -2}(\sigma ^{\epsilon})^t \\
&\quad\quad\quad\quad\times
\left\{C(U)[1+ \|H(Du)\|_{L^\fz(U)}]  + \left[H(Du^\epsilon)-\ez\ln\int_Ue^{\frac1\ez H(Du^\ez)}\,dz\right ]\right\} \,d x.
\end{align*}
By $y\le \frac{\ez\theta}{\theta-1}
e^{\frac{\theta-1}{\theta\ez} y}$, this gives
\begin{align*}
&\int_U (\sigma ^{\epsilon})^{t-2}|D\sigma^\epsilon|^2\phi^{2 mt }\,d x\\
&\quad\le C(H,U,V)\frac1{\epsilon  }m^2\int_U\phi^{2 mt -2}(\sigma ^{\epsilon})^t\\
&\quad\quad\quad\quad\quad\times
\left\{[1+ \|H(Du)\|_{L^\fz(U)}]  + \frac{\ez\theta}{\theta-1}
e^{\frac{\theta-1}{\theta\ez}\left [H(Du^\ez) -\ez\ln\int_Ue^{\frac1\ez H(Du^\ez)}\,dz \right]} \right\} \,d x\\
&\quad\le   C(H,U,V,\beta)m^2\int_U\phi^{2 mt -2}
\left\{(\sigma ^{\epsilon})^{t+(\theta-1)/\theta} +\frac1{\epsilon  }(\sigma ^{\epsilon})^t
[1+  \|H(Du)\|_{L^\fz(U)}]
  \right\} \,d x.
  \end{align*}
That is,
\begin{align*}
\|[\sigma ^{\epsilon}\phi^{2m}]^{t }\|_{L^\beta(U)} &\le
C(H,U,V,\beta)m^2t^2\\
&\quad\times\left\{\int_U \phi^{2 mt -2}    (\sigma ^{\epsilon})^{t+\frac{\theta-1}{\theta}}  \,dx+\frac1{\epsilon}[1+  \|H(Du)\|_{L^\fz(U)}] \int_U(\sigma^\epsilon)^t\phi^{2 mt -2}  \,d x\right\}.
\end{align*}
 Via H\"older's inequality, $\int_U \sigma ^{\epsilon}\,d x=1$ and $0\le\phi\le1$, we obtain
\begin{equation*}
\begin{split}
\int_U(\sigma ^{\epsilon})^{t+\frac{\theta-1}{\theta}} \phi^{2 mt
-2} \,d x&\le\left[\int_U[\sigma^\epsilon\phi^m]^{t\theta}\,d x\right]^{1/\theta}
 \left[\int_U\sigma^\epsilon\phi^{\frac{( mt -2)\theta}{\theta-1}}\,d x\right]^{1-1/{\theta}}\le    \|[\sigma^\epsilon\phi^m]^t \|_{L^\theta(U)}
\end{split}
\end{equation*}
and
\begin{equation*}
\begin{split}
\int_U(\sigma^\epsilon)^t \phi^{2 mt -2}\,d x&\le\left[\int_U[\sigma^\epsilon\phi^m]^{t\theta}\,d x\right]^{1/\theta}
 \left[\int_U\phi^{\frac{( mt -2)\theta}{\theta-1}}\,d x\right]^{1-1/{\theta}} \le C(U,\beta) \|[\sigma^\epsilon\phi^m]^t \|_{L^\theta(U)}.
\end{split}
\end{equation*}
We therefore conclude
\begin{align*}
 \|[\sigma^\epsilon\phi^{2m}]^t \|_{L^\beta(U)}\le
C(H,U,V,\beta) \frac1{\epsilon} m^2t^2[1+  \|H(Du)\|_{L^\fz(U)}] \|[\sigma^\epsilon\phi^m]^t \|_{L^\theta(U)},
\end{align*}
which completes the proof of Lemma \ref{LEM3.4}.
\end{proof}

\renewcommand{\thesection}{Appendix A}
 \renewcommand{\thesubsection}{ A }
\newtheorem{lemapp}{Lemma \hspace{-0.15cm}}
\newtheorem{corapp}[lemapp] {Corollary \hspace{-0.15cm}}
\newtheorem{remapp}[lemapp]  {Remark  \hspace{-0.15cm}}
\newtheorem{defnapp}[lemapp]  {Definition  \hspace{-0.15cm}}
\renewcommand{\theequation}{A.\arabic{equation}}

\renewcommand{\thelemapp}{A.\arabic{lemapp}}

\section{Some properties of $H$ and auxiliary functions} 

We  recall several properties of
$\lambda_H$ and $\Lambda_H$ in Lemma \ref{LEM7.1},
and give the continuity and lower bound  of $\wz\tau_H$ when $H\in C^2(\rr^2)$ in  Lemma \ref{LEM7.2}.
In Lemma \ref{LEM7.3} we recall a standard smooth approximation $\{H^\dz\}_{\dz\in(0,1]}$ to $H$,
and prove some useful  properties of $\lambda_{H^\dz}$ and $\Lambda_{H^\dz}$.
In Lemma \ref{LEM7.5} we give some useful properties of $\tau_H$.
Finally, for $H\in C^k(\rr^2)$ satisfying (H1)\&(H2) and given $R>0$,   in Lemma \ref{LEM7.8} we find  $\wz H\in C^k(\rr^2)$ satisfying
 (H1')\&(H2)  so that $\wz H=H$ in $\{p,|p|<R\}$.

The following basic properties of $\lambda_H$ are $\Lambda_H$ are known.
We omit the proofs.
\begin{lemapp}\label{LEM7.1} Suppose that  $H\in C^0(\rn)$ satisfies  (H1){\rm\&}(H2). The following hold.
\begin{enumerate}
\item[(i)] The function $\lambda_H$ is  decreasing   and
$\Lambda_H$ is   increasing. Both of
 $\lambda_H$ and $\Lambda_H$ are right-continuous in $[0,\fz)$, that is,
$$\mbox{$\Lambda_H(R)=\lim_{\ez\to0}\Lambda_H(R+\ez)$ and $\lambda_H(R)=\lim_{\ez\to0}\lambda_H(R+\ez)$
for all $R\ge0$}.$$
If $H\in C^2(\rn)$ additionally, then $\lambda_H,\Lambda_H  \in C^0([0,\fz))$.

\item[(ii)] For all $R>0$,  both of
$  H(p)-\frac12\lambda_H (R)|p |^2 $ and  $  \frac12\Lambda_H (R)|p |^2-H(p )$   are convex  in $H^{-1}([0,R ])$.

If $H\in C^1(\rn)$ additionally, then
$$\lambda_{H}(R)|p-q|^2\le \langle D_pH (p)-D_p H(q),p-q\rangle \le \Lambda_{H}(R)|p-q|^2\quad\forall p,q\in H^{-1}([0,R])$$
and
$$\frac{\lambda_{H}( R )}2|p-q|^2\le H(q)-H(p)-\langle D_p H(p),p-q\rangle\le \frac{\Lambda_{H}( R )}2|p-q|^2\quad\forall p,q\in H^{-1}([0,R]).$$
If $H\in C^2(\rn)$ additionally,
then
$$\frac{\lambda_H(H(p))}{2} |\xi|^2\le\langle D^2_{pp}H(p)\xi,\xi\rangle \le \frac{\Lambda_H(H(p))}{2} |\xi|^2\quad\forall p,\xi\in\rn.$$

\item[(iii)] For all $p\in\rn$, we have $$\frac{\lambda_{H}(H(p))} 2|p|^2\le H(p)\le \frac{\Lambda_{H}(H(p))}2|p|^2.$$

\item[(iv)] If $H\in C^1(\rn)$ additionally, then $ H^{1/2}\in C^{0,1}(\rr)$, $D_p H(0)=0$ and
$$|D_p H(p)|^2\le[\Lambda_{H}(H(p))]^2|p|^2\le \frac{2[\Lambda_{H}(H(p))]^2}{\lambda_{H}(H(p))} H(p)\quad\forall p\in \rr^2.$$
\end{enumerate}
\end{lemapp}

  We have the following properties for $\wz\tau_H$ as defined in \eqref{eq1.6}.
\begin{lemapp}\label{LEM7.2}
Suppose that  $H\in C^2(\rn)$ satisfies  (H1){\rm\&}(H2).
\begin{enumerate}
\item[(i)] The function
    $\wz \tau_H\in C^0(\rn)$ satisfies
 $$ \wz \tau_H(0)=\frac12 \quad{\rm and}\quad\wz\tau_H(p)\ge \frac12\left[\frac{\lambda_H(H(p))   }{  \Lambda_H(H(p))  }\right]^2    \quad {\rm in }\ \rn\setminus\{0\}.$$

\item[(ii)]  If $H^{\gz}$ is convex in $H^{-1}([0,R])$ for some $\gz\in[1/2,1)$ and $R>0$ ,
 then $\wz \tau_{  H}\ge 1-\gz$ in $H^{-1}([0,R])$.
\end{enumerate}
\end{lemapp}

\begin{proof}

(i) Since $\wz\tau_H \in C^0((0,\fz))  $, to see  $\wz\tau_H \in C^0([0,\fz))  $,
 it suffices to prove that $\wz\tau_H$ is continuous at $0$.
By $D_p H(0)=0$, we have
$$D_p H (p)=  D^2_{pp}H(0)p +o(|p| )\quad{\rm and }\quad H (p)= \frac12\langle D^2_{pp}H(0)p, p\rangle+o(|p|^2)\quad {\rm as}\ p\to0.$$
By this, the continuity of $(D^2_{pp} H )^{-1}$ at $0$ and Lemma \ref{LEM7.1}(iv), we have
\begin{align*}
\langle [D^2_{pp} H (p)]^{-1} D_p H (p),D_p H (p)\rangle&=
\langle [D^2_{pp} H (0)]^{-1} D_p H (p),D_p H (p)\rangle+o( |p|^2)\\
&= \langle [D^2_{pp} H (0)]^{-1}D^2_{pp}H(0)p,D^2_{pp}H(0)p\rangle+o(|p|^2)\\
&= \langle  D^2_{pp}H(0)p, p\rangle+o(|p|^2)\\
&=2H(p)+o(|p|^2)
\quad {\rm as}\ p\to0.
\end{align*}
Thus $\wz \tau_{H}(p)=  1/2+o(1)$ when $p\to0$, as desired.

Moreover, by Lemma \ref{LEM7.1} (iii) and (iv), we have
\begin{align*}
&\langle (D_{pp}^2 H )^{-1}(p) D_p H (p),D_p H (p)\rangle\le \frac1{\lambda_H(H(p))} |D_p H(p)|^2
 \le 2\left[\frac{\Lambda_H(H(p)) }{\lambda_H(H(p))}\right]^2 H(p) \quad\forall p\in\rn
\end{align*}
that is, $$\wz \tau_H(p)\ge \frac12\left[\frac{\lambda_H(H(p))   }{  \Lambda_H(H(p))  }\right]^2\quad\forall p\in\rn\setminus\{0\} $$ as desired.

(ii)  It suffices to prove that $\wz \tau_H\ge 1-\gz$ almost everywhere in $H^{-1}([0,R])$.

 Note that the convexity of $H^{\gz}$ implies that $H^{\gz}$ is second order differentiable
  almost everywhere in $H^{-1}([0,R])$, and moreover, at such a point $p$, we have
$\langle D^2_{pp} H^{\gz}(p)\xi,\xi\rangle\ge 0$ for all $\xi\in\rr^2$.
Thus, at such a point $0\ne p\in H^{-1}([0,R])$,  by $\wz \tau_H>0$ we have
\begin{align*}
\frac1{\wz\tau_H }=&\frac1{H }\langle (D^2_{pp} H )^{-1}  D_p H  ,D_p H  \rangle \\
&=\frac1{\gz^2} H^{2-2\gz }\langle (D^2_{pp} H )^{-1}  D_p H^{\gz}  ,D_p H^{\gz}  \rangle \\
&=\frac1{\gz^2} H^{2-2\gz}\langle (D^2_{pp} H )    (D^2_{pp} H )^{-1} D_p H^{\gz} ,(D^2_{pp} H )^{-1} D_p H^{\gz} \rangle \\
&= \frac1{\gz^3} H^{3-3\gz } \langle (D^2_{pp} H^{\gz} )  (D^2_{pp} H  )^{-1} D_p H^{\gz} , (D^2_{pp} H )^{-1}D_p H^{\gz} \rangle\\
&\quad
+   \frac1{\gz^3} ( 1-\gz ) H^{2-3\gz}\langle (D_p H^{\gz} \otimes D_p  H  )  (D^2_{pp}  H  )^{-1} D_pH^{\gz} ,(D^2_{pp} H  )^{-1}D_p H^{\gz} \rangle\\
&\ge
     ( 1-\gz ) H^{-1}  \langle ( D_p H  )  ,(D^2_{pp}  H  )^{-1}D_p H  \rangle^2\\
 &\ge    ( 1-\gz )  \frac1{[\wz\tau_H]^2 },
\end{align*}
that is,  $ \wz\tau_H  \ge 1-\gz $  as desired.
This completes the proof of Lemma \ref{LEM7.2}.
\end{proof}

If $H\in C^0(\rr^2)$ satisfies  (H1) (resp. (H1')) and (H2),  there is  a  standard smooth approximation satisfying  (H1) (resp. (H1')) and (H2).
 For each  $\delta\in(0,1]$, let $\wz H^\dz=\eta_\dz\ast H$, where $\eta_\dz$ is standard smooth mollifier. It is easy to see that $H^\dz$ is strictly convex. Thus, for each $\dz\in (0,1]$, there exists a unique point $p^\dz\in \rr^2$ such that
$$\wz H^\dz(p^\dz)= \min_{p\in\rr^2} \wz H^\dz(p)\le \wz H^\dz(0)$$
 Moreover, for $\dz\in (0,1]$, set
\begin{equation} \label{eq7.1} H^\dz(p)=\wz  H^\dz(p+p^\dz)-\wz H^\dz(p^\dz)\quad\forall p\in\rn.\end{equation}
Then we have the following result.
\begin{lemapp}\label{LEM7.3} Suppose that  $H\in C^0(\rn)$ satisfies  (H1) (resp. (H1')) and (H2).

\begin{enumerate}
\item[(i)]  For each $\dz\in(0,1]$,  $H^\dz\in C^\fz(\rn)$ satisfies  (H1) (resp. (H1')) and (H2).
In the case that $H\in C^0(\rn)$ satisfies    (H1')  and (H2), we have
\begin{equation}\label{eq7.3x}\mbox{$ \lambda_{H^\dz} \ge \lambda_H(\fz )$ and
 $  \Lambda_{H^\dz} \le \Lambda_H(\fz )$     \ {\rm in} \ $[0,\fz]$ \quad $\forall\,\dz\in(0,1]$}.\end{equation}

\item[(ii)] There exists a $\dz_H(0)\in (0,1]$  such that  for all $\dz\in(0,\dz_H(0)]$, we have
$|p^\dz|\le  \dz/\dz_H(0)$.
Hence,  $  H^\dz\to H$ locally uniformly in $\rn$.
  If $H\in C^2(\rn)$ additionally, then
  $  H^\dz\to H$  in  $C^{2}(\rn)$  as $\dz\to0$.

\item[(iii)] For
any $R>0$, there exist $\dz_H(R)\in (0,1]$  such that for all $\dz\in(0,\dz_H(R)]$, we have
\begin{equation}\label{eq7.2} \mbox{$\lambda_{H^\dz}(r)\ge \lambda_H(r+   \dz/\dz_H(R) )$ and
 $\Lambda_{H^\dz}(r)\le \Lambda_H(r+  \dz/\dz_H(R) )$   $\forall\,  r\in[0,R]$ }.\end{equation}
Moreover, for all $r>0$ and $r^\dz\to r$ as $\dz\to0$, we have
\begin{equation}\label{eq7.3}\mbox{$\liminf_{\dz\to0}\lambda_{H^\dz}(r^\dz)\ge \lambda_H(r  )$ and
 $\limsup_{\dz\to0} \Lambda_{H^\dz}(r^\dz)\le \Lambda_H(r  )$   ${\rm as}\   \dz\to   0$}.\end{equation}
    \end{enumerate}
\end{lemapp}

\begin{proof}
(i)  Given $R>1$,  assume that $H(p)-\frac\lambda2|p|^2$ is convex in $B(0,R+1)$.
  For any $p,q\in B(-p^\dz,R)$ and   any $\theta\in(0,1)$, write
 \begin{align*}
  H^\dz(\theta p+(1-\theta )q)-\frac\lambda2|\theta p+(1-\theta )q|^2
& =(H-\frac\lambda2|\cdot|^2) \ast\eta_\dz(\theta (p+p^\dz)+(1-\theta )(q+p^\dz))- \wz H^\dz(p^\dz) \\
& \quad+ (\frac\lambda2|\cdot|^2) \ast\eta_\dz(\theta p+(1-\theta )q+p^\dz)-\frac\lambda2|\theta p+(1-\theta )q|^2.
\end{align*}
Since  $p+\xi+p^\dz, q+\xi+p^\dz\in B(0,R+1)$ for all $|\xi|\le 1$, we have

\begin{align*}
& (H-\frac\lambda2|\cdot|^2) \ast\eta_\dz(\theta (p+p^\dz)+(1-\theta )(q+p^\dz))\\
&\quad\le \theta(H-\frac\lambda2|\cdot|^2) \ast\eta_\dz(  p+p^\dz  ) +(1-\theta)
(H-\frac\lambda2|\cdot|^2) \ast\eta_\dz( q+p^\dz )\\
&\quad=  \theta [H^\dz(p)-\frac\lambda2|p|^2]+(1-\theta)  [H^\dz(q)-\frac\lambda2|q|^2]\\
&\quad\quad-\frac\lambda2[\theta|p-\cdot|^2+(1-\theta) |q-\cdot|^2]\ast\eta_\dz(p^\dz)+\frac\lambda2[\theta|p|^2+(1-\theta)|q|^2].
\end{align*}
Thus
\begin{align*}
&  H^\dz(\theta p+(1-\theta )q)-\frac\lambda2|\theta p+(1-\theta )q|^2\\
&\quad= \theta [H^\dz(p)-\frac\lambda2|p|^2]+(1-\theta)  [H^\dz(q)-\frac\lambda2|q|^2]\\
&\quad\quad-\frac\lambda2[\theta|p-\cdot|^2+(1-\theta) |q-\cdot|^2 - |\theta p+(1-\theta )q-\cdot|^2]\ast\eta_\dz(p^\dz)
\\
&\quad\quad+\frac\lambda2[\theta|p|^2+(1-\theta)|q|^2 -|\theta p+(1-\theta )q|^2]
\end{align*}
Since  $$\theta |p-\xi|^2+(1-\theta) |q-\xi|^2 - |\theta p+(1-\theta )q-\xi|^2= \theta|p|^2+(1-\theta)|q|^2 -|\theta p+(1-\theta )q|^2  $$
for all $\xi\in\rr^2$, we get
\begin{align*}
&  H^\dz(\theta p+(1-\theta )q)-\frac\lambda2|\theta p+(1-\theta )q|^2 \le \theta [H^\dz(p)-\frac\lambda2|p|^2]+(1-\theta)  [H^\dz(q)-\frac\lambda2|q|^2],
\end{align*}
that is, $H^\dz(p)-\frac\lambda2|p|^2$ is convex in $ B(-p^\dz,R)$.

Similarly, if  $ \frac\Lambda2|p|^2 -H(p)$ is convex in $B(0,R+1)$, then
$\frac\Lambda2|p|^2-H^\dz(p)$ is convex in $ B(-p^\dz,R)$.
This implies that $H^\dz$ satisfies (H1)\&(H2) whenever  $H $ satisfies (H1)\&(H2).

If $H$ satisfies (H1'), then letting $\lambda=\lambda_H(\fz)$ above   we know that
 $H^\dz(p)-\frac{\lambda_H(\fz)}2|p|^2$ is convex in $ B(-p^\dz,R)$ for all $R\ge 1$ and hence in $\rn$.
 Similarly,  $\frac{\Lambda_H(\fz)}2|p|^2-H^\dz(p)$ is convex   in $\rr^2$.
 That is, $H^\dz$ satisfies (H1')\&(H2), in particular,  \eqref{eq7.3x} holds.

(ii)
Let \begin{equation}\label{eq7.6y1}
\wz \lambda_H(R)=\sup\{\lambda>0,H(p)-\frac\lambda2|p|^2 \mbox{ is convex in $B(0,R)$}\}
\end{equation}
and
\begin{equation}\label{eq7.6y2}\wz \Lambda_H(R)=\inf\{\Lambda>0,  \frac\Lambda2|p|^2- H(p) \mbox{ is convex in $B(0,R)$}\}.\end{equation}
Note that for any $R>0$, we have
$$\frac{\wz\lambda_H(R)}2|p|^2\le
H(p)\le \frac{\wz\Lambda_H(R)}2|p|^2\quad\forall |p|<R.$$
 Then
$$\wz H^\dz(0)\le \sup_{|q |\le \dz} H(q) \le \frac{\wz \Lambda_H(\dz)}2 \dz^2\le
 \frac{\wz \Lambda_H(1)}2 \dz^2 .$$
If $4\ge |p|>2\dz$, then
$$\wz H^\dz(p)>\inf_{|q-p|\le \dz} H(q)\ge
 \frac{\wz \lambda_H(5 )}2\inf_{|q-p|\le \dz}|q|^2\ge    \frac{\wz \lambda_H(5)}8|p |^2.
$$
If $|p| \ge 4 $, then
$$\wz H^\dz(p)> \frac{|p|}4 H( 4 p/|p|) \ge   \frac{\wz \lambda_H(5)}2 |p|\ge 2\wz \lambda_H(5) .
$$
For  $0<\dz< [ {\wz \lambda_H(5 )}/{4\wz \Lambda_H(1)}]^{1/2}$,
 we have $\wz H^\dz(p)>\wz H^\dz(0)$ whenever $|p|> [ {4\wz\Lambda_H(1)}/{\wz \lambda_H(5 )}]^{1/2}\dz $,
 that is, $|p^\dz|\le  [ {4\wz\Lambda_H(1)}/{\wz \lambda_H(5 )}]^{1/2}\dz$.

(iii) Note that \eqref{eq7.3} follows from \eqref{eq7.2}. Indeed,   for any $r>0$, let $R=r+1$.
If \eqref{eq7.2} holds,   by the decrease and the right-continuity of $ \lambda_H$ and $ \Lambda_H$ we get
$$\liminf_{\dz\to0}\lambda_H^\dz(r^\delta)\ge
\liminf_{\dz\to0}\lambda_H (r^\delta+\dz/\delta_H(R))\ge  \lambda_H(r), $$
and   $$\limsup_{\dz\to0}\Lambda_H^\dz(r^\delta)\le
\limsup_{\dz\to0}\Lambda_H (r^\delta+\dz/\delta_H(R))\le \Lambda_H(r),$$
as desired.

To prove \eqref{eq7.2}, with loss of generality  we may assume that $R\ge1$.
For $0<r<R $ and $\dz\in(0,1]$, let $$U^\dz_r=\bigcup\{B(p+p^\dz,\dz),\ p\in(H^\dz)^{-1}([0,r])\}.$$
We claim that
there exists $\dz_H(R)\in(0,1]$ such that
 $$\mbox{$U^\dz_r\subset H^{-1}([0,r+\dz/\dz_H(R)])  $ whenever $\dz\in(0,\dz_H(R)]$ and $0<r<R$.}$$
 Assume this holds for the moment. If $H(p)-\frac\lambda2|p|^2$ is convex in $H^{-1}([0,r+\dz/\dz_H(R)])$,
by the above claim, we know that
$(H-\frac\lambda2|\cdot|^2)\ast
\eta_\dz(\cdot+p^\dz)$ is convex in $(H^\dz)^{-1}([0,r])$  for some $\lambda>0$  whenever $\dz\in(0,\dz_H(R)]$ and $0<r<R$.
Thus, for any $p,q\in (H^\dz)^{-1}([0,r])$ and   any $\theta\in(0,1)$, by an argument similar to (i) we have
\begin{align*}
&H^\dz(\theta p+(1-\theta )q)-\frac\lambda2|\theta p+(1-\theta )q|^2
\le\theta [H^\dz(p)-\frac\lambda2|p|^2]+(1-\theta)  [H^\dz(q)-\frac\lambda2|q|^2].
\end{align*}
That is, $H^\dz(p)-\frac\lambda2|p|^2$ is convex in $(H^\dz)^{-1}([0,r])$, and hence
$\lambda_{H^\dz}(r)\ge \lambda_{H }(r+\dz/\dz_H(R))$.
Similar argument leads to that $\Lambda_{H^\dz}(r)\le \Lambda_{H }(r+\dz/\dz_H(R))$.

Finally, we prove the above claim.
Let  $s=H(q) >0$. If $H(q)\le 4R$, then  $ |q|^2\le  {2s/\lambda_H(s )}\le {8R/\lambda_H(4R)}$.
For $\dz\in(0,1]$, if $|p+p^\dz-q|<\dz$, we have 
\begin{align*}H^\dz(p)&=\wz H^\dz(p+p^\dz)-\wz H^\dz(p^\dz) \ge
 H\ast\eta_\dz(p+p^\dz)-\wz H^\dz(0)\\
 &\ge \min_{|p+p^\dz-\xi|\le \dz} H(\xi)- {\wz\Lambda_H(1)} \dz^2  \ge \min_{|\xi -q|\le2\dz} H(\xi)- {\wz\Lambda_H(1)} \dz^2 \\
    &\ge H(\xi)-\max_{|\xi-q|\le 2\dz} |H(q)-H(\xi)|-\wz\Lambda_H(1)\dz^2 \\
      &\ge s-2\|H\|_{C^{0,1}(B(0, 2[ {8R}/{\lambda_H(4R )}]^{1/2} ))}\dz- \wz\Lambda_H(1)\dz^2.
 \end{align*}
 Let
 $$\mbox{$\kz_H(R)=1+2\|H\|_{C^{0,1}(B(0, 2[ {8R}/{\Lambda_H(4R )}]^{1/2} ))}
 +  {\wz\lambda_H(1)}$ and  $\dz_H(R)=1/\kz_H(R)\in(0,1]$}. $$
 For any  $0<\dz<\dz_H(R)$,
 we have $$H^\dz(p)\ge s-  \kz_H(R)\dz\quad{\rm whenever}\ |p+p^\dz-q|<\dz.$$

If $H(p)=s>4R$ and $|p+p^\dz-q|<\dz$, letting $\theta\in(0,1] $ such that $H(\theta q)=4R$
we have $$H^\dz( p)>H^\dz(\theta p)>4R-\kz_H(R)\dz\ge 3R.$$
  Therefore, for any  $0<\dz\le \dz_H(R)$ and $0\le r\le R$,
  if $H(p)=s > r+\kz_H(R)\dz$, then  we have
   $H^\dz(p)\ge \min\{3R, s-\kz_H(R)\dz\} >r $  whenever $|p+p^\dz-q|<\dz$ as desired.
      This completes the proof of Lemma \ref{LEM7.3}.
\end{proof}

\begin{remapp}\label{REM7.6}\rm
Similarly to Lemma \ref{LEM7.3} (iii),  for any $R>0$, there exist $\wz\dz_H(R)\in (0,1]$  such that for all $\dz\in(0,\wz \dz_H(R)]$, we have
\begin{equation}\label{eq7.6} \mbox{$\wz \lambda_{H^\dz}(r)\ge \wz \lambda_H(r+   \dz/\bar\dz_H(R) )$ and
 $\wz \Lambda_{H^\dz}(r)\le \wz \Lambda_H(r+  \dz/\bar\dz_H(R) )$   $\forall\,  r\in[0,R]$ }
 \end{equation}
 where $\wz \lambda_{H^\dz}$ and $\wz \Lambda_{H^\dz}$ are defined as in \eqref{eq7.6y1} and \eqref{eq7.6y2}.
\end{remapp}
As a by-product of the proof of Lemma \ref{LEM7.3} (iii), we also have the following result.
\begin{corapp}\label{COR7.7} For any $R>0$, there exists $\kz_H(R)\ge1$ such that for all $r<R$ and $0<\dz<1/\kz_H(R)$,
we have
 \begin{equation}\label{eq7.7}
H  ^{-1}([0,r])\subset  (H^\dz)^{-1}([0,r+  \kz_H(R)\dz])\quad{\rm and }\quad
 (H^\dz)^{-1}([0,r])\subset H  ^{-1}([0,r+   \kz_H(R)\dz]).
 \end{equation}
\end{corapp}

\begin{proof}
 The second $\subset$ has been proved in the proof of Lemma \ref{LEM7.3} (iii) with $\kz_H(R)=1/\dz_H(R)$.
 To see the first $\subset$, it suffices to show that
 $H(p)<s$ implies $H(p)<s+\kz_H(R)\dz$ for $s<R$ for some $ \kz_H(R)\ge1$.
 This can be done similarly as in the proof of Lemma \ref{LEM7.3} (iii). Here we omit the details.
\end{proof}

The following remark says that
 the auxiliary function  $\tau_H $  given by
  \eqref{eq1.8} is reduced to \eqref{eq1.7} when  $H\in C^2(\rn)$  additionally.

\begin{remapp}\label{REM7.4}
\rm
If $H\in C^2(\rn)$ satisfies (H1)\&(H2), by $H^\dz\to H$ in $C^2(\rn)$ as $\dz\to0$,
 we know that $\wz \tau_{H^\dz}\in C^0(\rn)$ uniformly in $\dz\in(0,1]$, and hence
\begin{equation}\label{eq7.4}\sup_{\ez>0}\liminf_{\dz\to0}\inf_{ H(p)\le R+\ez}\wz \tau_{H^\dz }(p)= \inf_{\ez>0}\inf_{ H(p)\le R+\ez}\wz \tau_H (p)= \inf_{ H(p)\le R }\wz \tau_H (p),
\end{equation}
which implies 
 the coincidence of the definition \eqref{eq1.7} and
 \eqref{eq1.8}.
  \end{remapp}

By Lemma \ref{LEM7.2} we have  the following properties of $\tau_H$.
  \begin{lemapp}\label{LEM7.5} Suppose that $H\in C^0(\rn)$ satisfies (H1){\rm\&}(H2).
  We have the following.
  \begin{enumerate}
\item[(i)] If $H\in C^2(\rn)$, then    $\tau_H\in C^0([0,\fz))$  and
   $ \tau_H(0)=1/2$.  If  $H\in C^2(\rn)$ and $H^{\gz}$ is convex in $H^{-1}([0,R])$ for some $\gz\in[1/2,1)$ and $R>0$,
  then $ \tau_{  H}\ge 1-\gz$ in $ [0,R] $,  and $\tau_H = 1/2$ in $[0,R]$ when $\gz=1/2$.

   \item[(ii)]
 $\tau_H$ is a decreasing and right-continuous function in $[0,\fz)$, and
 $$ \frac12\left[\frac{\lambda_H   }{  \Lambda_H  }\right]^2\le \tau_H \le\frac12\quad {\rm in }\ [0,\fz).$$
\end{enumerate}
\end{lemapp}

\begin{proof}
 Note that
 (i)  follows  from Lemma \ref{LEM7.2} directly.
 Applying Lemma \ref{LEM7.2} to  $\wz \tau_{H^\dz}$, we  deduce (ii) from the definition  \eqref{eq1.8} of $\tau_H$.
\end{proof}

Finally, we have the following result.

\begin{lemapp}\label{LEM7.8}
Suppose that $H\in C^k(\rn)$ for some $k=0,1,2$
 satisfies (H1){\rm\&}(H2).
 For any $R >0$, there exists  a $\wz H\in C^k(\rn)$ satisfying  (H1'){\rm\&}(H2)  such that   $\wz H=H$ in $H^{-1}([0, R+1])$, and hence
 $\tau_{H}=\tau_{\wz H} $, $\lambda_{H}=\lambda_{\wz H}$ and
$\Lambda_{H}=\Lambda_{\wz H}$ in $[0, R+1)$.
\end{lemapp}

\begin{proof}
Without loss of generality, let $R\ge1$. First we assume that $H\in C^2(\rn)$.
 For  $k>0$, set
$$ H^{(k)} (p)=H(p)\eta (p)+k  \wz \phi\ast\wz\eta (p) \quad\mbox{ $\forall p\in\rr^2$},$$
 where $ \wz\phi=0$ on $B(0,2R )$ and $= |p|^2-R ^2 $ if $|p|>2R $; $\wz\eta \in C^\fz_c(B(0,R ))$  with   $\wz \eta =1$ on $B(0, \frac12R  )$,    $0\le \wz \eta\le 1$ and $\int_{\rr^2}\wz \eta=1$;
 and $\eta\in C^\fz( B(0,8R ))$ with
  $\eta =1$ on $B(0,4R )$    and $0\le \eta\le 1$.

Note that $H^{(k)}\in C^2(\rn)$, and $H=H^{(k)}$ in $B(0,R+1)$ for any $k>0$.
It suffices to show that  there exists $k_H>0$ large enough so that $\wz H=H^{(k_H)}$  satisfies  (H1')\&(H2).
To see this
let   $$k_H= 1+8\sup_{4R\le |p|\le 8R}|D^2_{pp}[H\eta](p)|.$$
By a direct calculation we know that
  $$\mbox{$\wz\Lambda_{\wz H }(r)\le
  \wz\Lambda_{ H } (8R )  +  k_{H }   <\fz$ and
  $
  \wz\lambda_{\wz H }(r)\ge
  \min\{  \wz\lambda_{H } (8R ) ,  k_{H }/2 \} >0$ for all $r>0$}$$
Since,
$\lambda_{\wz H } (\fz)=\lim_{r\to\fz}\wz \lambda_{\wz H }(r)$
and $\Lambda_{\wz H } (\fz)=\lim_{r\to\fz}\wz \Lambda_{\wz H }(r)$.
   We know that
 $\wz H $ satisfies  (H1')\&(H2).

 For   $H\in C^k(\rn)$ with $k=0,1$, let $\{H^\dz\}_{\dz\in(0,1]}$ be the smooth approximation of $H$ given in \eqref{eq7.1}.
 Let $ \wz{ H^\dz} $ be as above. It suffices to show that
\begin{equation}\label{eq7.8}
0<\liminf_{\dz\to0}\lambda_{\wz{ H^\dz} }(\fz)
<\limsup_{\dz\to0}\Lambda_{\wz{ H^\dz}}(\fz)<\fz.
\end{equation}
Indeed, this implies that $ \wz{ H^\dz} $ converges to some $\wz H\in C^k$ which satisfies  (H1')\&(H2).
Note that $\wz{ H^\dz}  =H^\dz $ in $B(0,R+1)$ implies that $\wz H=H$ in $B(0,R+1)$. Such $\wz H$ is as desired.

 To see \eqref{eq7.8},
 noting that $$|D^2_{pp}H^{\dz}(p)|\le 4   \wz \Lambda_{H^{\dz}}(p) \le  4\wz\Lambda_H (8R+1)  \quad\forall |p|<8R.$$
 We know that $k_{H^\dz}$ is bounded uniformly in $\dz\in(0,1]$.
 Moreover, for all $ r>0$ and $\dz\in(0,\dz_H(8R+1)]$ by \eqref{eq7.6} we have
 $$\wz\Lambda_{\wz{ H^\dz}}(r)\le
  \wz\Lambda_{H^\dz} (8R+1)  +  k_{H^\dz} <  \wz\Lambda_{H } (8R+2)  + \sup_{\dz\in(0,1]} k_{H^\dz} <\fz$$
 and   $$\wz\lambda_{\wz{ H^\dz}}\ge
  \min\{  \wz\lambda_{H^{\dz}} (8R+1) ,  k_{H^\dz}/2 \}>\min\{  \wz\lambda_H (8R+2) , 1/2\} >0.$$
 Noting that
 $$\Lambda_{\wz H^\dz}(\fz)= \liminf_{r\to\fz}\wz\Lambda_{\bar H^\dz}(r)\quad{\rm and }\quad \lambda_{\bar H^\dz}(\fz)= \liminf_{r\to\fz}\wz\lambda_{\bar H^\dz}(r),$$
 we obtain   \eqref{eq7.8} as desired.
 This completes the proof of Lemma \ref{LEM7.8}.
\end{proof}

\renewcommand{\thesection}{Appendix B}
\newtheorem{lembpp}{Lemma \hspace{-0.15cm}}
\newtheorem{thmbpp}[lembpp] {Theorem \hspace{-0.15cm}}
\newtheorem{corbpp}[lembpp] {Corollary \hspace{-0.15cm}}
\newtheorem{rembpp}[lembpp]  {Remark  \hspace{-0.15cm}}
\newtheorem{defnbpp}[lembpp]  {Definition  \hspace{-0.15cm}}
\renewcommand{\theequation}{B.\arabic{equation}}

\renewcommand{\thelembpp}{B.\arabic{lembpp}}

\section{Some known results of absolute minimizers}

%

Suppose that $H\in C^0(\rn)$ satisfies (H1){\rm\&}(H2).
We recall the existence and uniqueness in Lemma \ref{LEM7.9}  of absolute minimizers;
 identification with comparison properties with cones   in Lemma \ref{LEM7.13}.
Moreover we obtain their global absolute minimizing property in Lemma \ref{LEM7.14}.
 We also recall their identification with viscosity solutions in Lemma \ref{LEM7.10} when $H\in C^1(\rn)$ additionally,
 and their  $C^1$-regularity in Lemma \ref{LEM7.10x} when $n=2$.

 The following existence follows from \cite{bjw}, and uniqueness from   \cite{acjs}.

   \begin{lembpp}\label{LEM7.9} Suppose that $H\in C^0(\rn)$ satisfies (H1){\rm\&}(H2).
For any bounded domain $U\subset\rr^2$  and any $g\in C(\partial U)$, there exists a unique  $u\in C(\overline U)\cap AM_H(U)$ with $u|_{\partial U}=g$ on $\partial U$.
    \end{lembpp}

Next we recall the comparison property with cones.
 Associated to $H$  set
$$C^H_a(x)=\sup_{H(p)\le a}\{p\cdot x\}\quad \forall a\ge0,\ x\in\rr^2.$$
It is evident that $C^H_a\in C^{0,1}(\rr^2)$ is convex, positively homogeneous, subadditive
and $C^H_a(x) > 0$ for every $a > 0$ and $x\ne  0$.
See \cite[Lemma 2.18]{acjs} for more details and also the following lemma.
\begin{lembpp} \label{LEM7.11} Suppose that $H\in C^0(\rn)$ satisfies (H1){\rm\&}(H2).
 Let $U\subset\rn$ be any domain,  $u\in C^{0,1}(U)$  and $a\ge0$. The following are equivalent:
\begin{enumerate}
\item[(i)] $H(Du)\le a$ almost everywhere in $U$;
\item[(ii)] $u(x)-u(y) \le C^H_a(x-y)$ provided the line segment $[x,y]\subset U$.
\end{enumerate}
\end{lembpp}

\begin{defnbpp}\rm
Let $U\subset\rn$ be any domain.
\begin{enumerate}
\item[(i)] A function $u\in \usc(U)$ satisfies the comparison property with cones  for $H$ from above in $U$ if
$$\max_V\{u- C^H_a(x-x_0)\}=\max_{\partial V}\{u- C^H_a(x-x_0)\}$$
whenever $V\Subset\Omega$, $a\ge0$ and $x_0\in\rn\setminus V$; for short, write $u\in CCA_H(U) $.
\item[(ii)] A function $u\in \usc(U)$ satisfies the comparison property with cones  for $H$ from below in $U$ if
and $$\min_V\{u+ C^{ H}_a(x_0-x)\}=\min_{\partial V}\{ u+ C^{  H}_a(x_0-x)\}$$
whenever $V\Subset\Omega$, $a\ge0$ and $x_0\in\rn\setminus V$; for short, write $u\in CCB_H(U) $.
\item[(iii)] We say $u \in C^0(U)$ satisfies the comparison property with cones  for $H$   in $U$ (for short,  $u\in CC_H(U)$ )
 if  $u\in CCB_A(U)\cap CCB_H(U) $.
 \end{enumerate}
\end{defnbpp}

%

 The following characterization of absolute minimizers follows from \cite[Theorem 4.8]{acjs}.
  \begin{lembpp}\label{LEM7.13}
Suppose that $H\in C^0(\rn)$ satisfies (H1){\rm\&}(H2).
 For any domain $U\subset\rn$, $u\in  AM_H(U)$ if and only if $u\in  CC _H(U)$.

 %
%
  \end{lembpp}


We give the following result which will be used in Section 5.
 \begin{lembpp}\label{LEM7.14} Suppose that $H\in C^0(\rn)$ satisfies (H1){\rm\&}(H2).
For any bounded domain $U\in \rn$,
if  $u\in C(\overline U)\cap AM_H(U)\cap  C^{0,1}(\partial U)$,
then $u\in C^{0,1}(\overline U)$ with
   $$\|H(Du)\|_{L^\fz(U)}\le \sup_{|p|\le \|u\|_{C^{0,1}(\partial U)}}H(p) .  $$
  \end{lembpp}

 \begin{proof}
 Let $v \in C^{0,1}(\rn)$ be the Mcshane extension of $u|_{\partial U}$ in $\rn$
 so that $\|v\|_{C^{0,1}(\rn)}\le L:=\|u\|_{C^{0,1}(\partial U)}$.
 Let $a\ge \sup_{|p|\le L}H(p)  $. We have $a\ge \|H(Dv)\|_{L^\fz(\wz U)}$, where $\wz U$ is a convex set containing $U$.
 By Lemma \ref{LEM7.1}0, we have
\begin{equation}
u(x)-u(y)= v(x)-v(y)\le C^H_a(x-y)\quad \forall x,y\in \partial U.
\end{equation}
Let $V_\ez=\{x\in U, \dist(x,\partial U)>\ez\}$.
Given any $\dz>0$, by the continuity of $C^H_a$ and $u$,
for sufficiently small  $\ez>0$  for all
$z\in \partial U$ we have  $$\mbox{ $ u(x)\le u(z)+C^H_a(x-z)+ \dz$ for all $x\in \partial V_\ez$ and hence, by
$u\in CC_H(\Omega)$, for all $x\in  V_\ez$ .}$$
Letting $\ez\to0$ and $\dz\to 0$ in order,  we arrive at
 $$\mbox{ $ u(x)\le u(z)+C^H_a(x-z)$ for all $x\in  U$ and $z\in \partial U$.}$$

Given any $x\in U$, let $W_{x,\ez}:=\{y\in U\setminus \{x\}, \dist(y,\partial U)>\ez\}$.
For any $\delta>0$, if $\ez>0$ is sufficiently small, we have
 $$\mbox{ $ u(z)\ge u(x) -C^H_a(x-z)-\delta $ for all   $z\in \partial W_{x,\ez}$
 and hence, by $u\in CC_H(\Omega)$, for all  $z\in   W_{x,\ez}$.}$$
  Letting $\ez\to0$ and $\dz\to 0$ in order,  we arrive at
 $$\mbox{ $ u(z)\ge u(x)-C^H_a(x-z)$, equivelently, $ u(x)-u(z)\le  C^H_a(x-z)$, for all $x\in  U$ and $z\in   U$.}$$
So by Lemma \ref{LEM7.1}1 we get $u\in  C^{0,1}(\overline U)$ and $H(Du)\le a$ almost everywhere in $U$.
This completes the proof of Lemma \ref{LEM7.14}.
\end{proof}

As a consequence of Lemma \ref{COR7.7} we have the following result, which will be used also in Section 5.
\begin{lembpp}\label{cone1} Suppose that $H\in C^0(\rr^2)$ satisfies (H1'){\rm\&}(H2).
 For $\delta>0$, let $H^\dz\in C^\fz(\rn)$ as in \eqref{eq7.1}.
For  $a>0$ and $\ez>0$, there exists $\dz(\ez,a )>0$ such that
$$ C^{H}_{(1+\ez)^{-1}\wz a}(x)\le C^{H^{\dz}}_{\wz a}(x)\le  C^{H}_{(1+\ez)\wz a}(x)\quad \forall x\in \rn, \wz a\in[\frac12a,2a].$$
\end{lembpp}

The following  identification follows from   \cite{y06,cwy}.
    \begin{lembpp}\label{LEM7.10}
 Suppose that $H\in C^1(\rn)$ satisfies (H1){\rm\&}(H2).
  For any domain  $\Omega\subset\rr^2$, $u\in AM_H(\Omega)$  if and only if $u\in C^0(\Omega)$ is
a viscosity solution to \eqref{eq1.2}.
 \end{lembpp}

The following $C^1$-regularity  was proved in \cite{fwz} (see \cite{wy} for the case $H\in C^2(\rn)$).
The proof is based on the linear approximation property obtained in \cite{fwz} and also some necessary modifications of approach in \cite{wy,s05}.

 \begin{thmbpp}\label{LEM7.10x} Suppose that $H\in C^0(\rr^2)$ satisfies (H1){\rm\&}(H2).
  For any domain   $\Omega\subset\rr^2$, if $u\in AM_H(\Omega)$, then $u\in C^1(\Omega)$.
 \end{thmbpp}

\bigskip
 \noindent {\bf Acknowledgment.} The authors
 would like to thank Professor Yi Zhang and Professor Yifeng Yu for several valuable discussions in this paper.
The authors were supported  by
 National Natural Science of Foundation of China (No. 11522102\&11871088).

%
%
%
%
%
%
%
%
%
%
\end{document}